\documentclass[11pt,centertags,oneside]{amsart}

\usepackage{amsmath,amstext,amsthm,amscd}
\usepackage{amssymb}
\usepackage{a4wide}
\usepackage[mathscr]{eucal}
\usepackage{mathrsfs}
\usepackage{typearea}
\usepackage{charter}

\usepackage{comment}
\usepackage{cancel}

\usepackage{hyperref}
\usepackage{cleveref}
\hypersetup{
    colorlinks=true,
    linkcolor=blue,
    citecolor=blue,
    urlcolor=blue
}

\usepackage[normalem]{ulem} %% To remove underline from the title of journals in  references

\usepackage[a4paper,width=16.5cm,top=3cm,bottom=3cm]{geometry}

\numberwithin{equation}{section}

\allowdisplaybreaks
\emergencystretch=\maxdimen
\usepackage{multicol}
\usepackage{xcolor}

\usepackage{graphicx} % for rotatebox
\usepackage{tikz} 
\usetikzlibrary{matrix}
\usepackage{todonotes} % comment box

\theoremstyle{plain}
\newtheorem{thm}{Theorem}[section]
\newtheorem{coro}[thm]{Corollary}
\newtheorem{prop}[thm]{Proposition}
\newtheorem{lem}[thm]{Lemma}
\newtheorem*{thm*}{Theorem}
\newtheorem*{lem*}{Lemma}
\newtheorem*{prop*}{Proposition}
\newtheorem*{coro*}{Corollary}

\theoremstyle{definition}
\newtheorem{defn}[thm]{Definition}

\newtheorem*{defn*}{Definition}
\newtheorem*{ex*}{Example}
\newtheorem*{rmk*}{Remark}
\numberwithin{equation}{section}

\DeclareMathOperator*{\esssup}{ess\,sup}
\DeclareMathOperator*{\essinf}{ess\,inf}

\newcommand{\CF}{T}
\newcommand{\TF}{\mathcal{L}}

\newcommand{\M}{\mathcal M}

\newcommand{\F}{\widehat T}
\newcommand{\T}{\overline{T}}

\newcommand{\bH}{\mathbb{H}^2}

\newcommand{\bR}{\mathbb R}

\newcommand{\veps}{\varepsilon}

\title{Extreme value theorem for geodesic flow on the quotient of the theta group}

\author{Jaelin Kim$^*$}
\thanks{$^*$HUN-REN Alfr\'ed R\'enyi Institute of Mathematics, Re\'altanoda utca 13-15, Budapest 1053, Hungary. E-mail: kimjl@snu.ac.kr}
\author{Seul Bee Lee$^{\dagger}$}
\thanks{$^{\dagger}$Department of Mathematical Sciences, Seoul National University, 1 Gwanak-ro, Gwanak-gu, Seoul 08826, Republic of Korea. E-mail: seulbee.lee@snu.ac.kr}
\author{Seonhee Lim$^{\ddagger}$}
\thanks{$^{\ddagger}$Department of Mathematical Sciences and Research Institute of Mathematics, Seoul National University. E-mail: slim@snu.ac.kr, seonhee.lim@gmail.com}
\usepackage{pgfplots}
\pgfplotsset{compat=1.18}
\begin{document}

\begin{abstract} 
We establish an extreme value theorem for the geodesic flow $g_t$ on the hyperbolic surface $\Theta\backslash\mathbb{H}^2$ associated with the theta group $\Theta$. To capture excursions into both cusps of this surface, we introduce a generalized continued fraction algorithm obtained by splicing the even and odd--odd continued fraction maps into a single dynamical system. We prove that the natural extension of this map is isomorphic to the first return map of the geodesic flow on a suitable cross-section. Using spectral properties of the associated transfer operator, we derive a Galambos-type extreme value law for the digits of the spliced continued fraction. This symbolic result is then translated into a geometric extreme value theorem describing maximal cusp excursions of geodesics on $\Theta\backslash\mathbb{H}^2$.
\end{abstract} 

\maketitle

%\title{}
%
%\author{}
%\address{}
%\email{}
%
%\author{Seul Bee Lee}
%\address{}
%\email{}
%
%\author{}
%\address{}
%\email{}
%
%\subjclass[2010]{Primary 11J70; Secondary 37E05.}
%%11J70: Continued fractions and generalizations
%%37E05: Maps of the interval
%\keywords{Continued fractions; Diophantine approximation; Romik system}
%
%%\thanks{Research partially supported by the National Research Foundation of Korea (NRF-2018R1A2B6001624).}

\section{Introduction}

The interplay between continued fractions, symbolic dynamics, and hyperbolic geometry has proven fruitful for understanding quantitative properties of geodesic flows on hyperbolic surfaces.
A classic example is the modular surface
$\mathrm{SL}(2,\mathbb Z)\backslash\bH$,
where the regular continued fraction expansion of a real number encodes the forward itinerary of corresponding geodesics. Adler and Flatto were the first to develop a full coding of the geodesic flow on the modular surface \cite{AF84}, while Series gave a novel geometric description \cite{Ser85}; see also \cite{Art92}.
Statistical statements about continued fraction digits translate directly into geometric statements about cusp excursions \cite{Pol09, Moe82}.

For a general finite-index subgroup of $\mathrm{PSL}(2,\mathbb Z)$, Chang and Mayer \cite{CM00} constructed a finite-state transfer-operator coding by adjoining a coset coordinate to the modular continued-fraction dynamics. Since $\Theta$ has finite index, their construction gives an alternative coding in principle. Our spliced map instead uses a scalar interval map whose symbols distinguish the two cusps and whose natural extension has an explicit geometric cross section; this makes it possible to read simultaneous height excursions directly from the digits. Adapting the Chang--Mayer coding to the present height observable would therefore require an additional translation from its coset coordinate to cusp-specific geometric height.

In this article, we study the \emph{theta group}
$$\Theta = \left\{\begin{pmatrix}a&b\\c&d\end{pmatrix}\in\mathrm{SL}(2,\mathbb{Z}):\begin{pmatrix}a&b\\c&d\end{pmatrix}\equiv\begin{pmatrix}1&0\\0&1\end{pmatrix}\text{ or }\begin{pmatrix}0&-1\\1&0\end{pmatrix}\pmod{2}\right\},$$
which is a congruence subgroup of level $2$ and of index $3$ of $\mathrm{SL}(2,\mathbb Z)$.
The group $\Theta$ is generated by
$$\tau=\begin{pmatrix}1&2\\0&1\end{pmatrix}\quad\text{ and }\quad\sigma = \begin{pmatrix}0&-1\\1&0\end{pmatrix}.$$
Let $\mathcal{F}$ be the ideal triangle with vertices $-1$, $1$, and $\infty$. The triangle $\mathcal{F}$ is a fundamental domain for the action of $\Theta$ on $\bH$. See Figure~\ref{fig:fund}.
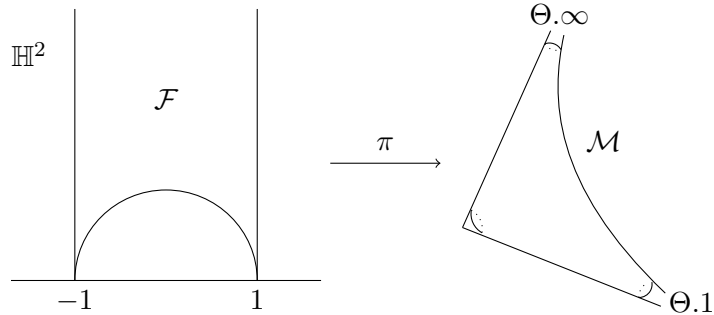
\begin{figure}
\begin{tikzpicture}[scale=1.2]
\node at (-1.5,2.5) {$\bH$}; 

\node[below] at (1, 0) {$1$};
\node[below] at (-1, 0) {$-1$};

\draw (-1.7, 0) -- (1.7, 0);

%\filldraw[gray!30!white] (-1,0)--(-1,3)--(1,3)--(1,0) arc (0:180:1);
\draw (-1,0) -- (-1,3);
\draw (1,0) -- (1,3);
\draw (-1,0) arc (180:0:1);

\node at (0,2) {$\mathcal F$};
\end{tikzpicture}
\begin{tikzpicture}[scale=0.58]
\draw[->] (-1, 1.5) -- node[above = 1pt] {$\pi$} (1.5, 1.5);

\node at (4.7+0.5, 2) {$\mathcal{M}$};

	\draw (3.5+.5,2.5+2.03) -- (1.5+.5,-2+2.03);
	\draw[dotted] (6,-3.8+2.03+0.2) arc (180:120:0.4); 
    \draw         (6,-3.8+2.03+0.2) arc (270:370:.3);
	\draw[dotted] (1.7+.5,-1.6+2.03) arc (60:-10:.5); 
    \draw         (1.7+.5,-1.6+2.03) arc (170:240:.5); 
	\draw[dotted] (3.37+.5,2.2+2.03) arc (210:280:.3); 
    \draw         (3.37+.5,2.2+2.03) arc (110:30:.3); 
	
	\draw (3.8+.5,2.4+2.03) .. controls (3.5+.5, 1+2.03) and (3.5+.5, -1+2.03) .. (6.1+.5,-3.5+2.03);
	
	\node at (4.2,4.9) {$\Theta.\infty$};
    \node at (6.5+0.7,-3.7+2.03) {$\Theta.1$};

	\draw (6+.5,-3.8+2.03) -- (1.5+.5,-2+2.03);

\end{tikzpicture}

\caption{A fundamental domain $\mathcal F$ of $\Theta$ and the quotient surface $\M=\Theta\backslash \bH$.}
\label{fig:fund}
\end{figure}
Under the projection map 
$$\pi:\bH\to \Theta\backslash\bH,$$ the vertical lines $\mathrm{Re}(z) = \pm1$ are identified with each other, and so are the left and right quarter arcs of the unit circle.
The quotient hyperbolic surface 
$$\mathcal M \;=\; \Theta\backslash\bH$$ is topologically equivalent to a sphere
with two inequivalent cusps (at the orbits of $\infty$ and~$1$)
and one elliptic point of order $2$.
Because $\mathcal M$ has \emph{two} cusps,
no single continued‐fraction–type algorithm that has been considered to date
captures \emph{all} cusp excursions of a generic geodesic on~$\mathcal M$.

For the cusp $\Theta . \infty,$ the \emph{even continued fraction} (ECF) map $T_e$ on $[0,1]$ provides a complete symbolic coding of the associated
“even” cusp excursions (see \cite{KL96} and \cite{BM18}).
The ECF was first introduced by Schweiger \cite{Sch82, Sch84} from a Diophantine approximation perspective.
Dually, for the cusp $\Theta.1$, Kim–Lee–Liao \cite{KLL22} introduced the \emph{odd–odd continued fraction} (OOCF) map $T_{o}$, 
which is conjugate to the ECF via the involution
$$\iota(x)=\frac{1-x}{1+x}.$$
Each of the ECF and OOCF encodes the geodesic excursions to only one cusp and therefore does not yield a full description of the cusp excursions of the geodesic flow into both cusps on the surface $T^1\mathcal M$.

\subsection{Spliced continued fraction.}
Our main idea for obtaining extreme value theorems is to
\emph{splice} the ECF and OOCF into a single interval map
\[
T:(0,1)\longrightarrow[0,1],
\qquad
T(x)=
\begin{cases}
T_e(x), & x\in\big(0,\tfrac12\bigr],\\[2pt]
T_{o}(x), & x\in\bigl(\tfrac12,1\bigr).
\end{cases}
\]
See Figure~\ref{fi:graphT} and Eq.~\eqref{eqn:2.1} for the precise definition.
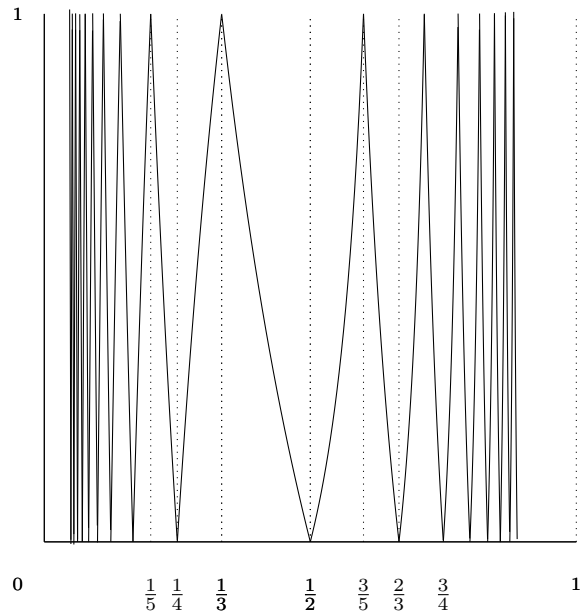
\begin{figure} 
\begin{tikzpicture}[scale=7]
\draw (1.4,0) -- (1.4+1,0);
\draw (1.4+0,0) -- (1.4+0,1);
%\draw[domain=1.4+1/2:1.4+1,black] plot (\x,{-(1/(\x-1.4)-2)});
\draw[domain=1.4+1/3:1.4+1/2,black] plot (\x,{(1/(\x-1.4)-2)});
\draw[domain=1.4+1/4:1.4+1/3,black] plot (\x,{-(1/(\x-1.4)-4)});
\draw[domain=1.4+1/5:1.4+1/4,black] plot (\x,{(1/(\x-1.4)-4)});
\draw[domain=1.4+1/6:1.4+1/5,black] plot (\x,{-(1/(\x-1.4)-6)});
\draw[domain=1.4+1/7:1.4+1/6,black] plot (\x,{(1/(\x-1.4)-6)});
\draw[domain=1.4+1/8:1.4+1/7,black] plot (\x,{-(1/(\x-1.4)-8)});
\draw[domain=1.4+1/9:1.4+1/8,black] plot (\x,{(1/(\x-1.4)-8)});
\draw[domain=1.4+1/10:1.4+1/9,black] plot (\x,{-(1/(\x-1.4)-10)});
\draw[domain=1.4+1/11:1.4+1/10,black] plot (\x,{(1/(\x-1.4)-10)});
\draw[domain=1.4+1/12:1.4+1/11,black] plot (\x,{-(1/(\x-1.4)-12)});
\draw[domain=1.4+1/13:1.4+1/12,black] plot (\x,{(1/(\x-1.4)-12)});
\draw[domain=1.4+1/14:1.4+1/13,black] plot (\x,{-(1/(\x-1.4)-14)});
\draw[domain=1.4+1/15:1.4+1/14,black] plot (\x,{(1/(\x-1.4)-14)});

\draw[domain=1.4+1/16:1.4+1/15,black] plot (\x,{-(1/(\x-1.4)-16)});
\draw[domain=1.4+1/17:1.4+1/16,black] plot (\x,{(1/(\x-1.4)-16)});
\draw[domain=1.4+1/18:1.4+1/17,black] plot (\x,{-(1/(\x-1.4)-18)});
\draw[domain=1.4+1/19:1.4+1/18,black] plot (\x,{(1/(\x-1.4)-18)});
\draw[domain=1.4+1/20:1.4+1/19,black] plot (\x,{-(1/(\x-1.4)-20)});
\draw[domain=1.4+1/21:1.4+1/20,black] plot (\x,{(1/(\x-1.4)-20)});

\draw[dotted, thin] (1.4+1/2,0) -- (1.4+1/2,1);

\node at (1.4-0.05,1){\tiny$1$};
\node at (1.4+1,-0.08){\tiny$1$};
\node at (1.4-0.05,-0.08){\tiny$0$};
\node at (1.4+1/2,-0.1) {\small$\frac12$};

\draw[dotted, thin] (1.4+1/3,0) -- (1.4+1/3,1);

\node at (1.4+1/3,-0.1) {\small$\frac13$};
\node at (1.4+1/4,-0.1) {\small$\frac14$};
\node at (1.4+1/5,-0.1) {\small$\frac15$};

\draw[dotted, thin](1.4+1/4,0)--(1.4+1/4,1);
\draw[dotted, thin](1.4+1/5,0)--(1.4+1/5,1);
\draw (1.4,0) -- (1.4+1,0);
\draw (1.4+0,0) -- (1.4+0,1);
%\draw[domain=1.4+0:1.4+1/3, black] plot (\x, {(\x-1.4)/(1-2*(\x-1.4))});
%\draw[domain=1.4+1/3:1.4+1/2, black] plot (\x,{1/(\x-1.4)-2});
%\draw[domain=1.4+1/2:1.4+1, black] plot (\x, {2-1/(\x-1.4)});
\draw[domain=1.4+1/2:1.4+3/5, black] plot (\x, {(2*(\x-1.4)-1)/(2-3*(\x-1.4))});
\draw[domain=1.4+3/5:1.4+2/3, black] plot (\x, {(2-3*(\x-1.4))/(2*(\x-1.4)-1)});
\draw[domain=1.4+2/3:1.4+5/7, black] plot (\x, {(3*(\x-1.4)-2)/(3-4*(\x-1.4))});
\draw[domain=1.4+5/7:1.4+3/4, black] plot (\x, {(3-4*(\x-1.4))/(3*(\x-1.4)-2)});
\draw[domain=1.4+3/4:1.4+7/9, black] plot (\x, {(4*(\x-1.4)-3)/(4-5*(\x-1.4))});
\draw[domain=1.4+7/9:1.4+4/5, black] plot (\x, {(4-5*(\x-1.4))/(4*(\x-1.4)-3)});
\draw[domain=1.4+4/5:1.4+9/11, black] plot (\x, {(5*(\x-1.4)-4)/(5-6*(\x-1.4))});
\draw[domain=1.4+9/11:1.4+5/6, black] plot (\x, {(5-6*(\x-1.4))/(5*(\x-1.4)-4)});
\draw[domain=1.4+5/6:1.4+11/13, black] plot (\x, {(6*(\x-1.4)-5)/(6-7*(\x-1.4))});
\draw[domain=1.4+11/13:1.4+6/7, black] plot (\x, {(6-7*(\x-1.4))/(6*(\x-1.4)-5)});
\draw[domain=1.4+6/7:1.4+13/15, black] plot (\x, {(7*(\x-1.4)-6)/(7-8*(\x-1.4))});
\draw[domain=1.4+13/15:1.4+7/8, black] plot (\x, {(7-8*(\x-1.4))/(7*(\x-1.4)-6)});
\draw[domain=1.4+7/8:1.4+15/17, black] plot (\x, {(8*(\x-1.4)-7)/(8-9*(\x-1.4))});
\draw[domain=1.4+15/17:1.4+8/9, black] plot (\x, {(8-9*(\x-1.4))/(8*(\x-1.4)-7)});

%\draw[domain=1.4:1.4+1/3, blue] plot (\x, {(\x-1.4)/(1-2*(\x-1.4))});
%\draw[domain=1.4+1/3:1.4+1/2, blue] plot (\x,{1/(\x-1.4)-2});
%\draw[domain=1.4+1/2:1.4+1, blue] plot (\x, {2-1/(\x-1.4)});
\draw[dotted, thin] (1.4+1,0)--(1.4+1,1);
\draw[dotted, thin] (1.4+1/2,0) -- (1.4+1/2,1);

\node at (1.4-0.05,1){\tiny$1$};
\node at (1.4+1,-0.08){\tiny$1$};
\node at (1.4-0.05,-0.08){\tiny$0$};
\node at (1.4+1/2,-0.1) {\small$\frac12$};
%\node at (1.4-0.05,1/2){\tiny$\frac 12$};
%\node at (1.4-0.05,1/3){\tiny$\frac 13$};

\draw[dotted, thin] (1.4+1/3,0) -- (1.4+1/3,1);
\draw[dotted, thin] (1.4+3/5,0) -- (1.4+3/5,1);
\draw[dotted, thin] (1.4+2/3,0) -- (1.4+2/3,1);
\node at (1.4+1/3,-0.1) {\small$\frac13$};
\node at (1.4+3/5,-0.1) {\small$\frac35$};
\node at (1.4+2/3,-0.1) {\small$\frac23$};
\node at (1.4+3/4,-0.1) {\small$\frac34$};

%\draw[step = 0.1, very thin, color=gray!40] (0,0) grid (2.5,1); \node at (0,0){$\bullet$};
\end{tikzpicture}
\caption{The graph of $\CF$.}
\label{fi:graphT}
\end{figure}

Iterating $T$ generates a new continued fraction, which we call the \emph{spliced continued fraction} (SCF) and denote by
\[
x
\;=\;
\bigl[0; 
(a_1,\varepsilon_1)_{s_1},
(a_2,\varepsilon_2)_{s_2},
\;\dots, (a_n,\veps_n)_{s_n}, \;\dots
\bigr],
\]
where 
$(a_n,\veps_n)_{s_n}\in\mathscr A:=\{(k, \veps)_{s}: k \ge 2, \veps =\pm 1, s \in \{ e, o \} \}\cup\{(1,1)_e\}$
(see Definition~\ref{def:SCF}). 
We prove that the double cover of the natural extension of~$T$
(after a simple coordinate change) is 
isomorphic to the first return map of the geodesic flow
on a carefully chosen cross section {$\pi(X)$} of $T^1\M$ (see Theorem~\ref{th:exc} and Theorem~\ref{th:com2}). %\sh{add the relevant theorem number.}
Hence, the SCF furnishes a \emph{single} symbolic coding
of \emph{all} geodesics on $\mathcal M$,
simultaneously keeping track of excursions to both cusps.

\medskip

Our main results are \emph{extreme value theorems} (EVT) for the SCF and the geodesic flow on the surface $\mathcal M$, generalizing the classical result of Galambos
for the regular continued fraction \cite{Gal72} and Pollicott on the geodesic flow of the modular surface \cite{Pol09}, respectively.

\begin{thm}\label{thm:1} For any  $y>0,$
\[
\lim_{N\to\infty}
\mu\! \; \left\{x\in(0,1): \max_{1\leq n \leq N} a_n(x) \le
\frac{2Ny}{\log(2+\sqrt3)}\, \right\}
\;=\;
\exp\left(-\frac1y\right),
\]
where $\mu$ is the $T$-invariant probability measure.
\end{thm}

The proof combines
\begin{enumerate}
    \item the spectral gap of the Ruelle–Perron–Frobenius operator for~$T$,
 \item an exponential‐mixing estimate for the induced Gibbs measure,
and
 \item a delicate geometric correspondence
between partial quotients and geodesic excursions,
adapting arguments of Pollicott (see \cite[Theorem 2]{Pol09}).
\end{enumerate}

\subsection{Extreme value theorem for cusp excursions.}
Using the EVT (Theorem~\ref{thm:1}) for the SCF, 
we obtain the EVT for the geodesic flow $g_t$ on $\mathcal M$. Recall that $g_t:TM \to TM$ sends a unit tangent vector $v$ to the unit tangent vector $g_tv$ at time $t$ on the oriented geodesic determined by $v$. Denote by $\mathcal T$ the tessellation of $\bH$ by $\Theta$-translates of the quadrilateral $\mathcal Q$ which has cusps $0,\pm 1,$ and $\infty.$ See Figure~\ref{fi:tes}.

\begin{thm}\label{thm:1.2}
For $y>0$, \[
\lim_{T\to\infty}
m\Bigl\{v\in T^1\mathcal M:\;
\max_{0\le t\le T}d_\M(\pi(i),\gamma^v(t))-\log T \le \log\left(\frac{4y}{\pi^2}\right)\Bigr\}
\;=\;
e^{-1/y},
\]
where $\gamma^v(t)$ denotes the base point of the tangent vector $g_t v$, and $m$ is the Liouville probability measure on $T^1 \mathcal M$.
\end{thm}

The proof of Theorem~\ref{thm:1.2} uses the following key properties of the SCF:\\
(1) For $v\in T^1\M$, we cut $\gamma^v$ at each return point to the cross section to obtain a sequence of geodesic segments.
For a lift $\tilde\gamma^v$, the sequence of SCF partial quotients $a_n$ of $x = 1/|\tilde\gamma^v_\infty|$ coincides with the forward tail of the bi-infinite sequence recording the number of quadrilaterals crossed by $\tilde\gamma^v$ in the tessellation $\mathcal T$ (see Figure~\ref{fi:tes}, Definition~\ref{de:exc} and Lemma~\ref{superlevelsetforexcursion}).\\
(2) The supremum of distance
$$\sup_{0\le t\le T} d_\M(\pi(i),\gamma^v(t))$$
translates into approximately the log of the maximal digit $A_N:=\max_{1\le n\le N}a_n$ most of the time; when it does not, it is bounded above by $\log A_{N+1}$ (see Section~\ref{sec:5}).

We would like to emphasize that the novelty here lies in geometric cusp excursions measured via height, not symbolic winding or combinatorial depth.
For a finitely generated essentially free Fuchsian group $\Gamma$, 
Galambos' theorem for the number of parabolic generators, which is called the maximal cuspidal winding number of the geodesic flow of $\Gamma\backslash T^1\bH$, was established in \cite{JKS13}.
Limiting distributions for excursions into a single fixed cusp were previously established for the geodesic flow on a finite volume, non-compact $(d+1)$-dimensional manifold of curvature $-1$ \cite{DFL22}.

Theorem~\ref{thm:1} is the first extreme value theorem for a non-essentially free group with several cusps.
Moreover, to the best of our knowledge, Theorem~\ref{thm:1.2} provides the first extreme value distribution result for geometrically measured height excursions, rather than symbolic cuspidal windings, on a hyperbolic surface with more than one cusp, taking simultaneous excursions into multiple cusps.

In another context, for the Hurwitz complex continued fractions (HCF)  with $\mathbb Z[i]$, the extreme value theorem (for digits of the HCF) has been established \cite{Kir21}.
In \cite{BP26}, the extreme value theorem was extended to $\mathbb Z[\sqrt{-d}]$ for $d=2,3,7,11$; based on this, the theorem for the geodesic flow on the Bianchi orbifolds was also proved.
On the other hand, an extreme value theorem has been obtained for the unipotent actions of $\mathrm{SL}(n,\mathbb Z)\backslash \mathrm{SL}(n,\mathbb R)$ \cite{KM26} and \cite{MSY25}.

\medskip
The paper is organized as follows. 
In Section~\ref{sec:2},
we define the spliced continued fraction and construct the natural extension of~$T$.
In Section~\ref{sec:3}, we relate the first return map of the geodesic flow on $\M$ to a carefully chosen cross section to the natural extension of the SCF.
Section~\ref{sec:4} contains the proof of the spectral gap of the transfer operators and the proof of the Galambos‐type extreme value theorem. 
Finally, Section~\ref{sec:5}
translates the symbolic result into the geometric extreme value law
for geodesic heights.
Appendix~\ref{app:scf-levy} relates the average excursion time to the exponential growth rate of the denominator of the principal convergents of the SCF.

\medskip
\noindent\textbf{Acknowledgements.} 
The authors are grateful to the anonymous referee for a careful reading of the manuscript and for valuable comments.
The authors thank Dong Han Kim and Lingmin Liao for helpful conversations.

The first author is partially supported by the KKP 139502 project.
The second author is partially supported by the Institute for Basic Science (IBS-R003-D1), BK21 SNU Mathematical Sciences Division, and the NRF of Korea under Project no. RS-2026-25478379.
The second and third authors are supported by the NRF of Korea under Project no. RS-2025-00515082. The third author is supported by NRF of Korea under Project no. RS-2023-00301976, and RS-2025-02293115.

\section{The spliced continued fraction $\CF$}\label{sec:2}
\subsection{The spliced Gauss map, invariant measure and the transfer operator}
The spliced continued fraction map $\CF:[0,1]\to[0,1]$, defined in the introduction,  is piecewise Möbius with countably many branches, each corresponding to a symbol 
$(k,\epsilon)_s$. 
The intervals 
$I_{(k,\epsilon)_s}$ 
partition 
$(0,1)$ and determine the SCF digit. The SCF map $T$ is given by
\begin{equation}\label{eqn:2.1}
\CF(x) = \begin{cases}
\frac{1}{x}-2 & \text{if }x\in I_{(1,+1)_e} : = [\frac{1}{3},\frac{1}{2}],\vspace{1ex}\\
\frac{1}{x}-2k & \text{if }x\in I_{(k, +1)_e} :=[\frac{1}{2k+1},\frac{1}{2k})~\text{for }k\ge 2,\vspace{1ex}\\
2k-\frac{1}{x} & \text{if }x\in  I_{(k, -1)_e} := [\frac{1}{2k},\frac{1}{2k-1})~\text{for }k\ge 2,\vspace{1ex}\\
\frac{kx-(k-1)}{k-(k+1)x} & \text{if }x\in I_{(k, -1)_o} :=(\frac{k-1}{k},\frac{2k-1}{2k+1}]~\text{for }k\ge 2, \vspace{1ex} \\
\frac{k-(k+1)x}{kx-(k-1)} & \text{if }x\in I_{(k, +1)_o} :=(\frac{2k-1}{2k+1},\frac{k}{k+1}]~\text{for }k\ge 2,
\end{cases}
\end{equation}
with $T(0)=0$ and $T(1)=1$.

\noindent \textbf{Remark.}
Let $R_1(x)=2- \frac{1}{x}$, $R_2(x)=\frac{1}{x}-2$, $R_3(x)=\frac{x}{1-2x}$.
The Romik map $T_r$ is defined by $T_r(x) = R_1(x)$ for $x \in [1/2,1]$, $T_r(x) = R_2(x)$ for $x \in [1/3,1/2]$, and $T_r(x) = R_3(x)$ for $x \in [0,1/3]$ (see \cite{Rom08}).
In the first three cases, $T(x)=R_1R_3^{k-1}(x)$ or $R_2R_3^{k-1}(x)$, which equals the difference between $1/x$ and the even integer nearest $1/x$.
In the latter two cases, $T(x)=R_3R_1^{k-1}(x)$ or $R_2R_1^{k-1}(x)$.
In other words, $T$ is obtained by combining the two accelerations (the so-called jump transformations) of $T_r$ associated with intervals $[0,1/2]$ and $[1/3,1]$.
See Figure~\ref{fi:graphT} for the graph of $\CF$. 

Recall from the introduction that we denote the set of the SCF digits by 
$$\mathscr A:=\big\{(k, \veps)_{s}: k \ge 2, \veps =\pm 1, s \in \{ e, o \} \big\}\cup\big\{(1,1)_e\big\}.$$

   \begin{defn}[spliced continued fraction expansion]\label{def:SCF}

Define $A: (0,1) \to \mathscr{A}$ by
\[
    A(x) := 
        (k, \veps)_{s} \;\; \text{if }x\in I_{(k,\veps)_s}.\\
   \]  
  Letting $(a_{n}, \veps_{n})_{s_{n}}
:= A(\CF^{n-1}x) \in \mathscr{A},$ for $x\in (0,1)$, we define a \emph{spliced continued fraction expansion} as follows:
$$x=[0;(a_{1},\veps_{1})_{s_{1}}, (a_{2},\veps_{2})_{s_{2}}, \dots, (a_{n},\veps_{n})_{s_{n}}, \dots ].$$ 
Every irrational number $x$ admits a unique well-defined spliced continued fraction expansion. 
For $x\in \mathbb Q$, the expansion has finite length
since $T^{n}x=0$ or $1$ for some $n$. Indeed, if $x=p/q\in(0,1)$ is in lowest terms, an $e$-branch sends it to a fraction whose unreduced denominator is $p<q$. For an $o$-branch, writing $u=q-p$, the unreduced denominator is either $q-ku$ or $(k+1)u-q$, according to the branch, and is again a positive integer smaller than $q$. Thus the denominator strictly decreases until an iterate reaches $0$ or $1$.
\end{defn}

The restriction of $T$ to each branch is bijective, and the inverse maps, called inverse branches, are as follows:
\begin{equation}\label{eq:inv_br}
 \begin{aligned}
 h_{(a, \veps)_e}(x):= T|_{I_{(a, \veps)_e}}^{-1} = \frac{1}{2a + \veps x}\qquad\text{and}\qquad
 h_{(a, \veps)_o}(x):= T|_{I_{(a, \veps)_o}}^{-1} =\frac{1}{1 + \dfrac{1}{(a-\overline{\veps}) + \cfrac{\veps}{1+x}}},
 \end{aligned}
 \end{equation}
 where we set $\overline{\veps}=\max(0,\veps)$.
Compositions of the inverse branches provide a continued fraction expansion, which is a combination of the above forms of inverse branches corresponding to the letters in $\mathscr A$.
In other words, if $x\in I_{(a_1,\varepsilon_1)_{s_1}}\cap T^{-1}I_{(a_2,\varepsilon_2)_{s_2}}\cap\cdots\cap T^{-n+1}I_{(a_n,\varepsilon_n)_{s_n}}$, then we have 
\begin{equation}\label{eq:inv_br2}
x = h_{(a_1,\varepsilon_1)_{s_1}}\circ h_{(a_2,\varepsilon_2)_{s_2}}\circ \cdots\circ h_{(a_n,\varepsilon_n)_{s_n}}(T^n(x)).
\end{equation}

For example, $$\left\{\frac{-1+\sqrt{13/5}}{2}\right\}= \bigcap_{i=0}^\infty \left(T^{-2i}I_{(2,-1)_e}\cap T^{-2i-1}I_{(3,1)_o}\right),$$ and
$$\frac{-1+\sqrt{13/5}}{2}=\dfrac{1}{2\cdot{\bf 2}+\dfrac{\bf -1}{1+\dfrac{1}{{\bf 3}-1+\dfrac{\bf 1}{1+\dfrac{1}{2\cdot {\bf 2}+\dfrac{\bf -1}{1+\dfrac{1}{{\bf 3}-1+\cdots}}}}}}}
=[0;\overline{(2,-1)_e, (3,1)_{o}}].$$
Here, the overline indicates that the block $(2,-1)_e, (3,1)_{o}$ repeats periodically.

\subsection{Natural extension of $T$ and dual SCF}

To analyze the excursions of the geodesic flow, we shall introduce a symbolic coding of geodesic flows based on their endpoints in Section~\ref{sec:3}. This coding involves the SCF expansion of the forward endpoint, together with the \emph{dual SCF expansion} of the backward endpoint. For an oriented geodesic $\gamma$, its forward endpoint $\gamma_\infty$ is the boundary point approached by $\gamma(t)$ as $t\to+\infty$; its backward endpoint $\gamma_{-\infty}$ is defined analogously as $t\to-\infty$. Let us first describe the dual continued fraction precisely.

According to \cite{Rok49}, there exists a unique (up to isomorphism) minimal invertible measure-preserving dynamical system 
\((\Omega, \T, \bar{\mu})\) such that \(([0,1], T, \mu)\) is a factor of \((\Omega, \T, \bar{\mu})\).

Such a system {$(\Omega, \T,\bar\mu)$} is called the \emph{natural extension} of $T$.
For a given continued fraction map $T$, each inverse branch of the dual continued fraction map $\F$ is given by the second coordinate of the natural extension $\T$ \cite{Pan22}.
We construct a natural extension $\T$ of $T$ and the dual SCF as follows. 
A similar approach was first considered in \cite{NIT77,Nak81}.

Since the SCF map can be represented by a M\"obius transformation on the forward endpoints of geodesics of $\M$ lifted to the hyperbolic plane, we induce the inverse branches by extending the M\"obius transformation to the pairs of forward and backward endpoints.

We begin by introducing the inverse branches of the dual SCF.
For $(b,\eta)_s\in\mathscr{A}$, we define
\begin{equation}\label{eq:inv.barT}
 \bar{h}_{(b,\eta)_e}(y) = \frac{\eta}{2b+y},\qquad\text{ and }\qquad
 \bar{h}_{(b,\eta)_{o}}(y) = \frac{1}{1+\frac{\eta}{(b-\overline{\eta})+\frac{1}{1+y}}},
 \end{equation}
 where we set $\overline{\eta}:=\max(0,\eta)$.
We prove in Theorem~\ref{thm:natural-extension} that these maps are indeed the inverse branches of the dual map occurring in the natural extension.
This definition is canonical in the sense that with the SCF together with the dual SCF, we obtain the shift map on bi-infinite sequences, which are realized via the diagonal action of M\"obius transformations on pairs of geodesic endpoints as in \eqref{eq:rho.f} and \eqref{eq:rho.b}.

The domain of the dual SCF is defined by
$$\mathbb I = \overline{\left\{\bar{h}_{(b_1,\eta_1)_{t_1}}\circ\bar{h}_{(b_2,\eta_2)_{t_2}}\circ\cdots\circ\bar{h}_{(b_n,\eta_n)_{t_n}}(0): (b_i,\eta_i)_{t_i}\in\mathscr{A},\ n\ge 1\right\}}.$$
\begin{lem}
The domain of the dual SCF map satisfies
$$\mathbb I = [\sqrt{3}-2,\sqrt{3}].$$    
\end{lem}
\begin{proof}
By direct calculation, for any $(b,\eta)_t$ in $\mathscr A$, we have $$\bar{h}_{(b,\eta)_t}(\sqrt{3}-2), \ \bar{h}_{(b,\eta)_t}(\sqrt{3})\in[\sqrt{3}-2,\sqrt{3}].$$
From the strict monotonicity of $\bar{h}_{(b,\eta)_t}$, we have $\bar{h}_{(b,\eta)_t}[\sqrt{3}-2,\sqrt{3}]\subset[\sqrt{3}-2,\sqrt{3}]$.
By induction, for any $n\ge1$, $\bar{h}_{(b_1,\eta_1)_{t_1}}\circ\bar{h}_{(b_2,\eta_2)_{t_2}}\circ\cdots\circ\bar{h}_{(b_n,\eta_n)_{t_n}}(0)\in [\sqrt{3}-2,\sqrt{3}]$ and $\mathbb I\subset [\sqrt{3}-2,\sqrt{3}]$.
The minimum and maximum are given by
\begin{align*}
& \min\mathbb I = \lim_{n\to \infty}(\bar{h}_{(2,-1)_e})^n(0)
=\frac{-1}{4+\frac{-1}{4+\ddots}}=\sqrt{3}-2, \\
& \max\mathbb I = \lim_{n\to \infty}(\bar{h}_{(2,-1)_o})^n(0)=\frac{1}{1+\frac{-1}{2+\frac{1}{1+\frac{1}{1+\frac{-1}{2+\frac{1}{1+\ddots}}}}}}=\sqrt{3}.
\end{align*}

Denote by $\bar{I}_{(b,\eta)_t}$ the closed interval with endpoints $\bar{h}_{(b,\eta)_t}(\sqrt{3})$ and $\bar{h}_{(b,\eta)_t}(\sqrt{3}-2)$.
Explicitly, we obtain
$$\begin{cases}
\bar{I}_{(b,-1)_e}=\left[\frac{-1}{2b-2+\sqrt{3}},\frac{-1}{2b+\sqrt{3}}\right],~b\ge 2, 
&\bar{I}_{(b,+1)_e}=\left[\frac{1}{2b+\sqrt{3}},\frac{1}{2b-2+\sqrt{3}}\right],~b\ge 1, \vspace{1ex}\\
\bar{I}_{(b,+1)_o}=\left[\frac{2b-3+\sqrt{3}}{2b-1+\sqrt{3}},\frac{2b-1+\sqrt{3}}{2b+1+\sqrt{3}}\right],~b\ge 2,
&\bar{I}_{(b,-1)_o}=\left[\frac{2b+1+\sqrt{3}}{2b-1+\sqrt{3}},\frac{2b-1+\sqrt{3}}{2b-3+\sqrt{3}}\right],~b\ge2.
\end{cases}$$
Therefore, ${\bar{I}_{(b,\eta)_t}}$, $(b,\eta)_t\in\mathscr A$ form a partition of $[\sqrt{3}-2,\sqrt{3}]$.
By induction, the intervals $\bar{I}_{(b_1,\eta_1)_{t_1}\cdots(b_n,\eta_n)_{t_n}}$ between
$$\bar{h}_{(b_1,\eta_1)_{t_1}}\circ\bar{h}_{(b_2,\eta_2)_{t_2}}\circ\cdots\circ\bar{h}_{(b_n,\eta_n)_{t_n}}(\sqrt{3})\quad \text{ and }\quad \bar{h}_{(b_1,\eta_1)_{t_1}}\circ\bar{h}_{(b_2,\eta_2)_{t_2}}\circ\cdots\circ\bar{h}_{(b_n,\eta_n)_{t_n}}(\sqrt{3}-2)$$
form a partition {of $[\sqrt{3}-2,\sqrt{3}]$}.
Since the length of $\bar{I}_{(b_1,\eta_1)_{t_1}\cdots(b_n,\eta_n)_{t_n}}$ goes to $0$ as $n\to \infty$, $$\left\{\bar{h}_{(b_1,\eta_1)_{t_1}}\circ\bar{h}_{(b_2,\eta_2)_{t_2}}\circ\cdots\circ\bar{h}_{(b_n,\eta_n)_{t_n}}(0): (b_i,\eta_i)_{t_i}\in\mathscr{A}\right\}$$ is dense in $[\sqrt{3}-2,\sqrt{3}]$.
\end{proof}
For all $x\in\mathbb I$, there exists either an infinite or a finite sequence
\begin{equation}\label{eq:DSCF}\langle (b_1,\eta_1)_{t_1},(b_2,\eta_2)_{t_2},\dots,(b_n,\eta_n)_{t_n},\cdots\rangle, \quad\text{ or }\quad\langle (b_1,\eta_1)_{t_1},(b_2,\eta_2)_{t_2},\cdots (b_n,\eta_n)_{t_n};\kappa\rangle,\;\;\kappa\in\{0,1\}
\end{equation}
such that 
$$x = \lim_{n\to\infty}\bar{h}_{(b_1,\eta_1)_{t_1}}\circ\bar{h}_{(b_2,\eta_2)_{t_2}}\circ\cdots\circ\bar{h}_{(b_n,\eta_n)_{t_n}}(0),\;\; \text{ or }\;\; x = \bar{h}_{(b_1,\eta_1)_{t_1}}\circ\bar{h}_{(b_2,\eta_2)_{t_2}}\circ\cdots\circ\bar{h}_{(b_n,\eta_n)_{t_n}}(\kappa),\;\;\kappa\in\{0,1\},$$
respectively.
In our convention, $0$ and $1$ correspond to the empty sequence.

\begin{thm}\label{thm:natural-extension}
Let $\Omega: =[0,1]\times[\sqrt{3}-2,\sqrt{3}].$
The map $\T:\Omega\to\Omega$ given by
%$$\T(x,y) = \begin{cases}
$$\T(x,y) = (\CF(x),\bar{h}_{(a_1,\varepsilon_1)_{s_1}} (y)) \quad\text{ if }x\in I_{(a_1,\varepsilon_1)_{s_1}}$$
is the natural extension of $T$ equipped with a $\T$-invariant probability measure
$$d\bar\mu = \frac{dxdy}{\log (2+ \sqrt{3})(1+xy)^2}.$$
\end{thm}
\begin{proof}
Let  
$$\Omega' = ((0,1)\setminus \mathbb Q)\times\left((\mathbb I\setminus \mathbb Q) \setminus\bigcup_{n=0}^\infty\bigcup_{A^{(n)}\in\mathscr A^{n}}\bar h_{A^{(n)}}\bigl(\{\sqrt3-2,\sqrt3\}\bigr)\right).$$
Here, $\mathscr A^0 = \{e\}$ and $\bar h_{e}(x)=x$.

For $(x,y)\in\Omega'$, let
$$x = [0;(a_1,\varepsilon_1)_{s_1},(a_2,\varepsilon_2)_{s_2},\cdots]\quad
\text{ and } \quad
y = \langle (b_1,\eta_1)_{t_1},(b_{2},\eta_{2})_{t_{2}},(b_{3},\eta_{3})_{t_{3}},\cdots\rangle,$$ where $(a_n,\varepsilon_n)_{s_n}, (b_n,\eta_n)_{t_n}\in\mathscr A$.
By definition of $\T$, 
\begin{equation}\label{eq:Tbar}
\T(x,y) = \left([0;(a_2,\varepsilon_2)_{s_2},(a_3,\varepsilon_3)_{s_3},\cdots],\langle (a_{1},\varepsilon_{1})_{s_1},(b_{1},\eta_{1})_{t_{1}},(b_{2},\eta_{2})_{t_{2}},\cdots\rangle\right).
\end{equation}
Thus, the restriction $\T|_{\Omega'}$ is bijective since it is equivalent to the two-sided shift of $\mathscr A^{\mathbb Z}$.

We postpone the proof of $\T$-invariance of $\bar\mu$ to Lemma~\ref{le:Tbar-inv}.
\end{proof}
\begin{coro}\label{abscontinv}

The measure
\begin{align*}
d\mu(x) & = \frac{1}{\log(2+\sqrt{3})}\cdot \frac{1}{x}\left(\frac{1}{1-(2-\sqrt{3})x}-\frac{1}{1+\sqrt{3}x}\right)dx \\ & = \frac{2}{\log(2+\sqrt{3})}\cdot\frac{1}{(1-(2-\sqrt{3})x)(1+\sqrt{3}x)}dx
\end{align*}
is an absolutely continuous $T$-invariant probability measure.
\end{coro}

\begin{proof}
Since the identity
$$\int_{\sqrt{3}-2}^{\sqrt{3}}\frac{dy}{(1+xy)^2}=\frac{1}{x}\left(\frac{1}{1+(\sqrt{3}-2)x}-\frac{1}{1+\sqrt{3}x}\right)$$
holds, the measure $\mu$ is the marginal measure of $\bar\mu$  obtained by integrating along the $y$-variable.
\end{proof}

We recall that the inverse branches of the dual continued fraction map $\F$ are $\bar{h}_{(b,\eta)_t}$ defined in \eqref{eq:inv.barT}.
Hence, $\F:\mathbb I\to \mathbb I$ is given by
$$
\F(x) = \begin{cases}
-\frac{1}{x}-2k & \text{if }x\in \left[\frac{-1}{2k-2+\sqrt{3}},\frac{-1}{2k+\sqrt{3}}\right),~k\ge 2, \vspace{1ex}\\
\frac{1}{x}-2k & \text{if }x\in \left(\frac{1}{2k+\sqrt{3}},\frac{1}{2k-2+\sqrt{3}}\right],~k\ge 2,\vspace{1ex} \\
\frac{1}{x}-2 & \text{if }x\in \left(\frac{1}{2+\sqrt{3}},\frac{1}{\sqrt{3}}\right),\vspace{1ex} \\
\frac{k-(k+1)x}{kx-(k-1)},&\text{if }x\in\left[\frac{2k-3+\sqrt{3}}{2k-1+\sqrt{3}},\frac{2k-1+\sqrt{3}}{2k+1+\sqrt{3}}\right),~k\ge 2,\vspace{1ex} \\
\frac{kx-(k+1)}{k-(k-1)x},&\text{if }x\in\left(\frac{2k+1+\sqrt{3}}{2k-1+\sqrt{3}},\frac{2k-1+\sqrt{3}}{2k-3+\sqrt{3}}\right],~k\ge2
\end{cases}
$$
with $\F(0)=0$ and $\F(1)=1$.
See Figure~\ref{fi:F} for the graph of $\F$.

   \begin{defn}\label{def:DCF}
The dual SCF is given by the sequence as in \eqref{eq:DSCF}.
More precisely, we define $B: \mathbb I \to \mathscr{A}$ by
\[
    B(x) := 
        (b, \eta)_{t} \;\; \quad \text{if }x\in \bar{I}_{(b,\eta)_t}.\\
   \]  
  Letting $(b_{n}, \eta_{n})_{t_{n}}
:= B(\F^{n-1}x) \in \mathscr{A},$ for $x\in \mathbb I$, a \emph{dual spliced continued fraction expansion} is given by 
$$x=\langle(b_{1},\eta_{1})_{t_{1}}, (b_{2},\eta_{2})_{t_{2}}, \dots, (b_{n},\eta_{n})_{t_{n}}, \dots \rangle.$$ 
Rational numbers have finite length expansion $\langle(b_{1},\eta_{1})_{t_{1}}, (b_{2},\eta_{2})_{t_{2}}, \dots, (b_{n},\eta_{n})_{t_{n}};\kappa\rangle$ since $\kappa=\F^{n}x\in\{0,1\}$ for some $n$.
If $x\not\in\mathbb Q$, then it has an infinite length expansion.
   \end{defn}

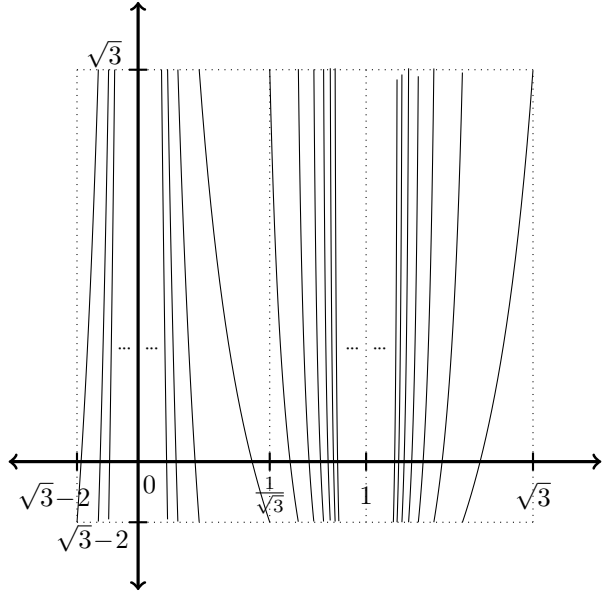
\begin{figure}
\begin{tikzpicture}[scale=3]
\draw[<->,very thick] (1.732-2-0.3,0)--(1.732+0.3,0);
\draw[<->,very thick] (0,1.732-2-0.3)--(0,1.732+0.3);
\draw[dotted] (1.732-2,1.732-2)--(1.732-2,1.732);
\draw[dotted] (1.732-2,1.732-2)--(1.732,1.732-2);
\draw[dotted] (1.732,1.732-2)--(1.732,1.732);
\draw[dotted] (1.732-2,1.732)--(1.732,1.732);
\draw[dotted] (1/1.732,1.732)--(1/1.732,1.732-2);
\draw[dotted] (1,1.732)--(1,1.732-2);

\draw[thick] (1/1.732,-0.04)--(1/1.732,0.04);
\draw[thick] (1,-0.04)--(1,0.04);
\draw[thick] (1.732,-0.04)--(1.732,0.04);
\draw[thick] (1.732-2,-0.04)--(1.732-2,0.04);
\draw[thick] (-0.04,1.732)--(0.04,1.732);
\draw[thick] (-0.04,1.732-2)--(0.04,1.732-2);

\node at (1.732,0-0.15) {\small$\sqrt{3}$};
\node at (1.732-2-0.1,0-0.15) {\small$\sqrt{3}\!-\!2$};
\node at (0+0.05,0-0.1) {\small$0$};
\node at (1/1.732,0-0.15) {\small$\frac{1}{\sqrt{3}}$};
\node at (1,0-0.15) {\small$1$};
\node at (0-0.15,1.732+0.07) {\small$\sqrt{3}$};
\node at (0-0.2,1.732-2-0.07) {\small$\sqrt{3}\!-\!2$};

\node at (-0.04,0.5) {\tiny$\cdot$};
\node at (-0.06,0.5) {\tiny$\cdot$};
\node at (-0.08,0.5) {\tiny$\cdot$};
\node at (0.04,0.5) {\tiny$\cdot$};
\node at (0.06,0.5) {\tiny$\cdot$};
\node at (0.08,0.5) {\tiny$\cdot$};
\node at (1-0.04,0.5) {\tiny$\cdot$};
\node at (1-0.06,0.5) {\tiny$\cdot$};
\node at (1-0.08,0.5) {\tiny$\cdot$};
\node at (1+0.04,0.5) {\tiny$\cdot$};
\node at (1+0.06,0.5) {\tiny$\cdot$};
\node at (1+0.08,0.5) {\tiny$\cdot$};

\draw[domain=-1/5.732:-1/3.732, black] plot (\x, {-1/\x-4});
\draw[domain=-1/7.732:-1/5.732, black] plot (\x, {-1/\x-6});
\draw[domain=-1/9.732:-1/7.732, black] plot (\x, {-1/\x-8});

\draw[domain=1/3.732:1/1.732, black] plot (\x, {1/\x-2});
\draw[domain=1/5.732:1/3.732, black] plot (\x, {1/\x-4});
\draw[domain=1/7.732:1/5.732, black] plot (\x, {1/\x-6});
\draw[domain=1/9.732:1/7.732, black] plot (\x, {1/\x-8});

\draw[domain=1/1.732:6/11+1.732/11, black] plot (\x, {(2-3*\x)/(2*\x-1)});
\draw[domain=6/11+1.732/11:32/46+2*1.732/46, black] plot (\x, {(3-4*\x)/(3*\x-2)});
\draw[domain=32/46+2*1.732/46:9*7/78-3/78+2*1.732/78, black] plot (\x, {(4-5*\x)/(4*\x-3)});
\draw[domain=9*7/78-3/78+2*1.732/78:48/59+1.732/59, black] plot (\x, {(5-6*\x)/(5*\x-4)});
\draw[domain=48/59+1.732/59:70/83+1.732/83, black] plot (\x, {(6-7*\x)/(6*\x-5)});
\draw[domain=70/83+1.732/83:96/111+1.732/111, black] plot (\x, {(7-8*\x)/(7*\x-6)});

\draw[domain=2-1.732/3:1.732, black] plot (\x, {(2*\x-3)/(2-\x)});
\draw[domain=16/11-1.732/11:2-1.732/3, black] plot (\x, {(3*\x-4)/(3-2*\x)});
\draw[domain=30/23-1.732/23:16/11-1.731/11, black] plot (\x, {(4*\x-5)/(4-3*\x)});
\draw[domain=96/78-1.732/39:30/23-1.732/23, black] plot (\x, {(5*\x-6)/(5-4*\x)});
\draw[domain=13*11/118-3/118-2*1.732/118:96/78-1.731/39, black] plot (\x, {(6*\x-7)/(6-5*\x)});
\draw[domain=15*13/166-3/166-2*1.732/166:13*11/118-3/118-2*1.732/118, black] plot (\x, {(7*\x-8)/(7-6*\x)});
\draw[domain=126/111-1.732/111:15*13/166-3/166-2*1.732/166, black] plot (\x, {(8*\x-9)/(8-7*\x)});

\end{tikzpicture}
\caption{The graph of the dual SCF map $\widehat T$}
\label{fi:F}
\end{figure}

The dual SCF expansion is obtained by iterations of the inverse branches $\bar{h}_{(b,\eta)_t}$ of $\F$ given in \eqref{eq:inv.barT}.
For example, we compute
$$\frac{-9+4\sqrt{3}}{11}=\dfrac{\bf -1}{2\cdot{\bf 2}+\dfrac{1}{1+\dfrac{\bf -1}{{\bf 3}+\dfrac{1}{1+\dfrac{\bf -1}{2\cdot{\bf 2}+\dfrac{1}{1+\dfrac{\bf -1}{{\bf3}+\dfrac{1}{1+\ddots}}}}}}}}.$$
Indeed, $\frac{-9+4\sqrt{3}}{11}$ satisfies the equation
$x = \frac{-1}{4+\frac{1}{1+\frac{-1}{3+\frac{1}{1+x}}}}.$

\section{Coding of geodesic flow on $T^1\M$}\label{sec:3}

Throughout the article, by a geodesic, we mean an oriented geodesic. For a geodesic $\gamma$ in $\bH$, we denote its forward and backward endpoints on the boundary $\partial \bH \simeq \bR \cup \{\infty\}$ by $\gamma_\infty$ and $\gamma_{-\infty}$, respectively.
As mentioned in the introduction, we consider the tessellation $\mathcal T$ given by the $\Theta$-translates of the ideal quadrilateral $\mathcal Q$ (see Figure~\ref{fi:tes}).
We remark that $\mathcal Q$ is the union of two copies of the fundamental domain $\mathcal F$ of the group $\Theta.$
\begin{figure}
\begin{tikzpicture}[scale=1.2]
\fill[gray!30] (-1,0) arc (180:0:.5) arc (180:0:.5) -- (1,4) -- (-1,4) -- (-1,0);
    
\draw (-5,0) arc (180:0:.5);
\draw (-4,0) arc (180:0:.5);
\draw (-3,0) arc (180:0:.5);
\draw (-2,0) arc (180:0:.5);
\draw (-1,0) arc (180:0:.5);
\draw (0,0) arc (180:0:.5);
\draw (1,0) arc (180:0:.5);
\draw (2,0) arc (180:0:.5);
\draw (3,0) arc (180:0:.5);
\draw (4,0) arc (180:0:.5);

\draw (0,0) arc (180:0:1/6);
\draw (1/3,0) arc (180:0:1/12);
\draw (1/2,0) arc (180:0:1/4);
\draw (0,0) arc (0:180:1/6);
\draw (-1/3,0) arc (0:180:1/12);
\draw (-1/2,0) arc (0:180:1/4);

\draw (2,0) arc (180:0:1/6);
\draw (2+1/3,0) arc (180:0:1/12);
\draw (2+1/2,0) arc (180:0:1/4);
\draw (2,0) arc (0:180:1/6);
\draw (2-1/3,0) arc (0:180:1/12);
\draw (2-1/2,0) arc (0:180:1/4);

\draw (4,0) arc (180:0:1/6);
\draw (4+1/3,0) arc (180:0:1/12);
\draw (4+1/2,0) arc (180:0:1/4);
\draw (4,0) arc (0:180:1/6);
\draw (4-1/3,0) arc (0:180:1/12);
\draw (4-1/2,0) arc (0:180:1/4);

\draw (-2,0) arc (180:0:1/6);
\draw (-2+1/3,0) arc (180:0:1/12);
\draw (-2+1/2,0) arc (180:0:1/4);
\draw (-2,0) arc (0:180:1/6);
\draw (-2-1/3,0) arc (0:180:1/12);
\draw (-2-1/2,0) arc (0:180:1/4);

\draw (-4,0) arc (180:0:1/6);
\draw (-4+1/3,0) arc (180:0:1/12);
\draw (-4+1/2,0) arc (180:0:1/4);
\draw (-4,0) arc (0:180:1/6);
\draw (-4-1/3,0) arc (0:180:1/12);
\draw (-4-1/2,0) arc (0:180:1/4);

\draw (-5,0) -- (-5,4);
\draw (-3,0) -- (-3,4);
\draw[very thick] (-1,0) -- (-1,4);
\draw[very thick] (1,0) -- (1,4);
\draw (3,0) -- (3,4);
\draw (5,0) -- (5,4);

\draw (-5.5,0) -- (5.5,0); 
\node at (-5,-.3)  {$-5$};	
\node at (-4,-.3)  {$-4$};	
\node at (-3,-.3)  {$-3$};	
\node at (-2,-.3)  {$-2$};	
\node at (-1,-.3)  {$-1$};	
\node at (0,-.3)  {$0$};	
\node at (1,-.3)  {$1$};	
\node at (2,-.3)  {$2$};	
\node at (3,-.3)  {$3$};	
\node at (4,-.3)  {$4$};	
\node at (5,-.3)  {$5$};

\node at (1/3,-.3) {$\frac13$};
\node at (1/2,-.3) {$\frac12$};

\node at (0,2) {$\mathcal Q$};
	
 \end{tikzpicture}
\caption{Tessellation $\mathcal T$ with the quadrilateral $\mathcal Q$ whose vertices are $-1,0,1,\infty$.}
\label{fi:tes}
\end{figure}
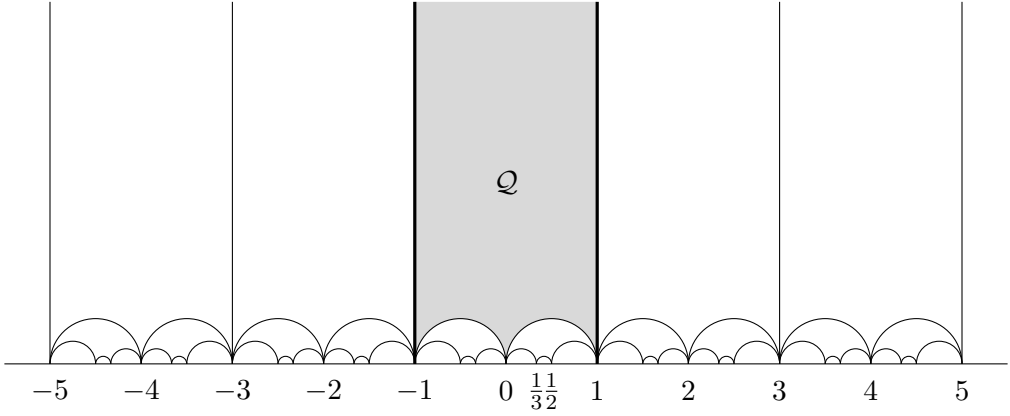
\subsection{Coding of the geodesic flow}\label{sec:cod} 
In this subsection, we construct the first return map $\Phi$ of the geodesic flow $g_t$ on $T^1\mathcal M$ with respect to a suitable cross section $\pi(X)$ and show that it is equivalent to a double cover of the natural extension.

Recall that we denote the projection map by $\pi:\bH\to \M.$ By abuse of notation, let us denote the projection map $T^1 \bH \to T^1 \M$ by $\pi$ as well. Let us first define the cross section $\pi(X).$
We first define $X^{\pm}$, which consists of \emph{some} unit tangent vectors based on the vertical sides of the fundamental domain, pointing outward the domain.
\begin{defn}[Cross section]
For $v\in T^1\bH$, we denote by $\gamma^v$ the geodesic in $\bH$ determined by $v$, and denote by $p(v)$ its base point.
Let $X=X_+\cup X_-$ with
$$
X_\pm = \left\{ v \in T^1\mathbb{H}^2 \;\middle|\; 
\mathrm{Re}(p(v)) = \pm 1 \;\text{ and }\; (\gamma^v_\infty, \gamma^v_{-\infty}) \in \mathcal{I}_\pm \right\},
$$
where the two sets $\mathcal I_\pm$ are
\begin{equation}\label{eq:I}
\mathcal I_+ = (1,\infty)\times [-\sqrt{3},-\sqrt{3}+2),\quad\text{ and }\quad\mathcal I_- = (-\infty,-1)\times (\sqrt{3}-2,\sqrt{3}].
\end{equation}
We set $\mathcal I:=\mathcal I_+\cup\mathcal I_-$.
We define the cross section to be $\pi(X) \subset T^1\M$.
\end{defn}

\begin{lem}\label{le:lift}
    Any geodesic $\gamma_{0}$ on $\M$ that is neither the projection of an edge of the tessellation $\mathcal T$ nor $\Theta.(0, \infty)$, $\Theta.(-1, 1)$ has a lift $\gamma$ on $\bH$ such that $(\gamma_{\infty}, \gamma_{-\infty}) \in \mathcal I$. In particular, this holds for almost every geodesic on $\M$.
\end{lem}

\begin{proof}
    The exceptional lines include the geodesic with endpoints $(-1,1)$ or $(0, \infty)$ and its $\Theta$-translates; this is a negligible family. Excluding them, fix a lift $\gamma$ of $\gamma_{0}$. We may assume that $\gamma_{-\infty}<\gamma_{\infty}$.
    The ideal triangle $\mathcal F$ with vertices $-1$, $1$, and $\infty$ is a fundamental domain of $\Theta$.
    For a geodesic $\gamma$ that passes through $\mathcal F$, there are three possible cases: 
    \begin{enumerate}
        \item $\gamma_{-\infty}<-1<1<\gamma_{\infty}$: $\gamma$ intersects the vertical lines $x=-1$ and $x=1$;
        \item $-1<\gamma_{-\infty}<1<\gamma_{\infty}$: either $\gamma_\infty=\infty$ or $\gamma$ intersects the semicircle with endpoints $-1$ and $1$, and the vertical line $x = 1$;
        \item $\gamma_{-\infty}<-1<\gamma_{\infty}<1$: $\gamma$ intersects the vertical line $x=-1$ and the semicircle with endpoints $-1$ and $1$.
    \end{enumerate}
    
   (1) In the first case, after translation by $\tau^{n}: z \mapsto z+2n$ for some $n \in \mathbb Z_{\ge 0}$, we have $\gamma_{-\infty} \in (-\sqrt{3},-\sqrt{3}+2)$ and $\gamma_{\infty}>1$. 
   
   (2) For the second case, if $\gamma_\infty=\infty$, then we replace $\gamma$ by $\sigma\gamma$, to ensure $\gamma_\infty <\infty$.
   If $\gamma_{-\infty}< 2-\sqrt{3}$, then $(\gamma_{\infty}, \gamma_{-\infty}) \in \mathcal I_+$. 
   When $2-\sqrt{3}<\gamma_{-\infty}<1$, if $\gamma_{\infty}>3$, then translating by $\tau^{-1}$ yields $(\gamma_{\infty}, \gamma_{-\infty}) \in \mathcal I_+$. 
   If $\gamma_{\infty} \in (1, 3]$, we repeatedly apply the map $g(z) = \tau\sigma(z) = 2 - 1/z$, which is a contraction on $(1,3]$.
        For the forward endpoint $z=\gamma_\infty \in (1, 3]$, we have $g(z) \in (1, 5/3] \subset (1, \infty)$, so it remains valid under iteration.
        For the backward endpoint $y=\gamma_{-\infty} \in [2-\sqrt{3}, 1)$, observe that $g$ is strictly increasing ($g'(y) = 1/y^2 > 0$) and maps the interval $[2-\sqrt{3}, 1)$ onto $[-\sqrt{3}, 1)$. Furthermore, $g(y) < y$ for all $y \in (0, 1)$ (as $(y-1)^2 > 0$).
        The orbit defined by $y_{k+1} = g(y_k)$ is strictly decreasing. Since the fixed point $1$ does not belong to the interval, the orbit must eventually leave $[2-\sqrt{3}, 1)$ to the left. Let $k$ be the smallest integer such that $y_k < 2-\sqrt{3}$. Since the map sends $[2-\sqrt{3}, 1)$ onto $[-\sqrt{3}, 1)$, and $y_{k-1} \ge 2-\sqrt{3}$, we must have $y_k \ge g(2-\sqrt{3}) = -\sqrt{3}$.
        Therefore, $y_k \in [-\sqrt{3}, 2-\sqrt{3})$, and the lift $g^k \gamma$ is in $\mathcal{I}_+$.
        
    (3) The last case can be reduced to Case 1 or Case 2 by a translation: the new endpoints are $\tau\gamma_{\infty} = \gamma_{\infty}+2 \in (1, 3)$ and $\tau\gamma_{-\infty} = \gamma_{-\infty}+2 < 1$. This reduces the configuration to Case 2 (if $\tau\gamma_{-\infty} > -1$) or Case 1 (if $\tau\gamma_{-\infty} < -1$).
\end{proof}

Let $\Phi$ be the first return map to $\pi(X)$ of the geodesic flow on $T^1\M$.
The map $\Phi$ is well-defined everywhere except for the vectors $u \in \pi(X)$ that escape to a cusp without intersecting the cross section $\pi(X)$ again, which are of Lebesgue measure 0. 
For an interval $I\subset \mathbb R$, let $I^{-1}$ denote $\{1/x:x\in I\}$.
Let $E =\left\{ \frac{k}{k-1},\ \frac{2k+1}{2k-1},\ k : k \ge 2 \right\}$ be the set of endpoints of $(I_{(k,\veps)_s})^{-1}$, where $I_{(k,\veps)_s}$ are domains of branches of $T$ as in \eqref{eqn:2.1}.
Let
$$\mathcal I^*= \{(\gamma_\infty, \gamma_{-\infty})\in\mathcal I: |\gamma_\infty| \notin E\} \quad\text{ and }\quad X^*= \{v\in X: |\gamma_\infty^{v}|\not\in E\}.$$

We shall see in Theorem~\ref{th:exc} that $\Phi:\pi(X^*)\to \pi(X)$ is well-defined and equivalent to a double cover of the natural extension $\widetilde{T}$.
For this purpose, we parametrize $\pi(X)$ by $\mathcal I$, and give an explicit description of $\Phi$.

Let us introduce the set of geodesics whose pairs of endpoints are contained in $\mathcal I$ 
and denote it by $A=A_+\cup A_-$, as follows:
$$A_\pm=\{\text{geodesics }\gamma\text{ on }\bH\text{ such that }(\gamma_\infty,\gamma_{-\infty})\in\mathcal I_\pm\}. $$
The following proposition formalizes the natural bijections between the four sets $\pi(X)$, $X$, $A$, and $\mathcal I$.

\begin{prop}\label{pr:bij}
(1) The restriction $\pi|_X: X\to \pi(X)$ is bijective.\\
(2) The canonical map $\varphi_1$ from $X$ to $A$, given by $v\mapsto \gamma^{v}$, is bijective.\\
(3) The canonical map $\varphi_2$ from $A$ to $\mathcal I$, given by $\gamma \mapsto (\gamma_{\infty},\gamma_{-\infty})$, is bijective.
\end{prop}

\begin{proof}
(1) Let $\ell_{\pm1}=\{v\in X: \mathrm{Re}(p(v)) = \pm1\}$.
 Let $u_1,u_2\in X$. 
The projection $\pi$ is injective on each of $\ell_1$ and $\ell_{-1}$. 
If $\mathrm{Re}(p(u_1))\not=\mathrm{Re}(p(u_2))$, then the vectors $\pi(u_1)$ and $\pi(u_2)$ point in different directions with respect to the line $\pi(\ell_1) = \pi(\ell_{-1})$ on $\mathcal{M}$, and hence $\pi(u_1) \neq 
\pi(u_2)$.

(2) Each $\gamma\in A_+$ intersects $\ell_{+1}$ at exactly one point, and similarly for $A_-$ and $\ell_{-1}$.
Thus, $\varphi_1$ has a well-defined inverse and is therefore bijective.

(3) Each $(\alpha,\beta)\in \mathcal I$ determines a unique geodesic on $\bH$. Thus, $\varphi_2$ is bijective.
\end{proof}

Let $ A^*=\{\gamma\in A:|\gamma_\infty|\not\in E\}.$
Using the above proposition, we first define the map $\rho: A^* \to A$, which acts as a shift map on the pair consisting of the SCF of $\gamma_\infty$, and the dual SCF of $\gamma_{-\infty}$. 

\begin{defn} Denote by $j_\gamma$ the sign of $\gamma_\infty$.
For a geodesic $\gamma\in A$ with endpoints given by
\begin{equation*}\begin{cases}
\gamma_\infty=j_\gamma\cdot[(a_1,\veps_1)_{s_1};(a_2,\veps_2)_{s_2},(a_3,\veps_3)_{s_3},\dots],\\
\gamma_{-\infty}=-j_\gamma\cdot\langle(a_0,\veps_0)_{s_0},(a_{-1},\veps_{-1})_{s_{-1}},(a_{-2},\veps_{-2})_{s_{-2}},\dots\rangle,
\end{cases}\end{equation*}
we define the map $\rho$ by
\begin{equation}\label{eq:rho.shift}\begin{cases}
(\rho(\gamma))_\infty=-\veps_1 j_\gamma \cdot[(a_2,\veps_2)_{s_2};(a_3,\veps_3)_{s_3},(a_4,\veps_4)_{s_4},\dots],\\
(\rho(\gamma))_{-\infty}=\veps_1 j_\gamma \cdot\langle(a_1,\veps_1)_{s_1},(a_0,\veps_0)_{s_0},(a_{-1},\veps_{-1})_{s_{-1}},(a_{-2},\veps_{-2})_{s_{-2}},\dots\rangle.
\end{cases}\end{equation}
 
\end{defn}
The map $\rho$ is well-defined by the definitions of the SCF and the dual SCF.
The formulas below follow from two observations. Since $|\gamma_\infty|>1$, the forward-endpoint formula must use the reciprocal $1/(j_\gamma\gamma_\infty)$. Moreover, a branch of $T$ may have determinant $-1$, and the factor $\varepsilon=\veps_1(|1/\gamma_\infty|)$ compensates for the resulting sign. On the backward endpoint, the maps $\bar h$ in~\eqref{eq:inv.barT} were chosen to realize the corresponding dual inverse branch.
In terms of the generators of $\Theta$, the map $\rho: A^* \to A$ is given by
\begin{equation} \label{eq:rho.def}
\rho(\gamma) = 
\begin{cases}
\sigma\tau^{-j_\gamma k}(\gamma), & \text{if } |\gamma_\infty| \in (2k, 2k+1), \quad k \ge 1, \\
\sigma\tau^{-j_\gamma k}(\gamma), & \text{if } |\gamma_\infty| \in (2k-1, 2k), \quad k \ge 2, \\
\tau^{-j_\gamma}(\sigma\tau^{-j_\gamma})^{k-1}(\gamma), & \text{if } |\gamma_\infty| \in \left(\frac{2k+1}{2k-1}, \frac{k}{k-1}\right), \quad k \ge 2, \\
(\sigma\tau^{-j_\gamma})^{k}(\gamma), & \text{if } |\gamma_\infty| \in \left(\frac{k+1}{k}, \frac{2k+1}{2k-1}\right), \quad k \ge 2,
\end{cases}
\end{equation}

The explicit formulas for the above linear fractional maps are as follows:
\begin{equation}\label{eq:rho.form}
\begin{matrix}
\sigma\tau^{-j_\gamma k}(z) = j_\gamma\cdot \dfrac{1}{2k-j_\gamma z},\\
\tau^{-j_\gamma}(\sigma\tau^{-j_\gamma})^{k-1}(z) = j_\gamma\cdot \dfrac{k j_\gamma z-(k+1)}{k-(k-1)j_\gamma z}\quad\text{and}\quad
(\sigma\tau^{-j_\gamma})^{k}(z) = j_\gamma\cdot\dfrac{-k+ (k-1) j_\gamma z}{kj_\gamma z-(k+1)}.
\end{matrix}
\end{equation}
Consequently, $\rho$ is directly related to $T$ via
\begin{equation}\label{eq:rho.f}
\rho(\gamma_\infty) = -\varepsilon \cdot \frac{j_\gamma}{T\left(\frac{1}{j_\gamma \gamma_\infty}\right)}.
\end{equation}
For the backward endpoint, we observe that
\begin{equation}\label{eq:rho.b}
\rho(\gamma_{-\infty}) = \veps j_\gamma\cdot\bar h_{(k,\veps)_s}(-j_\gamma \gamma_{-\infty})\quad\text{if }|1/\gamma_\infty|\in I_{(k,\veps)_s}.
\end{equation}

\begin{thm}\label{th:exc}
Let $\varphi = \varphi_1\circ (\pi|_X)^{-1}$, where $\varphi_1$ is the bijection in Proposition~\ref{pr:bij}.

\noindent(1) The first return map $\Phi$ is equivalent to $\rho$:
$$\begin{tikzpicture}
\node at (0,0) {$\pi(X^*)$};
\node at (5,0) {$\pi(X)$};
\node at (5,-1.5) {$A$};
\node at (0,-1.5) {$A^*$};

\draw[->] (0.6,0)-- node[above] {$\Phi$} (4.4,0);
\draw[->] (0.6,-1.5)-- node[above] {$\rho$} (4.4,-1.5);

\draw[->] (0,-0.5)-- node[left] {$\varphi$} (0,-1);
\draw[->] (5,-0.5)-- node[right] {$\varphi$} (5,-1);

\node at (2.5,-0.6) {$\circlearrowleft$};
\end{tikzpicture}$$

\noindent(2) For $u\in \pi(X)$, $|\gamma^u_\infty|\in E$ if and only if $\gamma^u$ escapes to a cusp without intersecting $\pi(X)$ for positive time.
\end{thm}

\begin{proof}
 Let $v \in \pi(X^*)$ and let $\gamma = \varphi(v) \in A^*$.
Without loss of generality, we assume that $\gamma_\infty >0$.
Let $(k,\veps)_s \in \mathscr A$ be such that $\gamma_\infty$ lies in $(I_{(k,\veps)_s})^{-1}$.
Denote by $\ell(\alpha,\beta)$ the geodesic connecting $\alpha$ and $\beta$.
Let
\begin{equation}\label{eq:Cke}
\mathcal C_{(k,\veps)_s}=\text{the geodesic connecting the endpoints of the interval $(I_{(k,\veps)_s})^{-1}$}
\end{equation}
(see the blue half circle in Figures~\ref{fig:5} and \ref{fig:6}).

Let $(t_i)$ be the sequence of times at which $\gamma$ crosses the edges of $\mathcal{T}$ in forward time.
We first claim that the time $t_k$ at which  $\gamma (t_k)$ lies on $\mathcal C_{(k,\veps)_s}$ is the first return time.

If $s=e$, then $\gamma$ crosses the $k-1$ vertical lines $\ell(2i+1,\infty)$, for $i=1,\dots,k-1$ of $\mathcal T$, that is, the base point of $\gamma(t_i)$ is contained in $\ell(2i+1,\infty)$.
See Figure~\ref{fig:5}.
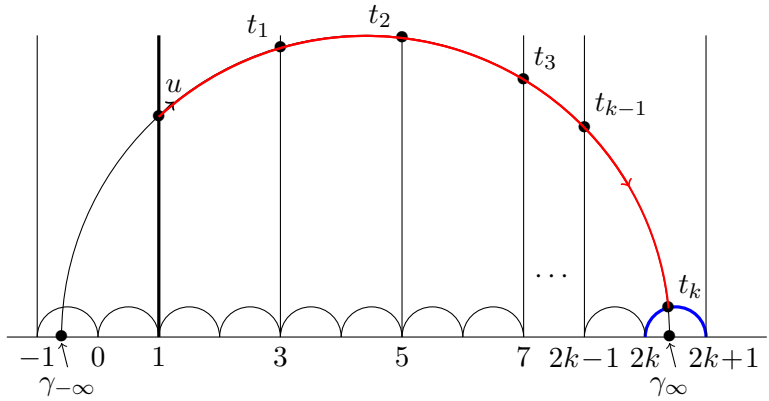
\begin{figure}[h]\color{black}
\begin{tikzpicture}[scale=0.8]
\draw (-1.5,0)--(11,0);
\draw (0,0) arc (180:0:1/2); 
\draw (0,0) arc (0:180:1/2); 
\draw[very thick] (1,0) -- (1,5);
\draw (-1,0) -- (-1,5);
\node[below] at (0,0) {$0$};
\node[below] at (1,0) {$1$};
\node[below] at (-1,0) {$-1$};

\node at (7.5,1) {$\cdots$};

\draw (2,0) arc (180:0:1/2);
\draw (2,0) arc (0:180:1/2);
\draw (3,0) -- (3,5); 
\node[below] at (3,0) {$3$};

\draw (4,0) arc (180:0:1/2);
\draw (4,0) arc (0:180:1/2);
\draw (5,0) -- (5,5); \node[below] at (5,0) {$5$};

\draw (6,0) arc (180:0:1/2);
\draw (6,0) arc (0:180:1/2);
\draw (7,0) -- (7,5); \node[below] at (7,0) {$7$};

\draw (8,0) -- (8,5);
\draw[very thick, blue] (9,0) arc (180:0:1/2);
\draw (9,0) arc (0:180:1/2);
\draw (10,0) -- (10,5);

\node[below] at (9,0) {$2k$};
\node[below] at (10.3,0) {$2k\!+\!1$};
\node[below] at (8,0) {$2k\!-\!1$};

%gamma geodesic
\draw (-1/2-.1,0) arc (180:0:5);
\node at (-1/2-.1,0) {$\bullet$};
\node[below] at (-1/2,-0.5) {$\gamma_{-\infty}$};
\draw[->] (-1/2,-0.5) -- (-1/2-.1,-0.1);
\node at (10-1/2-.1,0) {$\bullet$};
\node[below] at (10-1/2-.1,-0.5) {$\gamma_{\infty}$};
\draw[->] (10-1/2,-0.5) -- (10-1/2-.1,-0.1);

\node at (3, 4.8) {$\bullet$};
\node[above left] at (3, 4.8) {$t_1$};
\node at (5, 4.963) {$\bullet$};
\node[above left] at (5, 4.963) {$t_2$};
\node at (7, 4.27) {$\bullet$};
\node[above right] at (7, 4.27) {$t_3$};
\node at (8, 3.469) {$\bullet$};
\node[above right] at (8, 3.469) {$t_{k-1}$};

\node at (1, 3.65) {$\bullet$};
\node[above] at (1+0.25, 3.65+0.25) {$u$};
\draw[->, thick] (1, 3.65) -- (1+0.5/2,3.65+0.5/2);

\node at (9.376,0.484) {$\bullet$};
\node[above right] at (9.376,0.484) {$t_k$};

% 빨간 arc (xi to eta)
\draw[thick, red] (1, 3.65) arc (132.97:5.56:5); % 대략적인 각도 추정
\draw[thick, red,->] (1, 3.65) arc (132.97:30:5); % 화살표
\end{tikzpicture}
\caption{Illustration of the case $s=e$: the geodesic $\gamma$ successively crosses the $k-1$ vertical lines $\ell(2i+1,\infty)$, and meets $C_{(k,\varepsilon)_s}$ at time $t_k$, which is the first return time.}
\label{fig:5}
\end{figure}\\
The unique element of $\Theta$ that sends $\ell(2i+1, \infty)$ to $\ell(1, \infty)$ is $\tau^{-i}$, and the one that sends it to $\ell(-1, \infty)$ is $\tau^{-i-1}$.  
From 
$\tau^{-i}\gamma_{-\infty}<-\sqrt{3}$
for $i=1,\dots,k$, we deduce that
$\tau^{-i}\gamma(t_i)\not\in X$ and $\tau^{-i-1}\gamma(t_i)\not\in X.$

If $s=o$, then $\gamma$ crosses the $k-1$ arcs $\ell(1,\frac{i+1}{i})$ of $\mathcal T$ for $i=1,\dots,k-1$, that is, the base point of $\gamma(t_i)$ lies in $\ell(1,\frac{i+1}{i})$.
See Figure~\ref{fig:6}.
\begin{figure}\color{black}
\begin{tikzpicture}[scale=8]%, baseline={(0,-1)}]
\clip (1-0.1,-0.15) rectangle (2+0.1,0.7);
\draw (-1.1,0)--(2.1,0);
\draw (-1,0) -- (-1,1.25);
\draw (0,0) arc (180:0:1/2); 
\draw (0,0) arc (0:180:1/2); 
\draw[very thick] (1,0) -- (1,1.25);
%\draw (2,0) arc (180:135:1/2); 
\draw (2,0) arc (0:180:1/2); 
%\draw (3,0) -- (3, 1.5);
\node[below] at (0,0) {$0$};
\node[below] at (1,0) {$1$};
\node[below] at (-1,0) {$-1$};

\draw (1,0) arc (180:0:1/4) node[below] {\color{black}$\frac32$};
\draw (3/2,0) arc (180:0:5/6-3/4);
\draw (2,0) arc (0:180:1-5/6);

\draw (1,0) arc (180:0:1/6) node[below] {\color{black}$\frac43$};
\draw (4/3,0) arc (180:0:7/10-4/6);
\draw (3/2,0) arc (0:180:3/4-7/10);

\node at (5/4,0.13-0.01) {$.$};
\node at (5/4-0.01,0.13-0.02) {$.$};
\node at (5/4-0.02,0.13-0.03) {$.$};

\draw (1,0) arc (180:0:1/10+0.02) node[below] {\color{black}$\frac{k}{k-1}$};
\draw (1,0) arc (180:0:1/12);
\node[below] at (7/6-0.03,0) {$\frac{k+1}{k}$};
\draw[very thick, blue] (7/6,0) arc (180:0:13/22-7/12+0.01);
\draw (6/5+0.04,0) arc (0:180:6/10+0.02-7/12-13/22+7/12-0.01);

 \node[below] at (2,0) {$2$};

%\gamma geodesic
\draw (-1.1,0) arc (180:0:19/32+1.1/2);
%\node at (-0.4,0) {$\bullet$};
%\node[below] at (-0.4, 0) {$\gamma_{-\infty}$};
%\node at (19/16,0) {$\bullet$};
\node[below] at (19/16, -0.08) {$\gamma_{\infty}$};
\draw[->] (19/16,-0.08) -- (19/16,-0.01);

\node at (1.135,0.342) {$\bullet$};
\node at (1.163, 0.234) {$\bullet$};
\node at (1.175, 0.166) {$\bullet$};
%\node at (1.182, 0.111) {$\bullet$};
\node at (1.183, 0.102) {$\bullet$};
\node at (1.187, 0.017) {$\bullet$};

\node[above right] at (1.135-0.02,0.342+0.02) {$t_1$};
\node[above right] at (1.163, 0.234) {$t_2$};
\node[right] at (1.175, 0.166+0.03) {$t_3$};
%\node[right] at (1.182, 0.111) {$t_4$};
\node[left] at (1.183, 0.102) {$t_{k-1}$};
%excursion
\draw[thick, red] (1,0.624) arc(33.11:0.38:1.143);
\draw[thick, red,->] (1, 0.624) arc(33.11:6:1.143);

\node at (1,0.624) {$\bullet$};
\node[above right] at (1,0.624-0.03) {$u$};
\draw[->] (1,0.624) -- (1+1/30, 0.624-1.524/35);
%\node at (1.186,0.0076) {$\bullet$};
\node[above left] at (1.186,0.0076) {$t_k$};
\end{tikzpicture}
\caption{Illustration of the case $s=o$: the geodesic $\gamma$ successively crosses the $k-1$ arcs $\ell\!\left(1,\frac{i+1}{i}\right)$ and meets $C_{(k,\varepsilon)_s}$ at time $t_k$, which is the first return time.}
\label{fig:6}
\end{figure}

The unique element of $\Theta$ sending $\ell(1,\frac{i+1}{i})$ to $\ell(1,\infty)$ is $(\sigma\tau^{-1})^i$, and the one sending it to $\ell(-1,\infty)$ is $\tau^{-1}(\sigma\tau^{-1})^i$ for $i=1,\dots,k-1$.
In this case, $2-\sqrt{3}\le(\sigma\tau^{-1})^i\gamma_{-\infty}\le 1$.
Therefore, $(\sigma\tau^{-1})^i\gamma(t_i)\not\in X$.
Moreover, $\tau^{-1}(\sigma\tau^{-1})^i\gamma(t_i)\not\in X$ since $\tau^{-1}(\sigma \tau^{-1})^i(\gamma_\infty)>-1$.

In the case where $\gamma_\infty\not\in E$, 
whether $s = e$ or $s = o$, the $k$-th edge that $\gamma$ crosses in $\mathcal{T}$ is exactly $\mathcal{C}_{(k, \veps)_s}$, i.e., the base point of $\gamma(t_k)$ is contained in $\mathcal C_{(k,\veps)_s}$.
Let $g\in\Theta$ such that $g.\mathcal C_{(k,\veps)_s}=\ell(-\veps,\infty)$.
For each case, the explicit expression of $g$ is
$$g=\begin{cases}
\sigma\tau^{-k},&\text{if }s=e,\ \veps=\pm 1,\\
%\sigma\tau^{-k},&\text{if }s=e,\ \veps=-1,\\
\tau^{-1}(\sigma\tau^{-1})^{k-1},&\text{if }s=o, \ \veps=-1,\\
(\sigma\tau^{-1})^k,&\text{if }s=o,\ \veps=+1.
\end{cases}
$$
From \eqref{eq:rho.def}, we have $g.\gamma = \rho(\gamma)\in A$.
Thus $\Phi(v) = g.(\gamma(t_k))$.
Therefore, $\varphi \circ \Phi(v) = g.\gamma = \rho(\gamma) =\rho\circ \varphi(v)$.
Part (1), as well as our first claim, follows.

If $\gamma_{\infty}\in E$, then after crossing $k-1$ edges of $\mathcal T$, $\gamma$ does not intersect any further edge of $\mathcal T$, which implies that $\gamma$ escapes to a cusp.
\end{proof}

We now show that the first return map $\Phi$ is equivalent to the double cover $\widetilde T$ of the natural extension of $T$ defined as follows.
\begin{defn}
Let us consider the double cover of $\Omega$, namely, $$\Omega^\pm = (0,1)\times (\sqrt{3}-2,\sqrt{3}]\times\{\pm 1\},$$
and define 
$$(\Omega^\pm)^* = \Omega^\pm\setminus\{(x,y,j):x \in \partial I_{(k,\veps)_s} \text{ for some } (k,\veps)_s\}.$$
We define \emph{the double cover of the natural extension} $\widetilde T:(\Omega^\pm)^*\to\Omega^\pm$ by $$\widetilde T(x,y,j) = (\bar T(x,y), -\varepsilon_1j).$$
The last coordinate reflects the change in the direction of the geodesic after applying $\rho$.
\end{defn}
Let us define a bijection $J:A \to \Omega^\pm$ by
$$J(\gamma) = \begin{cases}\left(\frac1{\gamma_\infty},-\gamma_{-\infty},1\right)&\text{if }\gamma\in A_+,\\
\left(-\frac1{\gamma_\infty},\gamma_{-\infty},-1\right)&\text{if }\gamma\in A_-.
\end{cases}$$
From the definition of $\rho$ and \eqref{eq:Tbar}, it follows that $\widetilde T J = J \rho$.
By Theorem~\ref{th:exc}, $\Phi$ and $\widetilde T$ are equivalent.

\begin{thm}\label{th:com2}
The first return map $\Phi$ is equivalent to $\widetilde T$:
\begin{figure}[h] \centering
\begin{tikzpicture}
\node at (0,0) {$\pi(X^*)$};
\node at (5,0) {$\pi(X)$};
\node at (5,-1.5) {$\Omega^\pm$};
\node at (0,-1.5) {$(\Omega^\pm)^*$};
\draw[->] (0.6,0)-- node[above] {$\Phi$} (4.4,0);
\draw[->] (0.6,-1.5)-- node[above] {$\widetilde T$} (4.4,-1.5);
\draw[->] (0,-0.5)-- node[left] {$J\circ \varphi$} (0,-1);
\draw[->] (5,-0.5)-- node[right] {$J\circ \varphi$} (5,-1);
\node at (2.5,-0.6) {$\circlearrowleft$};
\end{tikzpicture}
\label{fi:diag2}
\end{figure}
\end{thm}
\color{black}

Using the bijection $\varphi_2 \circ \varphi : \pi(X) \to \mathcal{I}$ in Proposition~\ref{pr:bij} and the parametrization $(\alpha,\beta)\in\mathcal{I}$, we obtain a natural $\Phi$-invariant measure on $\pi(X)$ as the pullback $(\varphi_2\circ\varphi)^*\nu$ of 
$$\nu = \frac{d\alpha\, d\beta}{(\alpha - \beta)^2}.$$
By the change of variables $J\circ \varphi_2^{-1}:\mathcal I\to \Omega^\pm$, the pushforward measure is 
\begin{equation}\label{eq:msr}
\widetilde{\mu} := (J\circ \varphi_2^{-1})_*\nu = \frac{dx\, dy\, d\delta}{(1 + x y)^2},
\end{equation}
where $d\delta$ denotes the counting measure on $\{\pm1\}$. See \cite[Section~3.1]{Ser85} for further details.

\begin{lem}\label{le:Tbar-inv}
The probability measure
$$d\bar{\mu}=\frac{dxdy}{\log(2+\sqrt{3})(1+xy)^2}$$
is $\overline{T}$-invariant.
\end{lem}
\begin{proof}
The pushforward of the $\widetilde{T}$-invariant measure $\tilde{\mu}$ under the projection from $\Omega^\pm$ to $[0,1]\times \mathbb I$ is 
$$\frac{2dxdy}{(1+xy)^2}.$$
Thus, it is $\overline{T}$-invariant.
The normalizing constant is
\begin{align*}\int_0^1\int_{\sqrt{3}-2}^{\sqrt{3}}\frac{2dxdy}{(1+xy)^2} 
&= \int_0^1\frac{2}{x}\left(\frac{1}{1+(\sqrt{3}-2)x}-\frac{1}{1+\sqrt{3}x}\right)dx\\
&= 2\int_0^1 \left( \frac{\sqrt{3}}{1+\sqrt{3}x}-\frac{\sqrt{3}-2}{1+(\sqrt{3}-2)x}\right)dx =2\log(2+\sqrt{3}).
\end{align*}
\end{proof}

For almost every geodesic $\gamma_0$ on $\M$, fix a lift $\gamma\in A$ and set
$$z_\gamma:=J(\gamma)=\left(\frac{1}{|\gamma_\infty|},-j_\gamma\gamma_{-\infty},j_\gamma\right)\in\Omega^\pm.$$
Then the points corresponding to all lifts $\gamma'\in A$ of $\gamma_0$ form exactly the single two-sided orbit
$$\left\{J(\gamma'):\ \gamma'\in A,\ \pi(\gamma')=\gamma_0\right\}
=\left\{\widetilde T^n z_\gamma:\ n\in\mathbb Z\right\}.$$

\subsection{Return time of geodesic flow}\label{sec:3.2}
Let $r_1(x,y,j)$ be the first return time of $(J\circ \varphi)^{-1}(x,y,j)\in \pi(X^*)$ to $\pi(X)$, for $(x,y,j)\in (\Omega^\pm)^*$.
The suspension space with the roof function $r_1$ is defined by 
$$\Omega^{r_1}=(\Omega^\pm\times\mathbb R)/\sim,$$
where
$$(x,y,j,r_1(x,y,j)+s)\sim(\widetilde T(x,y,j),s)\qquad\text{for every }s\in\mathbb R,$$
and the suspension flow is $\phi_s[x,y,j,t]=[x,y,j,t+s]$.
The geodesics that escape into a cusp in forward or backward time form a set of Liouville measure zero. After discarding this set, there is a natural isomorphism between the suspension flow $\phi_s$ of $\Omega^{r_1}$ and the geodesic flow $g_s$ on $T^1\M$.

The first return time coincides with the hyperbolic length of the corresponding geodesic segment.
For almost every $v\in\pi(X^*)$, let $\widetilde v\in X^*$ be its unique lift and let $\tilde\gamma^v\in A$ be parameterized by $\dot{\tilde\gamma}^v(0)=\widetilde v$.
In the proof of Theorem~\ref{th:exc}, we deduce that 
if $1/\tilde\gamma^v_\infty\in \pm I_{(k,\veps)_s}$, then the first return point in $\pi(X)$ is the projection of the intersection of $\tilde\gamma^{{v}}$ 
and $\pm \mathcal C_{(k,\veps)_s}$, where $\mathcal  C_{(k,\veps)_s}$ is defined as in  \eqref{eq:Cke}.

\begin{figure}[t]
\begin{tikzpicture}[scale=1]
\draw (-1.1,0)--(6.2,0);
\draw (0,0) arc (180:0:1/2); 
\draw (0,0) arc (0:180:1/2); 
\draw (1,0) -- (1,3);
\draw (-1,0) -- (-1,3);
\node[below] at (0,0) {$0$};
\node[below] at (1,0) {$1$};
\node[below] at (-1,0) {$-1$};

\node at (2.5,1) {$\cdots$};

\draw[thick, blue] (5,0) arc (180:0:1/2);
\draw (5,0) arc (0:180:1/2);
\draw (6,0) -- (6,3);
\draw (4,0) -- (4,3);

\node[below] at (5,0) {\footnotesize$2a_1$};
\node[below] at (6,0) {\footnotesize$2a_1\!+\!1$};
\node[below] at (4,0) {\footnotesize$2a_1\!-\!1$};

%gamma geodesic
\draw (-1/2-.1,0) arc (180:0:3);
\node at (-1/2-.1,0) {$\bullet$};
\node[below] at (-1/2,-0.5) {$\gamma_{-\infty}$};
\draw[->] (-1/2,-0.5) -- (-1/2-.1,-0.1);
\node at (6-1/2-.1,0) {$\bullet$};
\node[below] at (6-1/2-.1,-0.5) {$\gamma_{\infty}$};
\draw[->] (6-1/2,-0.5) -- (6-1/2-.1,-0.1);

%excursion
\draw[thick, red] (1, 2.654) arc(117.83:9.2:3);
\draw[thick, red,->] (1, 2.654) arc(117.83:80:3);

\node at (1,2.654) {$\bullet$};
\node[above] at (1-.1,2.654) {$\xi$};
\node at (831/155,{6*sqrt(154)/155}) {$\bullet$};
\node[above] at (831/155+0.1,{6*sqrt(154)/155}) {$\eta$};
\end{tikzpicture}
\begin{tikzpicture}[scale=2.4, baseline={(0,-1)}]
\draw (-1.1,0)--(2.1,0);
\draw (-1,0) -- (-1,1.25);
\draw (0,0) arc (180:0:1/2); 
\draw (0,0) arc (0:180:1/2); 
\draw (1,0) -- (1,1.25);
%\draw (2,0) arc (180:135:1/2); 
\draw (2,0) arc (0:180:1/2); 
%\draw (3,0) -- (3, 1.5);
\node[below] at (0,0) {$0$};
\node[below] at (1,0) {$1$};
\node[below] at (-1,0) {$-1$};

%\draw (4,0) arc (180:0:1/2);
%\draw (4,0) arc (0:180:1/2);
%\draw (5,0) -- (5,3);
%\draw (3,0) -- (3,3);

\draw (1,0) arc(180:0:0.34);
\draw[thick, blue] (1.68,0) arc(0:180:0.125);
\draw (1.28,0) arc(180:0:0.075);
\draw (1,0) arc(180:0:0.14);
\node[below] at (1.25, 0) {$\frac{a_{1}+1}{a_{1}}$};
\node[below] at(1.45,-0.3) {$\frac{2a_{1}+1}{2a_{1}-1}$};
\draw[->] (1.45,-0.3) -- (1.45,-0.05);
\node[below] at (1.73+0.3,-0.3) {$\frac{a_{1}}{a_{1}-1}$};
\draw[->] (1.73+0.3,-0.3) -- (1.73,-0.05);

%\node[below] at (1+2/3,0) {$c$};

\node at (1.75,0.2) {$\dots$};
\node[below] at (2,0) {$2$};
%\node[below] at (3,0) {$3$}

%\gamma geodesic
\draw (-0.4,0) arc (180:0:1);
\node at (-0.4,0) {$\bullet$};
\node[below] at (-0.4, 0) {$\gamma_{-\infty}$};
\node at (2-0.4,0) {$\bullet$};
\node[below] at (2-0.4, 0) {$\gamma_{\infty}$};

%excursion
\draw[thick, red] (1, 0.916) arc(66.5:6.7:1);
\draw[thick, red,->] (1, 0.916) arc(66.5:40:1);

\node at (1,0.916) {$\bullet$};
\node[above] at (1+.05,0.9) {$\xi$};
\node at (1.595,0.113) {$\bullet$};
\node[above] at (1.55,0.115) {$\eta$};
\end{tikzpicture}
\caption{The first excursion of $\gamma$ when $\varepsilon_1=1$ and $s_{1}=e$ (left), and the first excursion of $\gamma$ when $\varepsilon_1=-1$ and $s_{1}=o$ (right).
The first excursion is the red arc from $\xi$ to $\eta$.
The blue arc is $\mathcal C_{(k,\veps)_s}$.
}
\label{fi:exc}
\end{figure}
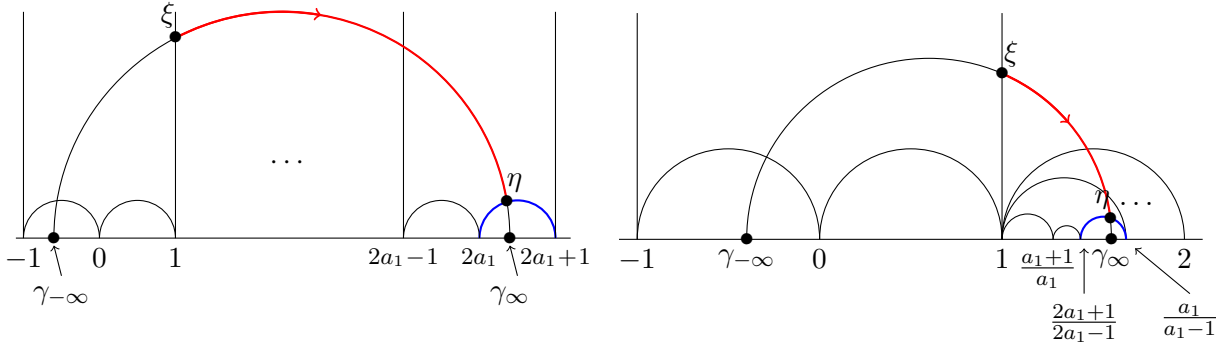

\begin{defn}\label{de:exc}
Let $\gamma\in A$, that is, $\gamma$ is a geodesic on $\bH$ satisfying $(\gamma_\infty,\gamma_{-\infty})\in\mathcal I$.
Let $(k,\veps)_s\in\mathscr A$ be a letter such that $1/|\gamma_\infty|\in I_{(k,\veps)_s}$.
We recall that $j_\gamma=\mathrm{sgn}(\gamma_\infty)$, and $\mathcal C_{(k,\veps)_s}$ and $\ell(\alpha,\beta)$ are as defined in the proof of Theorem~\ref{th:exc}.
Let 
$$\xi_\gamma:=\ell(j_\gamma,\infty)\cap \gamma\quad\text{ and }\quad \eta_\gamma:= j_\gamma\cdot\mathcal C_{(k,\veps)_s}\cap \gamma.$$
Let 
$$c(\gamma) := \text{the half-open geodesic segment of }\gamma\text{ from }\xi_\gamma\text{ to }\eta_\gamma, \text{ closed at $\xi_\gamma$ and open at $\eta_\gamma$}.$$

For $v\in T^1\M$,
there exists a unique lift $\tilde\gamma^v\in A$ of $\gamma^v$ such that $\tilde\gamma^v(0)$ lies in $c({\tilde\gamma^v})$.
\begin{enumerate}
\item 
We call $c({\tilde\gamma^v})$ \emph{the first excursion segment of $\gamma^v$}.
For each $n\ge 1$, we call $c(\rho^{n-1}(\tilde\gamma^v))$ \emph{the $n$-th excursion segment of $\gamma^v$}. See Figure~\ref{fi:exc}.
Note that 
$$\gamma^v|_{[-u(v),\infty)} = \bigsqcup_{n\ge 1}\pi(c(\rho^{n-1}(\tilde\gamma^v))),$$
where $u(v) := d_{\bH}(\xi_{\tilde\gamma^v}, \tilde\gamma^v(0))$.

\item \emph{The $n$-th excursion time} $r_n=r_n(v)$ is the hyperbolic length of the $n$-th excursion segment.
\item Let $T_n := T_n(v)=r_1+r_2+\dots+r_n$ and $T_0=0$. \emph{The $n$-th return time} is $T'_n:=T_n-u(v)$.
\end{enumerate}
\end{defn}

The suspension representation and the ergodicity of the return map now yield the almost-sure asymptotic growth rates of $T_N$ and $T_N'$.

\begin{thm}\label{th:exctime2}
For almost every $v\in T^1\M$,
$$\lim_{N\to \infty}\frac{T_N'}{N} = \lim_{N\to \infty}\frac{T_N}{N} = \frac{\pi^2}{2\log(2+\sqrt{3})}=:C^*.$$
\end{thm}

\begin{proof}
For each $v\in T^1\M$, $u(v)$ is a fixed constant independent of $N$, and hence
$$\lim_{N\to\infty}\frac{T'_N}{N} = \lim_{N\to \infty}\frac{T_N-u(v)}{N} = \lim_{N\to\infty}\frac{T_N}{N}$$
if the limit on the right-hand side exists.

A Liouville measure on $T^1\mathbb{H}$ is given by $\frac{dx\,dy\,d\theta}{y^2}.$
Here, in the coordinates $(x,y,\theta)$, the point $z=x+iy$ lies in a fundamental domain $\mathcal F$ for $\Theta$, and $\theta$ denotes the angle of the unit tangent vector based at $z$.
We now represent $v\in T^1\mathbb{H}$ by coordinates $(\alpha,\beta,t)$. Here, $\alpha$ and $\beta$ are respectively the backward and forward endpoints of the oriented geodesic determined by $v$. The point on this geodesic closest to a fixed base point is designated as the point at time $0$, and $t$ denotes the geodesic-flow time measured from this point. Under this change of coordinates, we obtain
$$\frac{dx\,dy\,d\theta}{y^2}=\frac{2\,d\alpha\,d\beta\,dt}{(\alpha-\beta)^2}.$$

The Liouville volume of $T^1\mathcal M$ is $\int_{T^1\mathcal F}\frac{dx\,dy\,d\theta}{y^2}=2\pi^2.$
Up to null sets, the suspension flow with roof function $r_1$ is naturally isomorphic, as a measure-preserving flow, to the geodesic flow on $T^1\mathcal M$. Therefore, by this isomorphism,
$$\int_{\Omega^\pm}r_1\,d\widetilde\mu=\int_{\Omega^\pm}\int_0^{r_1}dt\,d\widetilde\mu=\int_{T^1\mathcal F}\frac{d\alpha\,d\beta\,dt}{(\alpha-\beta)^2}=\pi^2.$$

The ergodicity of the geodesic flow $g_s$ on $T^1\M$ implies that its first return map $\Phi$ to $\pi(X)$ is ergodic with respect to $\widehat{\mu}:=(J\circ\varphi)^*\widetilde\mu$.
For $v\in T^1\M$, let $v_0 = \gamma^v(-u(v))$.
By the Birkhoff ergodic theorem and Theorem~\ref{th:com2}, for $m$-a.e. $v\in T^1\M$, 
$$\lim_{N\to \infty}\frac{T_N}{N} = \lim_{N\to \infty}\frac{\sum_{n=1}^N r_1(\Phi^{n-1}(v_0))}{N} = \frac{\int_{\pi(X)}r_1 d\widehat{\mu}}{\widehat\mu(\pi(X))} = \frac{\int_{\Omega^\pm}r_1 d\tilde\mu}{\tilde\mu(\Omega^\pm)} = \frac{\pi^2}{2\log(2+\sqrt{3})}.$$
Here, $\widetilde\mu(\Omega^\pm)=2\log(2+\sqrt{3})$, as computed in the proof of Lemma~\ref{le:Tbar-inv}.
\end{proof}

\section{Galambos' theorem} \label{sec:4}
In this section, by using the spectral gap of the transfer operator, we prove Galambos' theorem for the SCF, an extreme value theorem for the SCF digits.
Let $L^1 = L^1([0,1])$ denote the $L^1$-space on $[0,1]$ with respect to the Lebesgue measure whose $L^1$-norm is denoted by $\|\cdot\|_1$.
Likewise, $L^\infty = L^\infty([0,1])$ denotes the $L^\infty$-space on $[0,1]$ equipped with the essential supremum norm $\|\cdot\|_\infty$ with respect to the Lebesgue measure.

\begin{defn}[Transfer Operator]
The \emph{transfer operator} $\TF$ of the spliced continued fraction map $\CF$ is the operator on $L^1$ uniquely determined by the relation
\[
\int f(x) (g\circ T)(x) dx = \int \TF f(x) g(x) dx
\]
for $f \in L^1$ and $g \in L^{\infty}$.
Using the inverse branches $h_{(a,\veps)_s}$ of $T$ onto $I_{(a,\veps)_s}$, we express the transfer operator as follows. For $f \in L^1$,
\begin{align}\label{TFdefn}
    \TF f (x) =& \sum_{(a, \veps)_{s} \in \mathscr{A}} \left| (h_{(a, \veps)_s})'(x)\right| \left(f\circ h_{(a, \veps)_s}\right)(x).
\end{align}
More explicitly,   \begin{align*}
    \TF f (x)  =& \frac{1}{(x+2)^{2}}f\left(\frac{1}{x+2}\right)+\sum_{k=2}^{\infty} 
    \left[ \frac{1}{(x+2k)^{2}} f\left(\frac{1}{x+2k}\right) + \frac{1}{(2k-x)^{2}} f\left(\frac{1}{2k-x}\right) \right.\\
    +
    &\frac{1}{(k+(k+1)x)^{2}} f\left(\frac{kx+(k-1)}{k+(k+1)x}\right)
    +
    \left.\frac{1}{(kx+(k+1))^{2}} f\left(\frac{k+(k-1)x}{kx+(k+1)}\right)
    \right].
\end{align*}
\end{defn}

Let $\lambda$ be an absolutely continuous probability measure on $[0,1]$.
For $f \in L^\infty$, the essential variation $v(f)$ is defined by
\[
v(f)
= \lim_{a \to 0^+} \frac{1}{a} \int_0^1 |f(u+a)-f(u)|\,du,
\]
where $f(x)=f(1)$ for $x>1$, see \cite[\S 2.0.2]{IK13} for definition.

If $v(f)<\infty$, then $f$ is said to have bounded essential variation. 
The space $BV_\lambda$ of functions with bounded essential variation contains the space of classical bounded variation functions, and it is a commutative Banach algebra when endowed with the norm
\[
\|f\|_{v,\lambda} := v(f) + \|f\|_{1,\lambda},
\]
where $\|\cdot\|_{1, \lambda}$ is the $L^1$-norm with respect to $\lambda$. 
When $\lambda = \mathrm{Leb}$, we write $BV = BV_{\mathrm{Leb}}$ and $\|\cdot\|_v = \|\cdot\|_{v, \mathrm{Leb}}$.

The transfer operator $\TF$ defines a bounded operator on $BV$. 
Moreover, we shall verify that $\TF : BV \to BV$ has a spectral gap.
The measure $\mu$ here and below is the same invariant probability measure as in Corollary~\ref{abscontinv}; its density $f_\mu$ is an eigenfunction of $\TF$ with eigenvalue $1$.
By the spectral gap, it is the eigenfunction belonging to $1$. 

For the brevity of notations, for $A^{(n)}=(A_1,A_2,\dots,A_n)\in \mathscr A^n$, let $h_{A^{(n)}}:=h_{A_1}\circ\cdots\circ h_{A_n}$ denote the composition of inverse branches and $I_{A^{(n)}} := h_{A^{(n)}}([0,1])$ the image by the composition of inverse branches.   

\begin{thm}[Spectral gap]
    The transfer operator $\TF : (L^1,\|\cdot\|_1) \to (L^1,\|\cdot\|_1)$ has a unique eigenvalue $1$ of maximal modulus which is a simple eigenvalue.
    Define operators $\Pi$ and $R$ on $L^1$ by $$\Pi g(x) = f_\mu(x)\int_0^1 g(t) dt\qquad\text{and}\qquad R=\TF-\Pi.$$
    Then $R(BV) \subset BV$, $\|\Pi\|_1 \le 1$, $\|R^n\|_v =O(\theta^n)$ as $n \to \infty$ for some $0<\theta<1$, and $\TF^n = \Pi + R^n$.
\end{thm}

\begin{proof}
The proof will use Theorem~5.3.12 and Proposition~5.3.14 in \cite{IG90}.

The derivatives of the inverse branches $h_{(a, \veps)_s}$ of $T$ are
$$
h_{(a, \veps)_e}'(x)= \frac{-\veps}{(2a + \veps x)^2},\quad
h_{(a, -1)_o}'(x) =\frac{1}{((a+1)x+a)^2},\quad
h_{(a, +1)_o}'(x) =\frac{-1}{(ax+(a+1))^2},
$$
hence we have
\[
\frac{1}{(2a+1)^2}\le|h_{(a,\veps)_e}'(x)| \le \frac{1}{(2a-1)^2},\quad
\frac{1}{(2a+1)^2}\le|h_{(a,\veps)_o}'(x)|\le \frac{1}{a^2}.
\]

For $A^{(m)} = (A_1, \dots, A_m) \in \mathscr A^{m}$, 
\begin{equation}\label{eq:h'An}
h_{A^{(m)}}'=(h_{A_1}'\circ h_{A_2,\dots,A_m})\cdot(h_{A_2}'\circ h_{A_3,\dots,A_m})\cdots (h_{A_m}').
\end{equation}
Since $|h_{(a,\veps)_s}'(x)| \le 1/4$ for any $(a, \veps)_s$, we have for every $m\ge1$, 
\begin{equation}
    \sup_{A^{(m)}\in\mathscr A^m}\underset{t\in[0,1]}{\mathrm{ess~sup}}|h'_{A^{(m)}}(t)|\le \frac{1}{4^m}<1,\tag{$\mathrm{E}_m$}
\end{equation}

The second order derivatives of the inverse branches are
\[
h_{(a, \veps)_e}''(x)= \frac{2}{(2a + \veps x)^3},\quad
h_{(a, -1)_o}''(x) =\frac{-2(a+1)}{((a+1)x+a)^3}\quad\text{and}\quad
h_{(a, +1)_o}''(x) =\frac{2a}{(ax+(a+1))^3}.
\]
Here and below, $\operatorname{var}(f)$ denotes the classical total variation; it is distinct from the essential variation $v(f)$, although $v(f)\le\operatorname{var}(f)$. For a $C^1$ function, $\operatorname{var}(f)=\int_0^1|f'|$. Thus $\operatorname{var}(h'_{(a,\veps)_s})\le\|h''_{(a,\veps)_s}\|_\infty$, and since $|h_{(a,\veps)_s}''| \le 4/a^2$ for every $(a,\veps)_s$, $\TF$ satisfies the condition 
\begin{equation}\label{eq:bdddist}
    \sum_{(a,\veps)_s\in\mathcal A}\mathrm{var}\;h'_{(a,\veps)_s} <\infty.\tag{BV}
\end{equation}

One can verify that for every $(a, \veps)_{s} \in \mathscr{A}$, the distortion bound $|h_{(a, \veps)_{s}}''(x)| \le 3|h_{(a, \veps)_{s}}'(x)|$ holds.
More generally, for $A^{(n)} \in \mathscr A^n$, by \eqref{eq:h'An}
we have 
\begin{align*}
    \frac{h_{A^{(n)}}''}{h_{A^{(n)}}'}
    =\frac{h_{A_1}''\circ h_{A_2,\dots,A_n}}{h_{A_1}'\circ h_{A_2,\dots,A_n}} h_{A_2,\dots,A_n}'
    +\frac{h_{A_2}''\circ h_{A_3,\dots,A_n}}{h_{A_2}'\circ h_{A_3,\dots,A_n}} h_{A_3,\dots,A_n}'
    +\cdots + \frac{h_{A_{n-1}}''\circ h_{A_n}}{h_{A_{n-1}}'\circ h_{A_n}} h_{A_n}' + \frac{h_{A_n}''}{h_{A_n}'}.
\end{align*}
Using the above distortion bound for each branch, we obtain the estimate $$|h''_{A^{(n)}}|/|h'_{A^{(n)}}|\lesssim |h'_{A_2,\dots,A_n}|+|h'_{A_{3},\dots,A_n}|+\dots+|h'_{A_n}|+1<\sum_i 4^{-i}$$
which leads to the lemma stated below.
\begin{lem}[Bounded distortion]\label{le:bdd.dist}
There exists a constant $C>0$ such that for any $n$ and for every $A^{(n)}\in\mathscr A^n$,
$$\underset{t\in[0,1]}{\mathrm{ess~sup}}\frac{|h_{A^{(n)}}''(t)|}{|h_{A^{(n)}}'(t)|}< C.$$
\end{lem}

According to \cite[Proposition~5.3.4]{IG90}, for a $C^{1}$ piecewise monotonic transformation whose branch derivatives are absolutely continuous on $[0,1]$, Condition~(E$_m$) for some $m \in \mathbb{N}^*$ together with the bounded distortion property in Lemma~\ref{le:bdd.dist} imply
\begin{equation}\label{eq:C}
\frac{\operatorname*{ess\,sup}_{t\in [0,1]}\bigl|h'_{A^{(n)}}(t)\bigr|}
     {\operatorname*{ess\,inf}_{t\in [0,1]}\bigl|h'_{A^{(n)}}(t)\bigr|}
\le C,
\qquad
A^{(n)} \in \mathscr{A}^n,\; n \in \mathbb{N}. \tag{C}
\end{equation}
It follows from Theorem~5.3.12 and Proposition~5.3.14 in~\cite{IG90} that for a $C^{1}$ piecewise monotonic transformation, if conditions~(BV), (E$_m$) for some $m \ge 1$, and~(C) are satisfied, then the associated transfer operator has the desired spectral gap property in the statement.
Since these conditions have now been verified for $T$, this completes the proof.
\end{proof}

Denote the \emph{normalized transfer operator} by $\widetilde \TF$, which is defined by 
$$\widetilde \TF g := \frac{1}{f_\mu}\TF(f_\mu g) = \frac{1}{f_\mu}\sum_{A\in\mathscr A}|h_A'|(f_\mu\circ h_A)(g\circ h_A).$$
The $n$-th iteration of $\widetilde\TF$ is
\begin{equation}\label{eq:nth.op}
\widetilde\TF^n g = \frac{1}{f_\mu}\sum_{A^{(n)}\in\mathscr A^n}|h_{A^{(n)}}'|(f_\mu\circ h_{A^{(n)}})(g\circ h_{A^{(n)}}).
\end{equation}

\begin{lem}\label{le:v.mu}
For $A^{(n)}\in\mathscr A^n$, 
$$v(\widetilde\TF^n 1_{I_{A^{(n)}}}) = O(\mu(I_{A^{(n)}})).$$
\end{lem}

\begin{proof}
Since there exists $c>1$ such that $1/c<f_\mu<c$, $\mathrm{Leb}(I_{A^{(n)}})\asymp \mu(I_{A^{(n)}})$.
The interval $I_{A^{(n)}}$ is the image of $h_{A^{(n)}}$ of $[0,1]$ and $h_{A^{(n)}}$ is monotone.
Thus, 
$$\mathrm{Leb}(I_{A^{(n)}})=|h_{A^{(n)}}(1)-h_{A^{(n)}}(0)|=|h_{A^{(n)}}'(t^*)|\quad\text{ for some }t^*\in(0,1).$$

From \eqref{eq:nth.op},
$$\widetilde\TF^n 1_{I_{A^{(n)}}}= \frac{|h_{A^{(n)}}'|f_\mu\circ h_{A^{(n)}}}{f_\mu}=:g.$$
Since $g$ is differentiable, its total variation is $\int |g'|dx$.
Note that $g' = g(\log g)'$.
We have
\begin{align*}
(\log g)' 
& = \frac{h_{A^{(n)}}''}{h_{A^{(n)}}'}+\frac{f_\mu'\circ h_{A^{(n)}}}{f_\mu \circ h_{A^{(n)}}}h_{A^{(n)}}' - \frac{f_\mu'}{f_\mu}.
\end{align*}
From bounded distortion property in Lemma~\ref{le:bdd.dist} and the estimate $|f_\mu'/f_\mu|\le c'$ for some constant $c'$, we have 
$$|(\log g)'|\lesssim 1+|h_{A^{(n)}}'|.$$

Since $\esssup |g|\lesssim \esssup |h_{A^{(n)}}'|\le 1$, we have
\begin{align*}
v(g)  \le \mathrm{var}(g) = \int |g(\log g)'|dx 
 \lesssim \esssup |g|(1+\operatorname*{ess\,sup} |h_{A^{(n)}}'|) \lesssim \operatorname*{ess\,sup} |h_{A^{(n)}}'|.
\end{align*}
From \eqref{eq:C}, 
$$v(g)\lesssim \esssup|h_{A^{(n)}}'|\lesssim \essinf|h_{A^{(n)}}'|\le \mathrm{Leb}(I_{A^{(n)}})\asymp \mu(I_{A^{(n)}}).$$
\end{proof}

Following the argument of \cite[Proposition~2.1.7]{IK13}, or equivalently the standard argument combining the Doeblin--Fortet inequality above with the compact inclusion $BV\hookrightarrow L^1$, we obtain the following spectral property of $\widetilde \TF$.
\begin{prop} \label{spectralgap}
    The transfer operator $\widetilde\TF : (L^1, \|\cdot\|_1) \to (L^1, \|\cdot\|_1)$ is a bounded operator with spectral radius $1$. 
    Its only eigenvalue of modulus $1$ is $1$, and this eigenvalue is simple.
    In particular, there exists $0<\theta<1$ such that for all $f\in BV_\mu$ and all $n>0$,
    \begin{align*}
        \widetilde \TF^{n}f = \int f d\mu + \widetilde R^{n}(f),
    \end{align*}
where $\|\widetilde R^n\|_{v,\mu} = O(\theta^n)$.
\end{prop}
The spectral gap of $\TF$ implies that the invariant measure $\mu$ given in Corollary~\ref{abscontinv} is exponentially mixing and is the unique absolutely continuous invariant probability measure. 

In what follows, we use these properties to establish Theorem~\ref{thm:1}: Galambos’ theorem for the digits $a_n$.

Due to the inclusion-exclusion principle, we can describe the measure on the left-hand side in Theorem~\ref{thm:1}:
For  $C_0>0$ and $y>0$,
\begin{equation}\label{Inclusionexclusion}
 \mu\{x:\max_{1\le n\le N} a_n(x)\le C_0 Ny\}
= 1+\sum_{k=1}^N(-1)^{k}\sum_{1\le i_1<\cdots<i_k\le N}\mu\{x:\min_{1\le j\le k}a_{i_j}(x)>C_0 Ny\}.
\end{equation}

Denote by $S_{N}$ the set of $x\in[0,1]$ such that $a_1(x)> N$.
We have
$$S_{N} = \left[0,\frac{1}{2N+1}\right]\cup\left[\frac{N}{N+1},1\right].$$
The set $\{x:\min_{1\le j\le k}a_{i_j}(x)>C_0Ny\}$ is the intersection of preimages of $S_{\lfloor C_0Ny\rfloor}=\{x:a_{1}(x)>C_0Ny\}$ under the SCF map $T$.
Since $\mu$ is $\CF$-invariant, information on the mass of $S_N$ and on the (quasi-)independence of the associated events is required.
From now on, set 
$$C_0:=\frac{2}{\log (2+\sqrt{3})}.$$

\begin{lem} \label{le:S_N}
For an integer $N\ge 1$, the $\mu$-measure of $S_{N}$ is
$$\mu(S_{N}) 
=C_0\log \frac{2N+(\sqrt{3}+1)}{2N+(\sqrt{3}-1)}
=C_0\log\left(1+\frac{1}{N+\frac{\sqrt{3}-1}{2}}\right).$$
\end{lem}

\begin{proof}
By the change of variable $t=\frac{1}{x}$, 
we have
\begin{align*}
 \int_a^b d\mu =\frac{1}{\log(2+\sqrt{3})} \int_{1/b}^{1/a}\frac{1}{t-(2-\sqrt{3})}-\frac{1}{t+\sqrt{3}}dt  =\frac{C_0}{2} \left.\log\frac{t-(2-\sqrt{3})}{t+\sqrt{3}}\right]^{1/a}_{1/b}
\end{align*}
Applying the above identity with
\((a,b)=\bigl(0,\tfrac{1}{2N+1}\bigr)\) and
\((a,b)=\bigl(\tfrac{N}{N+1},1\bigr)\), respectively, yields
$$
\mu(S_N)= \int_0^{\frac{1}{2N+1}}d\mu +\int_{\frac{N}{N+1}}^1d\mu=C_0\log\frac{2N+(\sqrt{3}+1)}{2N+(\sqrt{3}-1)},
$$
which completes the proof.
\end{proof}

Next, we control correlations of the events ${a_n \ge L}$ using the spectral gap.

\begin{lem}\label{le:bdd}
There exists $\theta \in (0, 1)$ such that for $k\ge2$, $L\ge 1$ and $i_{1} < \cdots<i_{k}$, 
\begin{align} \label{eq5.3}
    \mu\left\{\min_{1\le j \le k} a_{i_{j}} > L \right\} 
    = \prod_{j=1}^{k-1}\left( 1 + O(\theta^{i_{j+1}-i_{j}}) \right)\prod_{j=1}^{k} \mu\left\{a_{i_{j}}> L\right\}.
\end{align}
The implied constants are independent of $L$, $k$, and $i_1,\ldots,i_k$.
\end{lem}
\begin{proof}
    For $N>1$, $\left\{\min_{1\le j \le k} a_{i_{j}} > N \right\} = \bigcap_{j=1}^{k} \CF^{-i_{j}+1}S_{N}$.
    Put $N = \lfloor L \rfloor$.

We write $\|f\|_{L^1(\mu)}:=\int |f|d\mu$.
By the H\"older inequality and the inequality $\|g\|_\infty\le \|g\|_{v,\mu}$ (see \cite[Proposition 2.0.1 (ii)]{IK13}), we have
$$\left|\int f\cdot g d\mu\right| \le \|f\|_{L^1(\mu)}\|g\|_\infty\le \|f\|_{L^1(\mu)}\|g\|_{v,\mu} \qquad\text{ for }f\in {L^1(\mu)},\; g\in BV_\mu.$$
From Proposition~\ref{spectralgap}, 
$$\left|\int f \cdot \widetilde{\TF}^n g d\mu - \int fd\mu \int g d\mu\right| = \left|\int f \cdot \widetilde R^n g d\mu\right| 
\le \|f\|_{L^1(\mu)}\|\widetilde R^n g\|_{v,\mu}
\le \|f\|_{L^1(\mu)}\|g\|_{v,\mu}O(\theta^n).$$
For a Borel set $U\subset [0,1]$, let $f = 1_U$ and $g = \widetilde \TF 1_{S_N}$.
Applying the above inequality, we have
\begin{equation}\label{eq:mixing1}|\mu(S_N\cap T^{-n}U)-\mu(U)\mu(S_N)| = \mu(U)\|\widetilde\TF 1_{S_N}\|_{v,\mu}O(\theta^{n}).
\end{equation}
From Lemma~\ref{le:v.mu}, $\|\widetilde\TF 1_{S_N}\|_{v,\mu} = O(\mu(S_N))$. 
Combining this with \eqref{eq:mixing1},
we have
$$\mu(S_N\cap T^{-n}U) = \mu(U)\mu(S_N)(1+O(\theta^n)).$$

Applying the estimate with $n=i_2-i_1$ and $U = \CF^{-i_{3}+i_{2}}S_{N} \cap \cdots \cap \CF^{-i_{k}+i_{2}} S_{N}$, we obtain
    \begin{equation}\label{eq:S_ind}\begin{split}
        \mu\left(\bigcap_{j=1}^{k} \CF^{-i_{j}+1}S_{N}\right)
        =&\mu\left(S_{N}\cap \bigcap_{j=2}^k \CF^{-i_{j}+i_{1}}S_{N}\right)
        = \mu \Bigg(S_N \cap T^{-i_2+i_1}\Big(S_N\cap  \bigcap_{j=3}^k T^{-i_j+i_2}S_N\Big)\Bigg)\\
        =&\mu(S_{N} \cap \CF^{-i_{3}+i_{2}}S_{N}\cap\cdots \cap \CF^{-i_{k}+i_{2}} S_{N})\mu(S_N)\left( 1+ O\left(\theta^{i_{2}-i_{1}}\right)\right)%\\
    \end{split}\end{equation}
Repeating this argument inductively yields \eqref{eq5.3}.
\end{proof}

\begin{proof}[Proof of Theorem~\ref{thm:1}]
We follow the strategy in \cite{Pol09}.
Let $y>0.$
Near $x=0$, $\log(1+x)=x+ O(x^2)$.
Thus,
$$C_0Ny \log\left(1+\frac{1}{C_0Ny+\frac{\sqrt{3}-1}{2}}\right) = \frac{1}{1+\frac{\sqrt{3}-1}{2C_0Ny}}+O\left(\frac{1}{N}\right) = 1+O\left(\frac{1}{N}\right).$$
Combining this with Lemma~\ref{le:S_N} and Lemma~\ref{le:bdd}, we have
\begin{align}\label{eq:minbdd}
    \mu\{x:\min_{1\le j\le k}a_{i_j}(x)>C_0Ny\} 
    = \prod_{j=1}^{k-1}\left( 1 + O(\theta^{i_{j+1}-i_{j}}) \right)
    \frac{1}{(Ny)^k}\left(1+O\left(\frac{1}{N}\right)\right)^{k}.
\end{align}

Take $M<N$. We first consider the summands in Equation~\eqref{Inclusionexclusion} with a `coarse partition' of $[1,N]$:
\begin{equation}\label{eq:coa.part}
1\le i_1<\cdots<i_k\le N\quad\text{ such that }\quad i_{j+1}-i_j\ge m
\end{equation}
with large enough $m$ and $1\le k\le M$.
For such $i_j$, from \eqref{eq:minbdd},
\begin{align} \label{eq:est1} 
    \mu\{x:\min_{1\le j\le k}a_{i_j}(x)>C_0Ny\}=
   \frac{1}{(Ny)^k}\left(1+O\left(\frac{1}{N}\right)\right)^{k} (1+ O(\theta^{m}))^{k-1}.
\end{align}
The number of tuples $(i_1,\cdots,i_k)$ satisfying \eqref{eq:coa.part} is
${{N-(m-1)(k-1)}\choose k}$.
Thus, the sum of \eqref{eq:est1} over the partitions $(i_1,\cdots,i_k)$ as in \eqref{eq:coa.part} equals
\begin{align*}
    \frac{{{N-(m-1)(k-1)}\choose{k}}}{(Ny)^{k}}\left(1 + O\left(\frac{1}{N}\right)\right)^{k}(1+ O(\theta^{m}))^{k-1}
    =\frac{1}{k!y^k}\left( 1+ O_{m,k}\left(\frac{1}{N}\right)\right)^{2k}(1+O(\theta^{m}))^{k-1}.
\end{align*}
Here, $1-C\theta^{m} \le 1+ O(\theta^{m}) \le 1+ C\theta^{m}$ for every $m$ due to (\ref{eq5.3}).
Taking the sum over all $1\le k\le M$, 
the right-hand side of \eqref{Inclusionexclusion} is %we obtain
\begin{align*}
    1+\sum_{k=1}^{M}\frac{(-1)^{k}}{k!y^k}\left( 1+ O_{m,k}\left(\frac{1}{N}\right)\right)^{2k}(1+O(\theta^{m}))^{k-1},
\end{align*}
which converges to $$\exp\left(-1/y\right)$$ by letting $N$ tend to $\infty$, and then letting $M$ and $m$ tend to infinity.

Therefore, it suffices to show that the sum of the remaining terms vanishes as $N \to \infty$ and $M \to\infty$.
We divide it into two cases.

(1) First, we consider $M < k \le N$.
From \eqref{eq:minbdd}, 
$$\sum_{1\le i_1<\cdots<i_k\le N}\mu\{x:\min_{1\le j\le k}a_{i_j}(x)>C_0Ny\}\le \begin{pmatrix}N\\k\end{pmatrix} \frac{( 1 + C_1 )^{k-1}}{(Ny)^k}\left(1+O\left(\frac{1}{N}\right)\right)^k\le \frac{1}{k!}\left(\frac{C'}{y}\right)^k,$$
where $C'$ is a uniform upper bound for $(1+C_1)(1+O(1/N))$.
The sum of the right-hand side over $M<k\le N$ converges when $N\to\infty$. 
By taking sufficiently large $M$, the sum can be arbitrarily small.

(2) In the other case, $1\le k \le M$ but $i_{j+1}-i_{j}<m$ for some $1\le j \le k-1$, the sum is bounded by 
\begin{align*}
    \frac{(1+C_1)^{k-1}}{(Ny)^{k}}\left({N\choose k}-{{N-(m-1)(k-1)}\choose k}\right)\left(1+O\left(\frac{1}{N}\right)\right)^{k}
\end{align*}
Since $\frac{1}{N^k}\left({N\choose k}-{{N-(m-1)(k-1)}\choose k}\right) = O_{m,k}(\frac1N)(1+O_{m,k}(\frac1N))^{k-1}$,
the sum vanishes as $N \to \infty$.
\end{proof}

After establishing the extreme value theorem for regular continued fractions with respect to the Gauss measure \cite{Gal72}, Galambos extended the result to all absolutely continuous measures \cite{Gal73}.
Following a similar approach, we prove the extreme value theorem in our setting for arbitrary absolutely continuous measures as well.

\begin{lem}
For $M\ge 1$ and $N\ge 1$,
let $U_{M}:= \left\{\max_{1\le i\le M}a_i\le C_0My\right\}$ and $V_{i}^{N} := \{a_i> C_0Ny\}$.
There are constants $C_2>0$ and $\theta\in (0,1)$ such that for $k\ge 2$ and $M = i_0< i_1<\dots<i_k$, 
$$\mu(U_M\cap V_{i_1}^N\cap V_{i_2}^N\cap \dots \cap V_{i_k}^N)\le \mu(U_M)\prod_{j=0}^{k-1}(1+C_2\theta^{i_{j+1}-i_j})\prod_{j=1}^k\mu(V_{i_j}^N).$$
\end{lem}

\begin{proof}
Let $g = 1_{U_M}$.
We denote by 
$$A^{(M)} = ((a_1,\veps_1)_{s_1},\dots,(a_M,\veps_M)_{s_M})\in\mathscr A^M.$$
We have 
$$U_M = \bigcap_{i=1}^M\{a_i\le C_0My\} = \bigcap_{i=1}^M \bigcup_{a_i\le C_0My} T^{-(i-1)}I_{(a_i,\veps)_{s}} = \bigcup_{A^{(M)}\in\mathscr A^M\text{ s.t. }a_i\le C_0My,\;\forall i}I_{A^{(M)}}.$$
From \eqref{eq:nth.op},
$$\widetilde\TF^M g = \frac{1}{f_\mu}\sum_{a_i\le C_0My,\;\forall i}|h_{A^{(M)}}'|(f_\mu\circ h_{A^{(M)}}).$$
From Lemma~\ref{le:v.mu}, $v(\widetilde \TF^M g) = O(\mu(U_M))$.

By a similar argument in \eqref{eq:mixing1}, for $n>M$ and for a Borel set $U'\subset[0,1]$, we have
$$\mu(U_M\cap T^{-n}U') = \mu(U_M)\mu(U')(1+O(\theta^{n-M})).$$
By applying the estimate with $n=i_1-1$ and $U' = S_{\lfloor C_0Ny\rfloor}\cap T^{-i_2+i_1}S_{\lfloor C_0Ny\rfloor}\cap\cdots\cap T^{-i_k+i_1}S_{\lfloor C_0Ny\rfloor}$, we obtain
$$\mu(U_M\cap V_{i_1}^N\cap V_{i_2}^N\cap \dots \cap V_{i_k}^N) = \mu(U_M)\mu\Big(S_{\lfloor C_0Ny\rfloor}\cap \bigcap_{j=2}^k T^{-i_j+i_1}S_{\lfloor C_0Ny\rfloor}\Big)(1+O(\theta^{i_1-M-1})).$$
From Lemma~\ref{le:bdd}, we have the conclusion. 
\end{proof}

\begin{coro}\label{co:Ren}
%Let $L_1 = C_0Ny$ and $L_2=C_0My$.
For any $M\ge 1$, we have
$$\lim_{N\to \infty}\mu(U_N\cap U_M) = \exp\left(-\frac{1}{y}\right)\mu(U_M).$$
\end{coro}

\begin{proof}
By the inclusion-exclusion principle,
$$\mu(U_N\cap U_M) = \mu(U_M) + \sum_{k=1}^{N-M}(-1)^k\sum_{M=i_0<i_1<\dots<i_k\le N}\mu(U_M\cap V_{i_1}^N\cap\dots\cap V_{i_k}^N).$$
From a similar argument to the proof of Theorem~\ref{thm:1}, we have the conclusion.
More precisely,
first, we have
$$\mu(U_M\cap V_{i_1}^N\cap\dots\cap V_{i_k}^N)\le \mu(U_M)\prod_{j=0}^{k-1}(1+C_2\theta^{i_{j+1}-i_j})\frac{1}{(Ny)^k}\left(1+O\left(\frac{1}{N}\right)\right)^k.$$
Thus, for $M=i_0<i_1<\dots<i_k\le N$ such that $i_{j+1}-i_j\ge m$ for all $0\le j\le k-1$, we have
$$\mu(U_M\cap V_{i_1}^N\cap\dots\cap V_{i_k}^N) = \mu(U_M)\frac{1}{(Ny)^k}\left(1+O\left(\frac{1}{N}\right)\right)^k(1+O(\theta^m))^{k},$$
On the other hand, for $M=i_0<i_1<\dots<i_k\le N$ such that $i_{j+1}-i_j< m$ for some $0\le j\le k-1$, 
the sum $\mu(U_M\cap V_{i_1}^N\cap\dots\cap V_{i_k}^N)$ over such tuples $(i_1,\dots,i_k)$ is a product of $\mu(U_M)$ and a vanishing factor.
Therefore, we have the conclusion.
\end{proof}

\begin{thm}\label{th:Gal2}
For an absolutely continuous probability measure $\mu'$ on $(0,1)$ and for all $y>0$, we have
$$\lim_{N\to \infty}\mu'\left\{x\in(0,1):\max_{1\le n\le N}a_n(x)\le C_0 N y\right\}=\exp\left(-\frac{1}{y}\right).$$
\end{thm}

\begin{proof}
The theorem follows from the main theorems in \cite{Ren63}: According to Theorems~3 and 4 of \cite{Ren63}, Corollary~\ref{co:Ren} ensures that for any $y$, our sequence of events 
$(U_M)_{M=1}^\infty$ is stable and mixing, i.e.,
\begin{equation}
    \mu(U_{M}\cap B) \to \exp\left(-\frac{1}{y}\right)\mu(B)
\end{equation}
for any measurable set $B$, as $M\to \infty$.
By Theorem~6 of \cite{Ren63}, for any probability measure $\mu'$ absolutely continuous with respect to $\mu$, the sequence $U_{M}$ is stable. Moreover, for any measurable $B$,
\begin{equation}
    \mu'(U_{M}\cap B) \to \exp\left(-\frac{1}{y}\right) \mu'(B),
\end{equation}
hence the theorem follows.
\end{proof}

\section{Extreme Value Theorem}\label{sec:5}

For almost every $\alpha\in[1,\infty)$, the partial quotients $a_n(1/\alpha)$ of the SCF are unbounded.
For $v\in T^1\M$, we consider the geodesic $\gamma^v$ determined by $v$, whose lift $\tilde\gamma^v\in A$ satisfies $\limsup_{n\to\infty}a_n(|1/\tilde\gamma^v_\infty|) = \infty$.
One can choose a unique lift $\tilde\gamma^v\in A$ such that the lift of $v$ lies on the first excursion segment defined in Definition~\ref{de:exc}-(1).

\begin{defn}
Let $\Gamma$ be a Fuchsian group such that $\mathcal S=\Gamma\backslash \bH$ has at least one cusp $\xi_0=\Gamma.\xi$ for $\xi\in\partial \bH$.
Let $\Gamma_\xi$ be the stabilizer of $\xi$ in $\Gamma$.
If there exists a unique maximal open horoball $H_\xi$ centered at $\xi$ such that $\Gamma_\xi\backslash H_\xi$ embeds in $\mathcal S$,
we call $\Gamma\backslash H_\xi$ \emph{the maximal Margulis neighborhood} of the cusp $\xi_0$. See \cite{HP02} for details.

For a cusp $\xi$ and $p\in \mathcal S$, we define the height of $p$ with respect to $\xi$ by a signed distance from the maximal Margulis neighborhood of $\xi$:
$$\mathrm{ht}_\xi(p): = 
\begin{cases}
d_{\mathcal S}(p,\partial(\Gamma\backslash H_\xi))&\text{if }\;\;p\in \Gamma\backslash H_{\xi},\\
-d_{\mathcal S}(p,\partial(\Gamma\backslash H_\xi))&\text{if }\;\;p\not\in\Gamma\backslash H_{\xi}.\end{cases}$$

The \emph{height} $\mathrm{ht}(p)$ for $p\in\mathcal S$ is defined to be the maximum of the signed distance between $p$ and the boundary of $\Gamma\backslash H_\xi$ over all the cusps $\xi$, i.e., 
$$\mathrm{ht}(p) = \max_{\xi:\text{cusps}}\mathrm{ht}_\xi(p).$$
\end{defn}

Let us consider two cusps $\Theta.\infty$ and $\Theta.1$ of $\M$.
The stabilizer $\Theta_\infty$ of $\infty$ and the stabilizer $\Theta_1$ of $1$ in $\Theta$ are  
$$\Theta_\infty=\left\langle\begin{pmatrix}1&2\\0&1\end{pmatrix}\right\rangle\qquad\text{and}\qquad\Theta_1=\left\langle\begin{pmatrix}2&-1\\1&0\end{pmatrix}\right\rangle.$$
Let
$$
H_\infty := \{z\in\bH: \mathrm{Im}(z)> 1\}
\qquad\text{and}\qquad
H_1 := \{z\in\bH: |z-(1+i)|<1\}.
$$
We can easily check that $\Theta\backslash H_\infty$ and $\Theta\backslash H_1$ are the maximal Margulis neighborhoods of $\Theta.\infty$ and $\Theta.1$, respectively.

We take a fundamental domain for $\Theta$ to be $\{z\in\bH:|z|\ge1,\;|z-2|>1,\;0\le \mathrm{Re}(z)<2\}$, and divide it into the following two regions corresponding to the two cusps:
\begin{align*}
& N_\infty=\{z\in\bH: |z-1|\ge \sqrt{2},\; 0\le\mathrm{Re}(z)<2\}, \;\text{and}\\
& N_1 =\{z\in\bH: |z-1|<\sqrt{2}, \;|z|\ge1,\; |z-2|>1, \; 0\le \mathrm{Re}(z)<2\}.
\end{align*}
Then
$$\mathrm{ht}(p) = \begin{cases}\mathrm{ht}_\infty(p) &\text{if }p\in \pi(N_\infty),\\
\mathrm{ht}_1(p) & \text{if }p\in \pi(N_1).\end{cases}$$

To present an asymptotic relation between the maximal height and the maximal distance up to the $n$-th excursion, we shall introduce some notions and notations.

Let $u$ be a point on the $n$-th excursion of $\gamma^v$. 
Then $u$ is a point on $c(\rho^{n-1}\tilde\gamma^v)$ defined in Definition~\ref{de:exc}.
Let $p=\pi(u)$.
If $s_n=e$ and $u\in N_\infty$, then $\mathrm{ht}(p)$ is the distance between $u$ and $\partial H_\infty$.
On the other hand, if $s_n=o$ and $u\in N_1$, then $\mathrm{ht}(p)$ is the distance between $u$ and $\partial H_1$, which equals the distance between $\tilde\iota(u)$ and $\tilde\iota(\partial H_1) = \partial H_\infty.$

We consider orientation-reversing involution maps $\iota$ and $\widetilde{\iota}$ that exchange the roles of $e$-type and $o$-type, given by 
\begin{equation}\label{eq:transf}
\iota(z) = \frac{1-\bar z}{1+\bar z}\qquad\text{ and }\qquad\widetilde{\iota}(z) = \frac{\bar{z}+1}{\bar{z}-1}.
\end{equation}
Note that $\iota$ is a reflection across the circle $|z+1|=\sqrt{2}$, and $\tilde{\iota}$ is a reflection across the circle $|z-1|=\sqrt{2}$.
It follows from the fact that $\iota$ and $\widetilde{\iota}$ exchange $\Theta.\infty$ and $\Theta.1$.
These maps allow us to define counterparts in $s=o$ of the quantities for $s=e$.

We define a modified geodesic for the $n$-th excursion by
\begin{equation*}
\gamma_n^* := 
\begin{cases}\rho^{n-1}(\tilde\gamma^v)& \text{if }|\rho^{n-1}(\tilde\gamma^v)_\infty| \ge 2,\\
\tilde\iota\rho^{n-1}(\tilde\gamma^v)& \text{if }\rho^{n-1}(\tilde\gamma^v)_\infty\in (1,2),\\
\iota\rho^{n-1}(\tilde\gamma^v)&\text{if }\rho^{n-1}(\tilde\gamma^v)_\infty\in (-2,-1).
\end{cases}\\
\end{equation*}
If $p_n$ denotes the midpoint of $\gamma_n^*$,
then we define
$$h_n:= d_{\bH}(p_n,\partial H_\infty).$$
We can consider $h_n$ as the first local maximum of the height function for the $n$-th excursion, provided $h_n$ is sufficiently large. 
Since the maximal height up to the $n$-th excursion and $\sup_{1\le n \le N}h_n$ are comparable to each other, the following comparison of maximal distance with $h_n$ gives the same asymptotic comparison for the maximal height.
\begin{lem}\label{le:5.4}
For almost every $v\in T^1\mathcal M$, there are sequences $\zeta_n\ge0$, $\eta_n>0$ such that $\zeta_n, \eta_n\to 0$ and 
\begin{equation*}
    \sup_{1\le n \le N} h_n - \zeta_N \le \sup_{-u(v)\le t\le T'_{N}} d_{\M}(\pi(i),\gamma^v(t))\le \sup_{1\le n \le N+1} h_n+\eta_N,
\end{equation*}
where $-u(v)$ is the initial time of the first excursion segment, and $T'_N$ is the $N$-th return time as in Definition~\ref{de:exc}.
\end{lem}

\begin{proof}
We partition the interval $[T'_0, T'_1]$ of the first excursion into times when the excursion exits from a cuspidal neighborhood, by which we mean $\pi(N_\infty)$ or $\pi(N_1)$, and visits to the other. 
That is, we take the partition $t_0 = T'_0 < t_1 < \cdots < t_k = T'_1$ where $t_i$ is the exit time from the $(i-1)$-th cuspidal neighborhood $N(i-1) \in \{ \pi(N_\infty), \pi(N_1)\}$ such that $\gamma^v(t) \in N(i-1)$ for $t>t_{i-1}$ whenever $t-t_{i-1}$ is small enough:
$$\inf \{ t > t_{i-1}: \gamma^v(t) \notin N(i-1)\}.$$
We shall estimate $\max_{t_{i-1} \le t \le t_i}\mathrm{ht}(\gamma^v(t))$ for each $i=1, \dots, k$.
We assume that $\tilde\gamma^v_\infty>0$.
The case of \(\rho^{n-1}(\tilde\gamma^v)_\infty<0\) can be treated in the same way.

Assume $\frac53 \le \tilde\gamma^v_\infty<4$. 
To estimate the maximal height on $[t_0, t_{k-1}]$, it suffices to compute the height of $\widetilde\gamma^v(t)$ from the cusps $\infty, 1,$ and $3$ since it passes only horoballs based at $\infty, 1, 3$ for $t_0 \le t \le t_{k-1}$.
Since both endpoints of $\tilde\gamma^v$ are uniformly bounded, $h_1$ is uniformly bounded as well.
The euclidean distance of endpoints of $\tilde\gamma^v$ from $1$ is bounded from below by a positive number.
When $\tilde\gamma^v$ visits $\tau(N_1)$, it enters $\tau(N_1)$ at time $\rightarrow t_{k-1}$, or the euclidean distance of endpoints of $\tilde\gamma^v$ from $3$ is bounded below by a positive number with $N(k-2)=\pi(N_1)$.
Thus $\max_{t_0\le t \le t_{k-1}}\mathrm{ht}(\gamma^v(t))$ is uniformly bounded, and there exists $c_1$ such that 
\begin{equation}\label{eq:unifht}
    h_1 - c_1 \le \max_{t_{0}\le t \le t_{k-1}} \mathrm{ht}(\gamma^v(t)) \le h_1 +c_1.
\end{equation}
Assume $\tilde\gamma^v_\infty \ge 4$.
Then $N(0) = \pi(N_\infty)$ and the midpoint $p_1$ of the semicircle $\tilde\gamma^v$ satisfies $\mathrm{Re}
(p_1) = (\tilde\gamma^v_\infty + \tilde\gamma^v_{-\infty})/2 >1$, hence 
\begin{equation}\label{eq:htlarge}
    \max_{t_0 \le t \le t_1} \mathrm{ht}(\gamma^v(t)) = \max_{t_0 \le t \le t_1} \mathrm{ht}_\infty(\gamma^v(t)) = h_1.
\end{equation}
If $\tilde\gamma^v_\infty$ is sufficiently close to $2a_1 \pm 1$, then $k=2$, otherwise, since the euclidean distance from $\{2a_1 \pm 1\}$ is bounded below by a positive number with $N(k-2)=\pi(N_1)$, we have $$\max_{t_1 \le t \le t_2} \mathrm{ht}(\gamma^v(t)) = \max_{t_1 \le t \le t_2} \mathrm{ht}_1(\gamma^v(t)) \le h_1.$$
In the case of $1\le \tilde\gamma^v_\infty \le \frac53$, since it is reduced to the case of $4 \le \tilde\gamma^v_\infty$ after applying $\widetilde\iota$, we have 
\begin{equation}
    \max_{t_0 \le t \le t_{k-1}} \mathrm{ht}(\gamma^v(t)) = h_1
\end{equation}
if $\tilde\gamma^v_\infty \ge 4$ or $1\le \tilde\gamma^v_\infty \le \frac53$.

We denote the partition of the interval $[T'_1, T'_2]$ by the exit times of the second excursion by $s_0 = T'_1 < s_1 < \cdots < s_l = T'_2$.
If $\frac53<|\rho(\tilde\gamma^v)_\infty|<4$, then since $\mathrm{sgn}(\rho(\tilde\gamma^v)_\infty)\rho(\tilde\gamma^v)_{-\infty}\in[-\sqrt3,2-\sqrt3]$,
the Euclidean distances of both endpoints of $\rho(\tilde\gamma^v)$ from $\pm1$
are bounded below by a positive constant, where
$\pm 1 = \mathrm{sgn}(\rho(\tilde\gamma^v)_\infty)$, thus the heights of $\widetilde\gamma^v(t)$ for $\infty$ and $1$ are uniformly bounded, 
\begin{equation}
\max_{t_{k-1}\le t\le t_k}\operatorname{ht}(\gamma^v(t))
\le
\max_{s_0\le t\le s_{\ell-1}}\operatorname{ht}(\gamma^v(t))+c
\le h_2+c_1+c,
\end{equation}
where $c$ is a uniform constant accounting for the possibility that
the two maxima are attained in different cusps. 

If $|\rho(\tilde\gamma^v)_\infty|\ge4$, then 
\begin{equation}
    \max_{t_{k-1}\le t\le t_k} \mathrm{ht}(\gamma^v(t)) \le \max_{s_{0}\le t\le s_{l-1}} \mathrm{ht}(\gamma^v(t))
    = h_2.
\end{equation}

We turn to the estimate of the distance function.
If $\frac53\le\tilde\gamma^v_\infty\le {4}$, then since both $h_1$ and $\max_{t_0\le t \le t_{k-1}}d_{\M}(\pi(i),\gamma^v(t))$ are bounded, there is $c_2>0$ such that
\begin{equation}\label{eq:shallowexcursion}
    h_1 - c_2 \le \max_{t_0\le t \le t_{k-1}}d_{\M}(\pi(i),\gamma^v(t)) \le h_1 +c_2.
\end{equation} 
Assume that $\tilde\gamma^v_\infty \ge 4$ or $1\le\tilde\gamma^v_\infty \le \frac53$. 
Let $t' \in [t_0, t_1]$ maximize the restricted distance function $d_{\M}(\pi(i),\gamma^v(t))|_{[t_0, t_1]}$.
Note that either $k=2$ or $\max_{t_1\le t \le t_{k-1}}d_{\M}(\pi(i),\gamma^v(t))$ is uniformly bounded.
Since the maximal distance $d_{\M}(\pi(i),\gamma^v(t'))$ is bounded by the sum of $h_1$ and the horocyclic width $2\exp(-\mathrm{ht}({\gamma}^v(t')))\le c_3e^{-h_1}$, for some constant $c_3 \ge 2\exp(c_2)$,
\begin{equation}\label{eq:deepexcursion}
    h_1 \le \max_{t_0 \le t \le t_{k-1}} d_{\M}(\pi(i),\gamma^v(t))
    \le 
    h_1 + c_3e^{-h_1}. 
\end{equation} 
If $\frac53 \le |\rho(\tilde\gamma^v)_\infty| \le 4$, both $\max_{t_{k-1}\le t \le t_k} d_{\M}(\pi(i),\gamma^v(t))$ and $\max_{s_0\le t \le s_{l-1}} d_{\M}(\pi(i),\gamma^v(t))$ are uniformly bounded, hence we have
\begin{equation*}
    \max_{t_{k-1}\le t \le t_k} d_{\M}(\pi(i),\gamma^v(t)) \le \max\{h_1, h_2\} +c_2.
\end{equation*}
If $|\rho(\tilde\gamma^v)_\infty| \ge 4$ or $1\le |\rho(\tilde\gamma^v)_\infty|\le \frac 53$, by the same way as in \eqref{eq:deepexcursion}, we have 
\begin{equation*}
    \max_{t_{k-1}\le t \le t_k}d_{\M}(\pi(i),\gamma^v(t)) 
    \le \max_{s_0\le t \le s_1}d_{\M}(\pi(i),\gamma^v(t)) 
    \le h_2 + c_3e^{-h_2}. 
\end{equation*}
Repeating the same argument for $\rho^{n-1}(\tilde\gamma^v)$, we have the same inequality for $h_n, h_{n+1}$ and $\max_{T'_{n-1}\le t \le T'_n}d_{\M}(\pi(i),\gamma^v(t))$ for $n = 1, \dots, N$.

Since $\sup_{t\in[0, T]}\mathrm{ht}(\gamma^v(t)) \to \infty$ as $T\to\infty$ for almost every $v$, for some finite number $n$, $(\gamma^*_n)_\infty$ is greater than $4$.
Therefore, by letting
\begin{align*}
    \zeta_N &= 
    \begin{cases}
        c_2 &{\rm if }\quad \frac1L\le |(\gamma^*_{n})_\infty| \le L, \textrm{ for all }1\le n \le N,\\
        0 &{\rm otherwise},
    \end{cases}
    \quad\\
    \eta_N &= M\exp\left(-\max_{1\le n\le N+1}h_n\right),
\end{align*} 
with sufficiently large constants $L>4$ and $M>c_3$, the lemma follows.
\end{proof}

To obtain a relation between $a_n(|1/\tilde\gamma_\infty^v|)$ and the height $h(\gamma^v(t))$, we first derive a relation between $(a_n)$ and the number of tiles crossed by the $n$-th excursion of $\tilde\gamma^v$ in the following lemma.

\begin{lem} \label{superlevelsetforexcursion}
For $\gamma\in A$, the integer $a_n(|1/\gamma_\infty|)$ is the number of quadrilaterals in $\mathcal T$ which the $n$-th excursion segment passes through.
\end{lem}
\begin{proof}
If $x\in I_{(k,+1)_e}\cup I_{(k,-1)_e}$, then $1/x\in [2k-1,2k+1]$. Thus, the first excursion passes $k$ quadrilaterals of $\mathcal T$.
Since $\iota(I_{(k,\veps)_o})= I_{(k,\veps)_e}$ for $k\ge 2$, if $x\in I_{(k,-1)_o}\cup I_{(k,+1)_o}$, then the first excursion passes $k$ quadrilaterals of $\mathcal T$.
\end{proof}

We establish some auxiliary lemmas before moving on to the proof of Theorem~\ref{thm:1.2}.
For simplicity of notation, for $N\ge 1$ and $v\in T^1\M$, let $$ A_N(v) : = \max_{1\le n\le N} a_n(x(v)),\qquad \text{where}\;\; x(v)=|1/\tilde\gamma^v_\infty|.$$

\begin{lem}\label{le:mA_N}
For all $y>0$,
$$\lim_{N\to \infty}m\left\{v\in T^1\M:A_N(v)\le C_0Ny\right\}=\exp\left(-\frac{1}{y}\right).$$
\end{lem}

\begin{proof}
Let $b: T^1\M\to \Omega^{r_1}$ be the bijection described at the beginning of Section~\ref{sec:3.2}, and let $\pi_1: \Omega^{r_1}\to [0,1]$ denote the projection onto the first coordinate.
The left-hand side equals
$$
\lim_{N\to \infty}
(\pi_1\circ b)_* m\left\{x\in [0,1]:\max_{1\le n\le N} a_n(x)\le C_0 N y \right\}.$$
The measure $(\pi_1\circ b)_* m$ is an absolutely continuous measure on $[0,1]$ with respect to the Lebesgue measure, whose density function is
$$\frac{1}{C^*}\int_{\sqrt{3}-2}^{\sqrt{3}}\frac{r_1(x,y)}{\log(2+\sqrt{3})(1+xy)^2}dy.$$ 
From Theorem~\ref{th:Gal2}, this limit is $\exp(-1/y).$
\end{proof}

\begin{lem}\label{le:A_N0}
Let $\delta(N)>0$ be a sequence converging to $0$.
Then
$$\lim_{N\to \infty}m\left\{v\in T^1\M: \frac{A_N(v)}{C_0Ny}\in (1-\delta(N), 1+\delta(N))\right\}=0.$$
\end{lem}

\begin{proof}
For any $\veps>0$, there is $N'$ such that $\delta(N)<\veps$ for all $N>N'$.
Thus, for $N>N'$,
\begin{align*}
\limsup_{N\to \infty}\; m\left( \frac{A_N(v)}{C_0Ny}\in (1-\delta(N), 1+\delta(N))\right)
&\le \limsup_{N\to \infty}\; m\left(\frac{A_N(v)}{C_0Ny}\in (1-\veps,1+\veps)\right) \\
& = \exp\left(-\frac{1}{y(1+\veps)}\right)-\exp\left(-\frac{1}{y(1-\veps)}\right),
\end{align*}
which goes to $0$ when $\veps\to 0$.
\end{proof}

Let 
$$N(T, v)\text{ be the number of returns to the cross section $\pi(X)$ of $g_t(v)$ for $0 \le t \le T$.}$$
Note that $N(T,v)$ is the same as the number of returns to $\pi(X)$ of $g_t(v)$ for $-u(v)\le t\le T$.
Here $-u(v)$ is the initial time of the first excursion segment as in Definition~\ref{de:exc}.
We use the concentration of $N(T,v)$ around its mean $T/C^*$.

\begin{lem}[Concentration of Return Times]\label{lem:concentration_returns}
The geodesic flow is ergodic. By the Birkhoff Ergodic Theorem, $N(T,v)/T \to 1/C^*$ almost surely. Consequently, for any $\epsilon > 0$,
$$
\lim_{T\to\infty} m\left\{v \in T^1\mathcal{M} : \left|N(T,v) - \frac{T}{C^*}\right| > \epsilon \frac{T}{C^*} \right\} = 0.
$$
\end{lem}

Let $$C := \frac{C_0}{C^*} = \frac{4}{\pi^2},$$
where $C^*$ is the constant in Theorem~\ref{th:exctime2}.
For $T>0$ and $v\in T^1\M$, let
$$M(T, v) := \max_{0 \le t \le T} \exp d_{\M}(\pi(i),\gamma^v(t))
\qquad\text{and}\qquad
\widetilde{M}(T, v) := \max_{-u(v) \le t \le T} \exp d_{\M}(\pi(i),\gamma^v(t)).
$$
We analyze the probability 
$$P_T(y) := m\left\{M(T,v) \le u_T(y)\right\},
\qquad\text{and}\qquad
\widetilde{P}_T(y) := m\left\{\widetilde{M}(T,v) \le u_T(y)\right\},
$$ 
where the threshold is $u_T(y) := CTy$.

\begin{proof}[Proof of Theorem~\ref{thm:1.2}]
Let $\mathrm{rad}_n$ be the Euclidean radius of half circle $\gamma_n^*$.
If $\mathrm{rad}_n\ge 1$, then $\mathrm{rad}_n = \exp(h_n)$, whereas if $\mathrm{rad}_n<1$, then $\mathrm{rad}_n = \exp(-h_n)$.
The latter case can occur only when $a_n=1$.
Lemma~\ref{superlevelsetforexcursion} implies that
\begin{equation*}
a_n-\frac{3-\sqrt{3}}{2}\le \exp(h_n) \le a_n+\frac{\sqrt{3}+1}{2}.
\end{equation*}
As we saw in Lemma~\ref{le:5.4}, if the return point at time $T'_N$ is already high in the cusp, this is accounted for by the $(N+1)$-st excursion and hence by $A_{N+1}$ in the upper bound.
Thus, from Lemma~\ref{le:5.4}, there exists a constant $c>0$ such that
\begin{equation*}
e^{-\zeta_N}(A_N-c)\le \widetilde{M}(T'_N,v)\le e^{\eta_N}\left(A_{N+1}+c\right).
\end{equation*}
Using $\widetilde{M}(T'_{N(T,v)},v)\le \widetilde{M}(T,v)\le \widetilde{M}(T'_{N(T,v)+1},v)$, the preceding inequality gives
\begin{equation}\label{eq:height_bracketing}
e^{-\zeta_{N(T,v)}}(A_{N(T,v)}(v)-c) \le \widetilde{M}(T, v) \le e^{\eta_{N(T,v)+1}}(A_{N(T,v)+2}(v)+c).
\end{equation}

Fix $\epsilon > 0$. Define the deterministic time intervals:
$$
N_T^- = \lfloor (1-\epsilon)T/C^* \rfloor 
\qquad \text{and} \qquad
N_T^+ = \lfloor (1+\epsilon)T/C^* \rfloor.
$$
Let $B_\epsilon(T) :=\{ v: N(T,v) \notin [N_T^-, N_T^+]\}$. We know $m(B_\epsilon(T)) \to 0$ as $T \to \infty$.
Let $G_\epsilon(T) = B_\epsilon(T)^c$ be the ``good set". 
On $G_\epsilon(T)$, we apply \eqref{eq:height_bracketing}:
\begin{equation}\label{eq:1.2}
e^{-\zeta_{N_T^-}}\left(A_{N_T^-}(v)-c\right) \le \widetilde{M}(T, v) \le e^{\eta_{N_T^-+1}}\left(A_{N_T^++2}(v)+c\right).
\end{equation}
Here we used that $(A_N)$ is nondecreasing, while $(\zeta_N)$ and $(\eta_N)$ are nonincreasing.

Let $E_N^\zeta = \{v: \zeta_{N}(v)\not=0\}$.
Since $|\rho^n(\tilde\gamma^v)_\infty|>L$ or $\frac1L>|\rho^n(\tilde\gamma^v)_\infty|$ for infinitely many $n$ for $m$-a.e. $v$, we have $m\left(\bigcap_{N=1}^\infty E_N^\zeta\right) = 0$.
Moreover, $(E_N^\zeta)$ is a nested decreasing sequence, and hence we have $\lim_{N\to \infty}m(E_N^\zeta) = m\left(\bigcap_{N=1}^\infty E_N^\zeta\right) = 0$. 
For the upper bound, from \eqref{eq:1.2},
\begin{equation*}\begin{split}
\widetilde{P}_T(y) &= m\Big(\widetilde{M}(T,v) \le u_T(y)\Big)
 \le m\left(e^{-\zeta_{N_T^-}}\left(A_{N_T^-}(v)-c\right) \le u_T(y)\right) + o(1)\\
& \le m\left(A_{N_T^-}(v)-c \le u_T(y)\right) + m(E_{N_T^-}^\zeta) + o(1) 
 = m\left(A_{N_T^-}(v)-c \le u_T(y)\right) + o(1).
\end{split}\end{equation*}
We analyze the threshold $u_T(y)$ relative to $N_T^-$. Since $T = N_T^- C^*/(1-\epsilon) + O(1)$,
\begin{equation*}
u_T(y)+c = CTy+c 
= C_0N_T^- \frac{y}{1-\epsilon}\left( 1 + \frac{O(1)}{N_T^-}\right).
\end{equation*}
By Lemma~\ref{le:mA_N} and Lemma~\ref{le:A_N0}, we have:
$$
\lim_{T\to\infty} m\left(A_{N_T^-}(v) \le u_T(y)+c\right) = \exp\left(-\frac{1-\epsilon}{y}\right).
$$
Thus, $\limsup_{T\to\infty} \widetilde{P}_T(y) \le e^{-(1-\epsilon)/y}$.

We fix $\delta>0$, and let $E_{N,\delta}^\eta = \{v:\eta_N(v)>\delta\}$. 
Since $\eta_N(v)\to 0$ for $m$-a.e. $v$, we have $m(E_{N,\delta}^\eta)\to 0$ as $N\to \infty$.
For the lower bound, from \eqref{eq:1.2},
\begin{equation*}\begin{split}
\widetilde{P}_T(y) & \ge m\Big(v \in G_\epsilon(T) \text{ and } \widetilde{M}(T,v) \le u_T(y)\Big)
\ge m\left(A_{N_T^++2}(v) 
\le e^{-\eta_{N_T^-+1}}u_T(y)-c\right) - o(1)\\
& \ge m\left(A_{N_T^+ +2}(v)\le e^{-\delta}u_T(y)-c\right)-m(E_{N_T^- +1,\delta}^\eta)-o(1) = m\left(A_{N_T^+ +2}(v)\le e^{-\delta}u_T(y)-c\right)-o(1)
\end{split}\end{equation*}
We analyze the threshold $u_T(y)$ relative to $T=C^*(N_T^+ + 2)/(1+\veps)+O(1)$.
Thus, 
$$
e^{-\delta}u_T(y)-c 
= e^{-\delta}CTy-c 
= C_0(N_T^+ + 2) \frac{ye^{-\delta}}{1+\epsilon} \left(1 + \frac{O(1)}{e^{-\delta}(N_T^+ + 2)}\right).
$$
By Lemma~\ref{le:mA_N} and \ref{le:A_N0},
$$
\lim_{T\to\infty} m\left(A_{N_T^+ + 2}(v) 
\le e^{-\delta}u_T(y)-c\right) = \exp\left(-\frac{1+\epsilon}{y e^{-\delta}}\right).
$$
Thus, $\liminf_{T\to\infty} \widetilde{P}_T(y) \ge e^{-e^{\delta}(1+\epsilon)/y}$.
By letting $\delta\to0$, we have $\liminf_{T\to\infty}\widetilde{P}_T(y) \ge e^{-(1+\epsilon)/y}$.

We have established
$$
e^{-(1+\epsilon)/y} \le \liminf_{T\to\infty}\widetilde{P}_T(y) \le \limsup_{T\to\infty} \widetilde{P}_T(y) \le e^{-(1-\epsilon)/y}.
$$
Since $\epsilon > 0$ was arbitrary, we let $\epsilon \to 0$. Both the lower and upper bounds converge to $e^{-1/y}$, which implies that
$$
\lim_{T\to\infty} \widetilde{P}_T(y) = e^{-1/y}.
$$

Let $K(v) = \max_{-u(v)\le t\le 0}\exp d_{\M}(\pi(i),\gamma^v(t))$.
We have $\widetilde{M}(T,v)=\max\{K(v),M(T,v)\}$.
Since $K(v)<\infty$ for $m$-a.e. $v$, we have $m(K(v)>CTy)\to 0$ as $T\to\infty$.
Hence,
$$0\le P_T(y)-\widetilde{P}_T(y) = m(M(T,v)\le CTy,\; K(v)>CTy)\le m(K(v)>CTy)\to 0.$$
Therefore, 
$$\lim_{T\to \infty}P_T(y) = \lim_{T\to \infty}\widetilde{P}_T(y)=e^{-1/y}.$$
\end{proof}

\appendix

\section{Excursion Times and the exponential growth rate of $Q_N$}
\label{app:scf-levy}
In this section, we relate the asymptotic average excursion time for the geodesic flow to the growth of the principal convergents of the SCF. 
For regular continued fractions, the almost-sure exponential growth rate of the denominators $q_n$ of the principal convergents is called the L\'evy constant. Under the standard coding of the geodesic flow on the modular surface, the corresponding cumulative return time is $2\log q_n+o(n)$.
We establish the analogous property for the principal convergents $P_n/Q_n$ of the SCF. 
To this end, we derive some basic properties of the convergents $P_N/Q_N$ and a formula expressing the hyperbolic length of each geodesic excursion segment in terms of the expansion data of the SCF and its dual.

\begin{defn}
We define the \emph{convergents} of the SCF of $x=[0;(a_1,\varepsilon_1)_{s_1},(a_2,\varepsilon_2)_{s_2},\dots]$ by
$$\frac{P_n(x)}{Q_n(x)} = h_{(a_1,\veps_1)_{s_1}}\circ h_{(a_2,\veps_2)_{s_2}}\circ\cdots\circ h_{(a_n,\veps_n)_{s_n}}(0).$$
See \eqref{eq:inv_br} for the definition of the inverse branches $h_{(a,\veps)_s}$.
We also define 
$$
\widehat{P}_n(x) = P_n(\iota(x))\quad\text{and}\quad
\widehat{Q}_n(x) = Q_n(\iota (x)).
$$
\end{defn}

For two sequences $(A_N)$ and $(B_N)$, we write $A_N\asymp B_N$ if there exists a constant $C>1$ such that $1/C<A_N/B_N<C$ for all $N$. In this case,
$$\lim_{N\to \infty}\frac{\log A_N}{N}=\lim_{N\to\infty}\frac{\log B_N}{N}.$$

\begin{thm}\label{th:exctime}
For $v\in T^1\M$ for which the lift in Definition~\ref{de:exc} is defined, set $x:=1/|\tilde\gamma^v_\infty|\in(0,1)$. For almost every $v$,
$$\lim_{N\to\infty}\frac{T'_N}{N}=\lim_{N\to\infty}\frac{T_N}{N}= \lim_{N\to\infty}\frac{2\log Q_N}{N} = \int_0^1 \log f d\mu,$$
where 
\begin{equation}\label{eq:f}
f(x):=\begin{cases}
\frac{1}{(Tx)^2},&\text{if }x\in(0,\frac12]\cap T^{-1}(0,\frac12],\vspace{0.5ex}\\
\frac{2}{(1-Tx)^2},&\text{if }x\in(0,\frac12]\cap T^{-1}(\frac12,1),\vspace{0.5ex}\\
\frac{(1+Tx)^2}{(1-Tx)^2},&\text{if }x\in(\frac12,1)\cap T^{-1}(\frac12,1),\vspace{0.5ex}\\
\frac{(1+Tx)^2}{2(Tx)^2},&\text{if }x\in(\frac12,1)\cap T^{-1}(0,\frac12].\end{cases}
\end{equation}
\end{thm}
\begin{proof} Let us first explain the proof up to Proposition~\ref{pr:r_n} and lemmas~\ref{le:inc}-\ref{le:2}.

Define $(\alpha_n, \beta_n) = \pm ( \rho^{n-1}(\gamma)_\infty, -\rho^{n-1}(\gamma)_{-\infty}),$ where $\pm$ is the sign of $\rho^{n-1}(\gamma)_\infty.$ 
Let us omit $\gamma$ for a while. In Proposition~\ref{pr:r_n}, we prove that
\begin{equation*}
    r_n = \begin{cases}
        \frac{1}{2} \log L(\alpha_n, \beta_n) & {\rm if }s_n =e,\\
        \frac{1}{2} \log L(\widetilde\iota\alpha_n, \iota\beta_n) &{\rm if }s_n =o
    \end{cases}
\end{equation*}
where $L(\alpha, \beta)$ is a function of $\alpha, \beta$.
For brevity, let 
\begin{equation}\label{eq:alphastar}
(\alpha_{n}^{\ast},\beta_{n}^{\ast}) = \begin{cases}(\alpha_{n},\beta_{n})&\text{ if }s_n=e,\\
(\tilde{\iota}(\alpha_n),\iota(\beta_n))&\text{ if }s_n=o.\end{cases}
\end{equation}
The trick is that 
$$ L(\alpha_n^*, \beta_n^*) =L_{\alpha, n} L_{\beta,n}, $$
where $(L_{\alpha,n})_{n\ge1}$ and $(L_{\beta,n})_{n\ge1}$ denote two separate families: $L_{\alpha,n}$ is a rational function depending only on $\alpha_n$, whereas $L_{\beta,n}$ depends only on $\beta_n$ and the first SCF digit of $\alpha_n$.
We show that, for almost every $x\in[0,1]$,
\begin{enumerate}
\item (Lemma~\ref{le:f}) we have 
$$ \lim_{N\to \infty}\frac{\log \prod_{n=1}^NL_{\alpha,n}
}{N} = \int_0^1\log fd\mu<\infty,$$
\item (Lemmas \ref{le:p/q.star} and \ref{le:1}) 
$$\lim_{N\to \infty}\frac{\log\prod_{n=1}^N L_{\alpha,n}-2\log R_N}{N}= 0,$$ where
\begin{equation}\label{eq:R}R_N:=
\begin{cases}Q_{N+1}&\text{if }s_N=s_{N+1},\\
Q_N&\text{if }s_N\not=s_{N+1},
\end{cases}\end{equation}
\item (Lemmas \ref{le:p/q.star} and \ref{le:2}) $\prod_{n=1}^N L_{\beta,n} \asymp Q_N^2$.
\end{enumerate}

Since $(\log R_N)/N$ has a limit by (1) and (2), and $Q_N\uparrow \infty$ by Lemma~\ref{le:inc}, we obtain
$$\lim_{N\to\infty}\frac{\log Q_N}{N} = \lim_{N\to\infty}\frac{\log R_N}{N}.$$

Combining (1), (2), and (3), we deduce that
$$\lim_{N\to \infty}\frac{\log \prod_{n=1}^NL_{\beta,n}}{N} =\lim_{N\to\infty}\frac{2\log Q_N}{N}= \lim_{N\to \infty}\frac{\log \prod_{n=1}^NL_{\alpha,n}}{N} = \int_0^1\log fd\mu$$
for almost all $x\in(0,1)$.
Therefore, 
\begin{align*}
\lim_{N\to\infty}\frac{\sum_{n=1}^N r_n}{N} 
& = \lim_{N\to\infty}\frac{\frac{1}{2}\sum_{n=1}^N \log L_{\alpha,n}L_{\beta,n}}{N}
 = \lim_{N\to\infty}\frac{2\log(Q_N)}{N}=\int_0^1\log fd\mu
\end{align*}
for almost all $x\in(0,1)$.
\end{proof}

We now begin the proof of the lemmas mentioned in the proof above.
In Proposition~\ref{pr:r_n}, we express $r_n$ in terms of the endpoints of $\rho^{n-1}(\tilde{\gamma}^v)$.
The expression takes a different form depending on whether $s_n=e$ or $s_n=o$.
If $s_n=e$, then the hyperbolic length of the $n$-th excursion segment can be obtained by direct calculation.
On the other hand, if $s_n=o$, then the segment is first mapped via $\iota$ when $\rho^{n-1}(\tilde{\gamma}^v)_\infty<0$ or $\widetilde{\iota}$ when $\rho^{n-1}(\tilde{\gamma}^v)_\infty>0$,
and the length is then computed after this transformation.

\begin{defn}
We define a substitution $\phi:\mathscr A\to \mathscr A$ by
$$\begin{cases} (1,1)_e \mapsto (1,1)_e,\\
(k,\varepsilon)_e \mapsto (k,\varepsilon)_{o} \text{ when }(k,\varepsilon)\ne (1,1),\\
(k,\varepsilon)_{o} \mapsto (k,\varepsilon)_e.
\end{cases}$$
\end{defn}
The involution $\iota(x)=\frac{1-x}{1+x}$ replaces each digit of the SCF and the dual SCF with its image under the substitution $\phi$.

\begin{lem}\label{le:iota}
For $x = [0;(a_1,\varepsilon_1)_{s_1},(a_2,\varepsilon_2)_{s_2},\dots]$ and $y = \langle(b_1,\eta_1)_{t_1},(b_2,\eta_2)_{t_2},\dots\rangle$, we have
$$\iota(x) = [0;\phi((a_1,\varepsilon_1)_{s_1}), \phi((a_2,\varepsilon_2)_{s_2}),\dots]\qquad\text{ and }\qquad\iota(y) = \langle \phi((b_1,\eta_1)_{t_1}),\phi((b_2,\eta_2)_{t_2}),\dots\rangle.$$
In particular, it follows that for $x=[0;(a_1,\varepsilon_1)_{s_1},(a_2,\varepsilon_2)_{s_2},\dots]$,
$$\frac{\widehat{P}_n(x)}{\widehat{Q}_n(x)} = h_{\phi(a_1,\veps_1)_{s_1}}\circ h_{\phi(a_2,\veps_2)_{s_2}}\circ\cdots\circ h_{\phi(a_n,\veps_n)_{s_n}}(0).$$ 
\end{lem}
\begin{proof}
Since $\iota \CF = \CF \iota$,
we have $x\in I_{(a,\varepsilon)_{s}}$ if and only if $\iota(x)\in I_{\phi((a,\varepsilon)_{s})}$.
Moreover, $T^{i-1}x\in I_{(a_i,\varepsilon_i)_{s_i}}$ implies $T^{i-1}(\iota x)=\iota(T^{i-1}x)\in I_{\phi((a_i,\varepsilon_i)_{s_i})}$.
Similarly, we have $\iota \F = \F\iota$ and $y\in \bar{I}_{(b,\eta)_t}$ if and only if $\iota(y)\in \bar{I}_{\phi((b,\eta)_t)}$.
\end{proof}

For the pair of endpoints $(\gamma_\infty,\gamma_{-\infty})\in \mathcal I$ of a geodesic $\gamma$ in $A$, let the corresponding pair of SCF and dual SCF expressions be
$$\gamma_\infty=j_\gamma\cdot[(a_1,\varepsilon_1)_{s_1};(a_2,\varepsilon_2)_{s_2},\dots]\quad\text{and}\quad\gamma_{-\infty}=-j_\gamma\cdot\langle(a_0,\veps_0)_{s_0},(a_{-1},\veps_{-1})_{s_{-1}},\dots\rangle,$$
where $j_\gamma =\mathrm{sgn}(\gamma_\infty)$.
From the fact that $\tilde\iota$ is conjugate to $\iota$ via both $z\mapsto -z$ and $z\mapsto 1/z$, it follows that
\begin{equation}\label{eq:CFconv}
\begin{split}
 j_\gamma\cdot[\phi((a_1,\varepsilon_1)_{s_1}); \phi((a_2,\varepsilon_2)_{s_2}),\dots]
& =\begin{cases}\tilde\iota(\gamma_\infty)&\text{if }j_\gamma=1,\\
\iota(\gamma_\infty)&\text{if }j_\gamma=-1,\end{cases}\\
 -j_\gamma\cdot\langle \phi((a_0,\veps_0)_{s_0}),\phi((a_{-1},\veps_{-1})_{s_{-1}}),\dots\rangle
& =\begin{cases}\tilde\iota(\gamma_{-\infty})&\text{if }j_\gamma=1,\\
\iota(\gamma_{-\infty})&\text{if }j_\gamma=-1.\end{cases}
\end{split}
\end{equation}

%For given $\gamma\in A$,
In Section~\ref{sec:cod}, we observed that the first return map $\Phi$ of the geodesic flow corresponds to the two-sided shift on $\mathscr A^\mathbb Z$, and $\Phi$ is expressed explicitly via $\rho$.
Now we consider $\{\gamma^{(n)}=\rho^{n-1}(\gamma)\}$, which is the $\rho$-orbit of $\gamma$.
Equivalently, we consider the $\Phi$-orbit $\{(\gamma^{(n)}_{\infty},\gamma^{(n)}_{-\infty})\}$ of $(\gamma_\infty,\gamma_{-\infty})\in\mathcal I$.
From \eqref{eq:rho.shift}, we have, for $n\ge 1$,
$$\gamma_\infty^{(n)}=(-1)^{n-1}\veps_1\veps_2\cdots\veps_{n-1}j_\gamma\cdot[(a_n,\veps_n)_{s_n};(a_{n+1},\veps_{n+1})_{s_{n+1}},(a_{n+2},\veps_{n+2})_{s_{n+2}},\dots],$$
$$\gamma_{-\infty}^{(n)} = (-1)^n\veps_1\veps_2\cdots \veps_{n-1}j_\gamma\cdot\langle (a_{n-1},\veps_{n-1})_{s_{n-1}},(a_{n-2},\veps_{n-2})_{s_{n-2}},(a_{n-3},\veps_{n-3})_{s_{n-3}},\dots\rangle.$$

For brevity, we define $$\alpha_n(\gamma): = |\gamma_\infty^{(n)}|\in(1,\infty)\qquad\text{and}\qquad \beta_n(\gamma):=-\mathrm{sgn}(\gamma_{\infty}^{(n)})\cdot\gamma_{-\infty}^{(n)}\in(\sqrt{3}-2,\sqrt{3}].$$
Equivalently, 
\begin{equation}\label{eq:albe1}
\begin{split}
\alpha_n &= [(a_n,\veps_n)_{s_n};(a_{n+1},\veps_{n+1})_{s_{n+1}},\dots],\\
\beta_n &= \langle (a_{n-1},\veps_{n-1})_{s_{n-1}},(a_{n-2},\veps_{n-2})_{s_{n-2}},\dots\rangle.
\end{split}
\end{equation}
By combining Lemma~\ref{le:iota} and Eq.~\eqref{eq:CFconv}, we describe the image of the substitution $\phi$ as
\begin{equation}\label{eq:albe2}
\begin{split}
\tilde\iota(\alpha_n) & = [\phi((a_n,\veps_n)_{s_n});\phi((a_{n+1},\veps_{n+1})_{s_{n+1}}),\dots],\\
\iota(\beta_n) & =  \langle \phi((a_{n-1},\veps_{n-1})_{s_{n-1}}),\phi((a_{n-2},\veps_{n-2})_{s_{n-2}}),\dots\rangle.
\end{split}
\end{equation}
Using \eqref{eq:albe1} and \eqref{eq:albe2}, we express the $n$-th excursion time $r_n$ of $\gamma$ in terms of $\alpha_n^*$ and $\beta_n^*$.

\begin{prop}\label{pr:r_n}
The $n$-th excursion time of $\gamma$ is
$$r_n = \frac{1}{2}\log L(\alpha_n^*,\beta_n^*),$$
where $(\alpha_n^*,\beta_n^*)$ is defined as in \eqref{eq:alphastar},
$$L(\alpha,\beta) = \frac{\alpha-1}{-(\alpha-2a_1)(\alpha-(2a_1+\varepsilon_1))}\cdot \frac{(\beta+2a_1)(\beta+(2a_1+\varepsilon_1))}{\beta+1},$$
and $(a_1,\veps_1)_{s_1}$ is the first SCF digit of $\alpha$.
\end{prop}

\begin{proof}
For an $e$-type $n$-th excursion, replace $\gamma$ by $\rho^{n-1}(\gamma)$ and relabel its endpoints and first digit as $\gamma_\infty,\gamma_{-\infty}$ and $(a_1,\varepsilon_1)_{s_1}$. It is therefore enough to compute $r_1$ under the assumptions $j_\gamma=1$ and $s_1=e$; the case $j_\gamma=-1$ follows by symmetry.
Let $\xi$ be the point where $\gamma$ intersects the vertical line $\mathrm{Re}(z)=1$, and $\eta$ be the point at which $\gamma$ intersects the arc connecting $2a_1$ and $2a_1+\varepsilon_1$, see Figure~\ref{fi:exc}.
Using the relation between the hyperbolic distance and the cross-ratio, 
$$r_1=d_{\bH}(\xi,\eta)=\log \frac{|\gamma_\infty-\xi|\cdot |\eta-\gamma_{-\infty}|}{|\gamma_\infty-\eta|\cdot |\xi-\gamma_{-\infty}|}.$$
 By the geometric mean theorem, $\frac{|\gamma_{\infty}-\xi|}{|\xi-\gamma_{-\infty}|}=\sqrt{\frac{|\gamma_{\infty}-1|}{|1-\gamma_{-\infty}|}}$ and $\frac{|\eta-\gamma_{-\infty}|}{|\gamma_{\infty}-\eta|} = \sqrt{\frac{|\mathrm{Re}\,\eta-\gamma_{-\infty}|}{|\gamma_\infty-\mathrm{Re}\,\eta|}}$, so
$$d_{\bH}(\xi,\eta)=\frac{1}{2}\log \frac{|\gamma_\infty-1|\cdot |\mathrm{Re}\,\eta-\gamma_{-\infty}|}{|1-\gamma_{-\infty}|\cdot |\gamma_{\infty}-\mathrm{Re}\,\eta|}.$$
Since $\eta$ is the intersection of two half circles
$$\left(\mathrm{Re}(z)-\frac{4a_1+\varepsilon_1}{2}\right)^2+\mathrm{Im}(z)^2=\frac{1}{4}\quad\text{ and }\quad\left(\mathrm{Re}(z)-\frac{\gamma_\infty+\gamma_{-\infty}}{2}\right)^2+\mathrm{Im}(z)^2=\left(\frac{\gamma_\infty-\gamma_{-\infty}}{2}\right)^2,$$
the real part of $\eta$ is
$$\mathrm{Re}\, \eta = \frac{\gamma_\infty\gamma_{-\infty}-2a_1(2a_1+\varepsilon_1)}{(\gamma_\infty+\gamma_{-\infty})-(4a_1+\varepsilon_1)}.$$
Substituting the expression for $\mathrm{Re}\;\eta$ into the above ratio, we obtain
\begin{align*}
\frac{|\gamma_\infty-1||\mathrm{Re}\,\eta-\gamma_{-\infty}|}{|1-\gamma_{-\infty}||\gamma_\infty-\mathrm{Re}\,\eta|} 
 = \frac{\gamma_\infty-1}{-(\gamma_\infty-2a_1)(\gamma_\infty-(2a_1+\varepsilon_1))}\cdot \frac{(\gamma_{-\infty}-2a_1)(\gamma_{-\infty}-(2a_1+\varepsilon_1))}{1-\gamma_{-\infty}}.
\end{align*}
Applying this calculation to $\rho^{n-1}(\gamma)$ proves that if $s_n=e$, then $r_n=\frac{1}{2}\log L(\alpha_n,\beta_{n})$.

When $s_n=o$, we first send $\rho^{n-1}(\gamma)$ to its $e$-type counterpart by $\widetilde{\iota}$ or $\iota$, as described in \eqref{eq:transf}, and then apply the preceding calculation.
By \eqref{eq:albe2}, if $s_n=o$, then $r_n = \frac{1}{2}\log L(\tilde\iota(\alpha_n),\iota(\beta_{n}))$.
\end{proof}

Let 
$$x=x_1=1/\alpha_1 = [0;(a_1,\varepsilon_1)_{s_1},(a_2,\varepsilon_2)_{s_2},\dots]$$
and $x_n=1/\alpha_n$ for $n\ge 1$.
Thus, we have
$$x_n = 1/\alpha_n = [0;(a_n,\veps_n)_{s_n},(a_{n+1},\veps_{n+1})_{s_{n+1}},\dots].$$

Next, we consider a matrix formula associated with the convergents of the SCF expansion of $x$.
Each matrix $M=\left(\begin{smallmatrix}a&b\\c&d\end{smallmatrix}\right)$ acts on $\mathbb H^2$ by a linear fractional map or its complex conjugate as
$$M.z=\frac{az+b}{cz+d}\quad\text{ if }\det(M)=1,\qquad\text{ and }\qquad M.z=\frac{a\overline{z}+b}{c\overline{z}+d}\quad\text{ if }\det(M)=-1.$$
Let $M_{(a,\veps)_s}$ denote the matrix associated with the inverse branch $h_{(a,\veps)_s}$, so that we have $M_{(a,\veps)_s}.x = h_{(a,\veps)_s}(x)$.
Set $\overline\veps := \max(0,\veps)$.
With this notation, the matrices take the form
$$M_{(a,\varepsilon)_e} = \begin{pmatrix}0&1\\ \varepsilon&2a\end{pmatrix},\quad M_{(a,\veps)_o} = \begin{pmatrix}0&1\\1&1\end{pmatrix}\begin{pmatrix}0&1\\ \veps &a-\overline\veps \end{pmatrix}\begin{pmatrix}0&1\\1&1\end{pmatrix} = \begin{pmatrix} a-\overline\veps & a-\overline\veps +\veps \\ a-\overline\veps +1 & a-\overline\veps+\veps+1\end{pmatrix}.$$

Let us denote by
$$\begin{cases}
M_n(x):=M_{(a_1,\varepsilon_1)_{s_1}}M_{(a_2,\varepsilon_2)_{s_2}}\cdots M_{(a_n,\varepsilon_n)_{s_n}},\\
\widehat{M}_n(x) := M_{\phi((a_1,\varepsilon_1)_{s_1})}M_{\phi((a_2,\varepsilon_2)_{s_2})}\cdots M_{\phi((a_n,\varepsilon_n)_{s_n})}.
\end{cases}$$
The numerator $P_n(x)$ ($\widehat{P}_n(x)$, resp.) and denominator $Q_n(x)$ ($\widehat{Q}_n(x)$, resp.) of a convergent are the (1,2)-entry and the (2,2)-entry of $M_n(x)$ ($\widehat{M}_n(x)$, resp.).
It implies that 
$$M_n(x).0 = P_n/Q_n\qquad \text{and} \qquad \widehat{M}_n(x).0=\widehat{P}_n/\widehat{Q}_n.$$

We write
$$M_n(x) =: \begin{pmatrix}U_n & P_n\\ V_n& Q_n\end{pmatrix}\qquad\text{and}\qquad
\widehat{M}_n(x) =: \begin{pmatrix}\widehat{U}_n & \widehat{P}_n\\ \widehat{V}_n& \widehat{Q}_n\end{pmatrix}.$$
If $s_n=e$, then $U_n=\veps_n P_{n-1}$ and $V_n = \veps_n Q_{n-1}$, if $s_n=o$, then $\widehat U_n=\veps_n \widehat P_{n-1}$ and $\widehat V_n = \veps_n \widehat Q_{n-1}$.

From \eqref{eq:inv_br2}, we have $x = M_n.x_{n+1}$ and $\iota(x) = \widehat M_n.\iota(x_{n+1})$.
Thus, we have \begin{equation}\label{eq:Mx}
x=\frac{\veps_nP_{n-1}x_{n+1}+P_n}{\veps_nQ_{n-1}x_{n+1}+Q_{n}}\quad\text{if }s_n=e,\qquad
\iota(x)=\frac{\veps_n\widehat P_{n-1}\iota (x_{n+1})+\widehat P_n}{\veps_n\widehat Q_{n-1}\iota (x_{n+1})+\widehat Q_{n}}\quad\text{if }s_n=o,
\end{equation}
\begin{equation}\label{eq:M^-1x}
x_{n+1} =  \frac{-\veps_n(Q_{n}x-P_{n})}{Q_{n-1}x-P_{n-1}}\quad\text{if }s_n=e,
\quad\text{and}\quad 
\iota(x_{n+1}) =  \frac{-\veps_n(\widehat Q_{n}\iota (x)-\widehat P_{n})}{\widehat Q_{n-1}\iota(x)-\widehat P_{n-1}}\quad\text{if }s_n=o.
\end{equation}
Moreover, since $M_{(a,\veps)_s}^t$ represents $\bar{h}_{(a,\veps)_s}$ as in \eqref{eq:inv.barT}, 
$$\beta_{n+1} = \bar{h}_{(a_{n},\veps_{n})_{s_{n}}}\circ \bar{h}_{(a_{n-1},\veps_{n-1})_{s_{n-1}}}\circ\cdots \bar{h}_{(a_{1},\veps_{1})_{s_{1}}}(\beta_1)= M_n^t.\beta_1.$$
Thus,
\begin{equation}\label{eq:Mtbeta}\beta_{n+1} = \frac{\varepsilon_n(P_{n-1}\beta_{1}+Q_{n-1})}{P_n\beta_{1}+Q_n}\quad\text{if }s_n=e,\qquad%\text{and}\quad
\iota(\beta_{n+1}) = \frac{\varepsilon_n(\widehat{P}_{n-1}\iota(\beta_{1})+\widehat{Q}_{n-1})}{\widehat{P}_n\iota(\beta_{1})+\widehat{Q}_n}\quad\text{if }s_n=o.\end{equation}

\begin{lem}\label{le:inc}
For all $x\in(0,1)$, the sequence $(Q_n)_{n\ge1}$ is strictly increasing.
\end{lem}

\begin{proof}
If $s_1=e$, then $Q_1+V_1=2a_1+\veps_1\ge 3$.
If $s_1=o$, then $Q_1+V_1=2a_1+1\ge 3$.
For $n\ge 2$, 
$$\begin{cases}Q_n = 2a_n Q_{n-1}+V_{n-1}, \;\; V_n = \veps_nQ_{n-1},&\text{if }s_n=e,\\
Q_n = (a_n-\overline\veps_n+\veps_n)(Q_{n-1}+V_{n-1})+Q_{n-1},\;\; V_n = (a_n-\overline\veps_n)(Q_{n-1}+V_{n-1})+Q_{n-1}& \text{if }s_n=o.
\end{cases}$$
In either case, $Q_n+V_n>Q_{n-1}+V_{n-1}$, and hence $Q_n+V_n\ge 3$ for all $n$ by induction.
Moreover,
$$Q_n-Q_{n-1}=
\begin{cases}(2a_n-1)Q_{n-1}+V_{n-1}& \text{if }s_n=e,\\
(a_n-\overline\veps_n+\veps_n)(Q_{n-1}+V_{n-1})& \text{if }s_n=o,
\end{cases}
$$
which implies that $Q_n-Q_{n-1}\ge Q_{n-1}+V_{n-1}>0$ in both cases.
\end{proof}

\begin{lem}\label{le:p/q.star}
If $s_j=e$, then we have
$$\begin{pmatrix}{\widehat{P}_j}\\{\widehat{Q}_j}\end{pmatrix} = \frac{1}{2}\begin{pmatrix}{\varepsilon_j(Q_{j-1}-P_{j-1})+(Q_{j}-P_{j})}\\{\varepsilon_j(Q_{j-1}+P_{j-1})+(Q_{j}+P_{j})}\end{pmatrix}.$$
If $s_j=o$, then we have
$$\begin{pmatrix}{P_j}\\{Q_j}\end{pmatrix} = \frac{1}{2}\begin{pmatrix}{\varepsilon_j(\widehat{Q}_{j-1}-\widehat{P}_{j-1})+(\widehat{Q}_{j}-\widehat{P}_{j})}\\{\varepsilon_j(\widehat{Q}_{j-1}+\widehat{P}_{j-1})+(\widehat{Q}_{j}+\widehat{P}_{j})}\end{pmatrix}.$$
Moreover, we have $Q_n \asymp \widehat{Q_n}.$
\end{lem}

\begin{proof}
The map $\iota$ is represented by the matrix $E:=\begin{pmatrix}-1&1\\1&1\end{pmatrix}$, and its inverse is $\tfrac{1}{2}E$.
We have
$E^{-1}
M_{(a_n,\varepsilon_n)_{s_n}}
E = 
M_{\phi((a_n,\varepsilon_n)_{s_n})}.$
Thus
$$\begin{pmatrix}{\widehat{P}_j}\\{\widehat{Q}_j}\end{pmatrix}=\widehat{M}_j\begin{pmatrix}0\\1\end{pmatrix} = E^{-1} M_j E\begin{pmatrix}0\\1\end{pmatrix}
\quad\text{and}\quad
\begin{pmatrix}{P_j}\\{Q_j}\end{pmatrix}=M_j\begin{pmatrix}0\\1\end{pmatrix} = E^{-1} \widehat{M}_j E \begin{pmatrix}0\\1\end{pmatrix}.$$
Thus, $U_j=\varepsilon_jP_{j-1}$ and $V_j=\varepsilon_jQ_{j-1}$ if and only if $s_j=e$.
For the same reason, 
$\widehat U_j=\varepsilon_j \widehat P_{j-1}$ and $\widehat  V_j=\varepsilon_j\widehat Q_{j-1}$ if and only if $s_j=o$.
This completes the proof of the matrix identities.

If $s_j=e$, then $\widehat{P}_j+\widehat{Q}_j = \veps_j Q_{j-1} + Q_j$,
that is, $$\frac{\widehat{Q}_j}{Q_j}=\frac{1+\veps_jQ_{j-1}/Q_j}{1+\widehat{P}_j/\widehat{Q}_j}.$$
If $s_j=e$, then 
\begin{equation}\label{eq:eQ/Q}
\veps_j \frac{Q_{j-1}}{Q_j} = M_j^t.0 = \langle(a_j,\veps_j)_{s_j},(a_{j-1},\veps_{j-1})_{s_{j-1}},\cdots,(a_1,\veps_1)_{s_1}\rangle\in\left[\sqrt3-2,\frac1{\sqrt{3}}\right]. 
\end{equation}
From $\widehat{P}_j/\widehat{Q}_j\in(0,1)$, there exists a constant $C$ such that $1/C<\widehat{Q}_j/Q_j<C$ for all $j$ with $s_j=e$.
The same conclusion holds symmetrically when $s_j=o$.
\end{proof}
\begin{lem}\label{le:f}
For almost every $x\in [0,1]$, we have 
$$\lim_{N\to \infty}\frac{\sum_{n=1}^N\log L_{\alpha,n}}{N}=\int_0^1 \log fd\mu.$$
\end{lem}
\begin{proof}
By Proposition~\ref{pr:r_n}, the $N$-th return time (total excursion time) is $$T_N=r_1+r_2+\cdots+r_{N}=\frac{1}{2}\sum_{n=1}^N\log L(\alpha^\ast_n,\beta^\ast_n).$$
We recall $(\alpha_{n}^{\ast},\beta_{n}^{\ast})$ and define $(\alpha_{n+1}^{\ast\ast},\beta_{n+1}^{\ast\ast})$ for $n\ge 1$ here as
$$(\alpha_{n}^{\ast},\beta_{n}^{\ast}) = \begin{cases}(\alpha_{n},\beta_{n})&\text{ if }s_n=e,\\
(\tilde{\iota}(\alpha_n),\iota(\beta_n))&\text{ if }s_n=o,\end{cases}\quad\text{and}\quad
(\alpha_{n+1}^{\ast\ast},\beta_{n+1}^{\ast\ast}) := \begin{cases}(\alpha_{n+1},\beta_{n+1})&\text{ if }s_n=e,\\
(\tilde{\iota}(\alpha_{n+1}),\iota(\beta_{n+1}))&\text{ if }s_n=o.\end{cases}$$
The following relations hold:
$$\alpha_{n}^\ast=2a_n+\frac{\veps_n}{\alpha_{n+1}^{\ast\ast}}
\qquad\text{and}\qquad  
\beta_{n+1}^{\ast\ast}=\frac{\veps_n}{2a_n+\beta_{n}^{\ast}}.$$
From the relations, we express $L(\alpha^\ast_n,\beta^\ast_n)$ as the product of
$$\frac{\alpha_{n}^\ast-1}{-(\alpha_{n}^\ast-2a_n)(\alpha_{n}^\ast-(2a_n+\veps_n))} = \frac{\alpha_{n}^\ast-1}{\alpha_{n+1}^{\ast\ast}-1}\cdot (\alpha_{n+1}^{\ast\ast})^2=:L_{\alpha,n},$$
$$\frac{(\beta_{n}^{\ast}+2a_n)(\beta_{n}^{\ast}+2a_n+\veps_n)}{\beta_{n}^{\ast}+1}=\frac{\beta_{n+1}^{\ast\ast}+1}{\beta_{n}^{\ast}+1}\cdot\frac{1}{(\beta_{n+1}^{\ast\ast})^2}=:L_{\beta,n}.$$
We will approximate the logarithmic average of $L_{\alpha, n} L_{\beta, n}$ by that of $L_{\alpha,n}^2$.

With the relation $\alpha_n=1/T^{n-1}x$, we regard $\alpha_n^{\ast}$ and $\alpha_{n+1}^{\ast\ast}$ as functions on $(0,1)$ given by
$$\alpha_1^\ast(x) = \begin{cases}
\frac{1}{x}&\text{if }x\in(0,\frac12),\\
\frac{1+x}{1-x}&\text{if }x\in(\frac12,1),
\end{cases}
\qquad 
\alpha_n^\ast(x) = \alpha_1^\ast(T^{n-1}x),
$$
and
$$\alpha_2^{\ast\ast}(x) = \begin{cases}
\frac{1}{Tx}&\text{if }x\in(0,\frac12),\\
\frac{1+Tx}{1-Tx}&\text{if }x\in(\frac12,1).
\end{cases}
\qquad 
\alpha_{n+1}^{\ast\ast}(x) = \alpha_2^{\ast\ast}(T^{n-1}x),
$$
By Birkhoff's ergodic theorem, for almost every $x\in(0,1)$, 
\begin{equation}\label{eq:Birk}
\lim_{N\to \infty}\frac{\sum_{n=1}^N\log L_{\alpha,n}}{N}=\int_0^1 \log \frac{\alpha_1^\ast-1}{\alpha_2^\ast-1} d\mu + \int_0^1\log\frac{\alpha_2^\ast-1}{\alpha_2^{\ast\ast}-1}(\alpha_2^{\ast\ast})^2 d\mu
\end{equation}
if the integrals are finite.

Since the density of $\mu$ with respect to Lebesgue measure is bounded away from zero and infinity, $\mu$-integrability is equivalent to Lebesgue integrability.
The integral
$$\int_0^1\log (\alpha_1^\ast-1) dx = \int_0^{1/2}\log\frac{1-x}{x}dx+\int_{1/2}^1\log\frac{2x}{1-x}dx = \frac{5}{2}\log 2$$
is finite.
Thus, $\log (\alpha_1^\ast-1)$ is integrable with respect to $\mu$.
Since $\alpha_2^\ast(x)=\alpha_1^\ast(Tx)$, the first integral of \eqref{eq:Birk} vanishes.

The integrand in the second term on the right-hand side of \eqref{eq:Birk} is the function $\log f$ since $$\frac{\alpha^\ast_{2}-1}{\alpha^{\ast\ast}_{2}-1}(\alpha^{\ast\ast}_{2})^2
=\begin{cases}
\frac{1}{(Tx)^2}&\text{if }x,\ Tx\in(0,\frac12),\vspace{0.5ex}\\
\frac{2}{(1-Tx)^2}&\text{if }x\in(0,\frac{1}{2}),\ Tx\in(\frac{1}{2},1),\vspace{0.5ex}\\
\frac{(1+Tx)^2}{(1-Tx)^2}&\text{if }x,\ Tx\in(\frac12,1),\vspace{0.5ex}\\
\frac{(1+Tx)^2}{2(Tx)^2}&\text{if }x\in(\frac12,1),\ Tx\in(0,\frac12).
\end{cases}$$
The function $f$ is bounded above by $g(Tx)$, where
$$g(x)=\frac{1}{x^2}\cdot \frac{2}{(1-x)^2}\cdot \frac{(1+x)^2}{(1-x)^2}\cdot\frac{(1+x)^2}{2x^2}.$$
The function $\log f$ is integrable, as is $\log g$.
Therefore, 
$$\lim_{N\to \infty}\frac{\sum_{n=1}^N\log L_{\alpha,n}}{N}=\int_0^1 \log fd\mu.$$
\end{proof}

Now let us prove parts (2) and (3) of the proof of Theorem~\ref{th:exctime}.
We use the factorizations of $L_{\alpha,n}$ and $L_{\beta,n}$ displayed in the proof of Lemma~\ref{le:f}, specifically the factors containing $(\alpha_{n+1}^{**})^2$ and $(\beta_{n+1}^{**})^{-2}$. In the product of the $L_{\alpha,n}$, the adjacent factors cancel whenever $\alpha_{n+1}^{**}=\alpha_{n+1}^*$. Thus only the indices at which $s_n$ changes remain:
for $k_0=1$, $k_1:=\min\{k:s_{k-1}\not=s_k\}$ and $k_i: = \min\{k>k_{i-1}:s_{k-1}\not=s_k\},$
\begin{align*}
    \prod_{n=k_i}^{k_{i+1}-1}L_{\alpha,n} = \frac{\alpha^\ast_{k_i}-1}{\alpha^{\ast\ast}_{k_{i+1}}-1}(\alpha^{\ast\ast}_{k_i+1}\alpha^{\ast\ast}_{k_i+2}
\cdots\alpha^{\ast\ast}_{k_{i+1}})^2 ;\\ %\quad\text{and}\quad
\prod_{n=k_i}^{k_{i+1}-1}L_{\beta,n} = \frac{\beta^{\ast\ast}_{k_{i+1}}+1}{\beta^{\ast}_{k_{i}}+1}(\beta^{\ast\ast}_{k_i+1}\beta^{\ast\ast}_{k_i+2}
\cdots\beta^{\ast\ast}_{k_{i+1}})^{-2}.
\end{align*}

\begin{lem}\label{le:1}
For $\mu$-a.e. $x\in(0,1)$, we have 
$$\lim_{N\to \infty}\frac{\log\prod_{n=1}^N L_{\alpha,n}-2\log R_N}{N}= 0,$$
where $R_N$ is the sequence in \eqref{eq:R}.
\end{lem}
\begin{proof}
For $k_i\le j< k_{i+1}$, iterating \eqref{eq:M^-1x} from $k_i$ through $j$ and multiplying the resulting one-step identities cancels every intermediate numerator and denominator. Since the type $s_n$ is constant for $k_i\le n<k_{i+1}$, this telescoping gives
$$(\alpha^{\ast\ast}_{k_i+1}\alpha^{\ast\ast}_{k_i+2}
\cdots\alpha^{\ast\ast}_{j+1})^{-2} = 
\begin{cases}
(x_{k_i+1}x_{k_i+2}\cdots x_{j+1})^2 = \dfrac{(Q_{j}x-P_{j})^2}{(Q_{k_i-1}x-P_{k_i-1})^2},&\text{if }s_{k_i}=e, \vspace{1ex}\\
(\iota(x_{k_i+1}) \iota(x_{k_i+2})\cdots \iota(x_{j+1}))^2 = \dfrac{(\widehat{Q}_{j} \iota(x)-\widehat{P}_{j})^2}{(\widehat{Q}_{k_i-1}\iota(x)-\widehat{P}_{k_i-1})^2},&\text{if }s_{k_i}=o.
\end{cases}
$$
By Lemma~\ref{le:p/q.star}, if $s_{n-1}=e$, then we have
\begin{equation}\begin{split}\label{eq:3.18}
&\widehat{Q}_{n-1} \iota(x)-\widehat{P}_{n-1}\\
&= \frac{1}{2}\left[\veps_{n-1}(Q_{n-2}+P_{n-2})+(Q_{n-1}+P_{n-1})\right]\cdot\frac{1-x}{1+x}-\frac{1}{2}\left[\veps_{n-1}(Q_{n-2}-P_{n-2})+(Q_{n-1}-P_{n-1})\right]\\
&= \frac{\veps_{n-1}(P_{n-2}-xQ_{n-2})+(P_{n-1}-xQ_{n-1})}{1+x}.
 \end{split}\end{equation}
By applying \eqref{eq:M^-1x}, we have
$$\frac{\widehat{Q}_{n-1} \iota(x)-\widehat{P}_{n-1}}{Q_{n-1} x-P_{n-1}} = \frac{\frac{1}{x_n}-1}{1+x}=\alpha_1\cdot \frac{\alpha_n-1}{\alpha_1+1}.$$
Thus, when $s_{k_i}=o$, and hence $s_{k_i-1}=e$, we have
$$\frac{\tilde{\iota}(\alpha_{k_i})-1}{\alpha_{k_i}-1}\cdot\frac{(\widehat{Q}_{k_i-1} \iota(x)-\widehat{P}_{k_i-1})^2}{(Q_{k_i-1} x-P_{k_i-1})^2} = \frac{2}{(\alpha_{k_i}-1)^2}\cdot\frac{\alpha_1^2(\alpha_{k_i}-1)^2}{(\alpha_1+1)^2} = \frac{2\alpha_1^2}{(\alpha_1+1)^2}.$$
Similarly, when $s_{k_i}=e$, we have
$$\frac{\alpha_{k_i}-1}{\tilde{\iota}(\alpha_{k_i})-1}\cdot\frac{(Q_{k_i-1}x-P_{k_i-1})^2}{(\widehat{Q}_{k_i-1} \iota(x)-\widehat{P}_{k_i-1})^2} = \frac{2{(\tilde{\iota}(\alpha_1))}^2}{(\tilde{\iota}(\alpha_1)+1)^2}=\frac{(\alpha_1+1)^2}{2\alpha_1^2}.$$

In the product $\prod_{n=1}^N L_{\alpha,n}$, after excluding the initial
factor $(\alpha_1^*-1)$, the other initial factor given by
$$
\begin{cases}
x^2, & \text{if } s_1=e,\\
\iota(x)^2, & \text{if } s_1=o,
\end{cases}
$$
the common terminal factor
$(\alpha_{N+1}^{**}-1)^{-1}$, and the other terminal factor given by
\begin{equation}\label{eq:lem3.18}
\begin{cases}
(Q_Nx-P_N)^{-2}, & \text{if } s_N=e,\\
(\widehat Q_N\iota(x)-\widehat P_N)^{-2}, & \text{if } s_N=o,
\end{cases}
\end{equation}
the remaining factors alternate between
$\frac{2\alpha_1^2}{(\alpha_1+1)^2}$ and its reciprocal.
Since \(\frac{\alpha_1}{\alpha_1+1}\in\big[\tfrac12,1\big]\), the remaining product is bounded above and below by positive constants independent of \(N\).

All the initial factors are positive and finite for each fixed $x$, and are independent of $N$.
We shall show that $\lim_{N\to\infty}\frac{\log(\alpha_{N+1}^{**}-1)^{-1}}{N}=0$.
Since
$$
\alpha_{N+1}^{**}-1=
\begin{cases}
\displaystyle \frac{1-T^Nx}{T^Nx}, & \text{if } s_N=e,\\[2mm]
\displaystyle \frac{2T^Nx}{1-T^Nx}, & \text{if } s_N=o,
\end{cases}
$$
we have $\left|\log\bigl(\alpha_{N+1}^{**}-1\bigr)\right|
\leq \psi(T^Nx),$
where $\psi(y):=\log 2+|\log y|+|\log(1-y)|.$
Since $\psi\in L^1(\mu)$, a standard consequence of the pointwise ergodic theorem gives
$$
\frac{1}{N}
\left|\log\bigl(\alpha_{N+1}^{**}-1\bigr)\right|
\leq \frac{\psi(T^Nx)}{N}
\to 0 \text{ for $\mu$-a.e. $x$},
$$
see, for example,
\cite[Chapter~1, Section~1.2, p.~10]{Kre85}.

Finally, it is enough to show that the factor given by \eqref{eq:lem3.18} is comparable to $R_N^2$.
Let $s_{k_i}=e$. For $k_i\le N\le k_{i+1}-2$, we have $s_N=s_{N+1}=e$.
From \eqref{eq:Mx},
$$|Q_Nx-P_N| = \left|Q_N\frac{\varepsilon_{N+1}P_{N}x_{N+2}+P_{N+1}}{\varepsilon_{N+1}Q_{N}x_{N+2}+Q_{N+1}}-P_{N}\right| = \frac{1}{Q_{N+1}\left|\frac{\varepsilon_{N+1}Q_N}{Q_{N+1}}x_{N+2}+1\right|}.$$
For $N=k_{i+1}-1$, we have $s_N=e$ and $s_{N+1}=o$. 
Thus,
$$|Q_Nx-P_N| = \left|Q_N\frac{\varepsilon_N P_{N-1}x_{N+1}+P_{N}}{\varepsilon_NQ_{N-1}x_{N+1}+Q_{N}}-P_N\right| = \frac{1}{Q_N\left|\frac{\varepsilon_N Q_{N-1}}{Q_N}+\frac{1}{x_{N+1}}\right|}.$$
From \eqref{eq:eQ/Q} and the fact that $x_j\in(\frac12,1)$ is equivalent to $s_j=o$, it follows that \(R_N\,|Q_N x-P_N|\) is bounded above and below by positive constants independently of \(N\), whenever \(s_N=e\).

The case $s_{N}=o$ is treated analogously.
In this case, the quantities
$$
\begin{cases}\widehat{Q}_{N+1}|\widehat{Q}_N \iota(x)-\widehat{P}_N|&\text{ when }N\not= k_{i+1}-1,\\  \widehat{Q}_{N} |\widehat{Q}_N \iota(x)-\widehat{P}_N|&\text{ when }N= k_{i+1}-1,
\end{cases}
$$
are bounded above and below by positive constants,
and the conclusion follows from \(Q_N\asymp \widehat Q_N\).
\end{proof}

\begin{lem}\label{le:2}
We have $\prod_{n=1}^N L_{\beta,n} \asymp Q_N^2$ for $x\in(0,1)$.
\end{lem}
\begin{proof}
Let $\mathcal{E}_{j}=\prod_{n=k_i}^{j}\varepsilon_{n}$.
Likewise, for $k_i\le j< k_{i+1}$, iterating \eqref{eq:Mtbeta} from $k_i$ through $j$ and multiplying telescopes the intermediate factors. Hence, for $k_i\le j<k_{i+1}$,
$$\beta^{\ast\ast}_{k_i+1}\beta^{\ast\ast}_{k_i+2}\cdots \beta^{\ast\ast}_{j+1} = 
\begin{cases}
\beta_{k_i+1}\beta_{k_i+2}\cdots\beta_{j+1}=\mathcal{E}_{j}\cdot\dfrac{P_{k_i-1}\beta_{1}+Q_{k_i-1}}{P_{j}\beta_{1}+Q_{j}}, &\text{if }s_{k_i}=e,\vspace{1ex} \\ 
\iota(\beta_{k_i+1}) \iota(\beta_{k_i+2})\cdots \iota(\beta_{j+1}) = \mathcal{E}_{j}\cdot \dfrac{\widehat{P}_{k_i-1}\iota(\beta_{1})+\widehat{Q}_{k_i-1}}{\widehat{P}_{j}\iota(\beta_{1})+\widehat{Q}_{j}},&\text{if }s_{k_i}=o.
\end{cases}$$

In a similar way as in \eqref{eq:3.18}, from Lemma~\ref{le:p/q.star}, if $s_{k_i-1}=e$, then 
$$\widehat{P}_{k_i-1}\iota(\beta_{1})+\widehat{Q}_{k_i-1} = \frac{\varepsilon_{k_i-1}(Q_{k_i-2}+P_{k_i-2}\beta_{1})+(Q_{k_i-1}+P_{k_i-1}\beta_{1})}{1+\beta_{1}}.$$
Combining with \eqref{eq:Mtbeta}, we obtain
$$\frac{P_{k_i-1}\beta_{1}+Q_{k_i-1}}{\widehat{P}_{k_i-1}\iota(\beta_{1})+\widehat{Q}_{k_i-1}} = \frac{1+\beta_1}{\beta_{k_i}+1}.$$
Therefore,
$$\frac{\beta_{k_i}+1}{\iota(\beta_{k_i})+1}\cdot\frac{(P_{k_i-1}\beta_{1}+Q_{k_i-1})^2}{(\widehat{P}_{k_i-1}\iota(\beta_{1})+\widehat{Q}_{k_i-1})^2}= \frac{(1+\beta_{k_i})^2}{2}\cdot\frac{(1+\beta_{1})^2}{(\beta_{k_i}+1)^2}=\frac{(1+\beta_{1})^2}{2}.$$
In a similar way, if $s_{k_i-1}=o$, then we have
$$\frac{\iota(\beta_{k_i})+1}{\beta_{k_i}+1}\cdot\frac{(\widehat{P}_{k_i-1}\iota(\beta_{1})+\widehat{Q}_{k_i-1})^2}{(P_{k_i-1}\beta_{1}+Q_{k_i-1})^2}=\frac{(1+\iota(\beta_{1}))^2}{2}=\frac{2}{(1+\beta_{1})^2}.$$
Therefore,
$$\prod_{n=1}^N L_{\beta,n}\asymp
\begin{cases} (P_N\beta_1+Q_N)^2 = Q_N^2\left(\frac{P_N}{Q_N}\beta_1+1\right)^2 &\text{if }s_N=e,\\
(\widehat{P}_N\iota(\beta_1)+\widehat{Q}_N)^2 =  \widehat{Q}_N^2\left(\frac{\widehat{P}_N}{\widehat{Q}_N}\iota(\beta_1)+1\right)^2&\text{if }s_N=o.
\end{cases}$$
Since $P_j/Q_j,\widehat{P}_j/\widehat{Q}_j\in(0,1)$, $Q_j\asymp \widehat{Q}_j$ and $\beta_{j},\iota(\beta_{j})\in[\sqrt{3}-2,\sqrt{3}]$,
the conclusion follows.
\end{proof}

\end{document}